\documentclass[11pt]{report}

\usepackage[numbers]{natbib}
\usepackage{graphics}
\usepackage[toc,page]{appendix}
\usepackage{amscd,ragged2e}
\usepackage{amsmath}
\usepackage{graphicx}
\usepackage{epsfig}
\usepackage{graphicx}

\usepackage{dsfont}
\usepackage{amssymb}
\usepackage{amsmath}
\usepackage{amsthm}

\newcommand{\beq}{\begin{equation}}\newcommand{\enq}{\end{equation}}\newcommand{\G}{\mathbb{G}}\newcommand{\F}{\mathcal{F}}
\newcommand{\C}{\mathbb{C}}\newcommand{\R}{\mathbb{R}}\newcommand{\dsp}{\displaystyle}
\newcommand{\N}{\mathbb{N}}\newcommand{\Z}{\mathbb{Z}}\newcommand{\eps}{\varepsilon}

\newcommand{\Raw}{\Rightarrow}
\newcommand{\raw}{\rightarrow}\newcommand{\new}{\newline}\newcommand{\bld}{\mathbf}

\newcommand{\ind}{\mathds{1}}
\newcommand{\cone}{\begin{center}}\newcommand{\ctwo}{\end{center}}\newcommand{\Ahat}{\widehat{A}}

\newcommand{\bs}{\backslash}\newcommand{\beqa}{\begin{eqnarray*}}\newcommand{\enqa}{\end{eqnarray*}}
\newcommand{\half}{\frac{1}{2}}

\newcommand{\la}{\langle}\newcommand{\ra}{\rangle}
\newcommand{\KP}{\mathbb{KP}}
 \newcommand{\wxn}{\addtocontents{toc}{\protect\setcounter{tocdepth}{0}}}
\newcommand{\wxo}{\addtocontents{toc}{\protect\setcounter{tocdepth}{1}}}
\newtheorem{theorem}{Theorem}[section]
\newtheorem{lemma}[theorem]{Lemma}
\newtheorem{proposition}[theorem]{Proposition}

\newenvironment{definition}[1][Definition]{\begin{trivlist}
\item[\hskip \labelsep {\bfseries #1}]}{\end{trivlist}}
\newenvironment{example}[1][Example]{\begin{trivlist}
\item[\hskip \labelsep {\bfseries #1}]}{\end{trivlist}}
\newenvironment{remark}[1][Remark]{\begin{trivlist}
\item[\hskip \labelsep {\bfseries #1}]}{\end{trivlist}}

\begin{document}
\begin{titlepage}
\cone
\huge{Random Walks} \\ \huge{on
       Finite Quantum Groups}

       \bigskip

      \LARGE{Diaconis--Shahshahani Theory} \\
          \LARGE{for Quantum Groups}

          \bigskip

          \large{J.P. McCarthy}
          \author{J.P. McCarthy}

          \bigskip

\vspace{1.5in}

      \normalsize{ A thesis submitted to the National University of Ireland, Cork for the
degree of Doctor of Philosophy

\bigskip

Supervisor: Dr Stephen Wills

\bigskip

Head of Department: Dr Martin Kilian

\vspace{1.5in}

\bigskip

Department of Mathematics

\bigskip

College of Science, Engineering and Food Science

\bigskip

National University of Ireland, Cork

\bigskip

January 2017}

\ctwo
\end{titlepage}

\tableofcontents
\begin{abstract}
Of central interest in the study of random walks on finite groups are \emph{ergodic} random walks.  Ergodic random walks converge to random in the sense that as the number of transitions grows to infinity, the state-distribution converges to the uniform distribution on $G$. The study of random walks on finite groups is generalised to the study of random walks on \emph{quantum groups}. Quantum groups are neither groups nor sets and rather what are studied are finite dimensional algebras that have the same properties as the algebra of functions on an actual group --- except for commutativity.

\bigskip

The concept of a random walk converging to random --- and a metric for measuring the distance to random after $k$ transitions --- is generalised from the classical case to the case of random walks on quantum groups.

\bigskip

A central tool in the study of ergodic random walks on finite groups is the Upper Bound Lemma of Diaconis and Shahshahani. The Upper Bound Lemma uses the representation theory of the group to generate upper bounds for the distance to random and thus can be used to determine convergence rates for ergodic walks. The representation theory of quantum groups is very well understood and is remarkably similar to the representation theory of classical groups. This allows for a generalisation of the Upper Bound Lemma to an Upper Bound Lemma for \emph{quantum} groups.

\bigskip

The Quantum Diaconis--Shahshahani Upper Bound Lemma is used to study the convergence of ergodic random walks on classical groups $\Z_n$, $\Z_2^n$, the dual group $\widehat{S_n}$ as well as the `truly' quantum groups of Kac and Paljutkin and Sekine.

\bigskip

Note that for all of these generalisations, restricting to commutative subalgebras gives the same definitions and results as the classical theory.

\end{abstract}

\section*{Acknowledgements}
In this section of my MSc thesis, I signed off my thanks to my advisor Dr Stephen Wills by saying that I was very much looking forward to working with him on my PhD work. These proved
to be prophetic words as Steve once again provided assistance and good council whenever I required it. I would like to express particular gratitude for his great patience over these last six years. Combining study with a heavy lecturing load meant there were more than a few months, nay, quarters where very little work got done. He also provided a calming presence: although I never openly expressed doubts about the viability of this project, that does not mean that there were not any.

\bigskip

The second round of thanks must go to the two heads of the Department of  Mathematics that I have worked under at the Cork Institute of Technology: Dr David Flannery and Dr \'{A}ine N\'{i} Sh\'{e}. Their support and advice --- first from Dave and then \'{A}ine --- has throughout been unwavering and, in the parlance of our times, `on point'. Credit also to Dr Noel Barry --- Head of Academic Affairs --- who allowed me to enter the Staff Doctorate Scheme in CIT and thus receive funding to support my PhD studies.

\bigskip

It is with a definite sadness that my dear grandmother Roses died (she hated the phrase `passed away') during the course of my PhD studies. From 2001 until 2012, she was a champion of my academic life here in Cork and I will shed more than one tear when, hopefully, I go out to her grave and triumphantly show her the oul parchment. I would like to apologise to my mother --- her famous bad luck meant she would only ask how the PhD was going when the PhD was not going at all. Thus, for showing her interest \& concern, she only ever received sharp words by way of a thank you. I would like to thank my father for nodding in agreement with me whenever I claimed ``look, my uncle Patrick spent ten years at it --- it'll be grand''.

\bigskip

Very soon after Roses left my life, Rebecca entered. Luckily I have the rest of my life to show my gratitude to her (although it must be remarked upon that my uncle Patrick did not finish his PhD until he was married --- hopefully Rebecca will be spared this spectre).


\justify
\chapter{Introduction}
The innocuous sounding question --- how many shuffles are required to mix up a deck of cards? --- leads to considering `shuffles' $\sigma_i\in S_{52}$ chosen according to a fixed probability distribution, and asking how large should $k$ be so that the distribution of the random variable
$$\sigma_k\cdots \sigma_2\cdot\sigma_1$$
is approximately uniform on $S_{52}$. The culture of generalisation in mathematics leads us to consider the following problem. Given a finite group, $G$, and elements $s_i\in G$ chosen according to a fixed probability distribution, how large should $k$ be so that the distribution of the random variable
$$s_k\cdots s_2\cdot s_1$$
is approximately uniform on $G$? Such problems arise in the theory of random walks on finite groups and were the subject of the author's MSc thesis \citep{MSc}.

\bigskip

It became apparent during the development of quantum mechanics that classical, Kolmogorovian probability was unable to describe quantum mechanical phenomena such as, for example, Heisenberg's Uncertainty Principle.  Just like the fact that classical probability had been studied for years before Kolmogorov lay down the measure-theoretic, axiomatic foundation of the subject in the early 1930s (ironically not very long after the work of Hilbert, Dirac, von Neumann and others on quantum mechanics), quantum probability had been studied  --- primarily in the field of quantum mechanics --- for the bones of half a century before maturing in the 1970s and 1980s.

 \newpage

Taking a line through the uncertainty principle, observables $a$ and $b$ (measurable quantities) need not commute: the observable $ab$ need not be the same as the observable $ba$, and therefore, rather than a real-valued function on a state space, observables might behave more like matrices. Considering further postulates about the nature of quantum mechanics (justified by the experimental verification of their consequences \citep{Brans}), Dirac and von Neumann were led to the following axioms:
\begin{itemize}
\item the observables of a quantum mechanical system are defined to be the self-adjoint elements of a $\mathrm{C}^*$-algebra.
\item the states of a quantum mechanical system are defined to be the states of the $\mathrm{C}^*$-algebra.
\item the value $\rho(a)$ of a state $\rho$ on an element $a$ is the expectation value of the observable $a$ if the quantum system is in the state $\rho$.
\end{itemize}
Moving away from quantum mechanics, the basic definition in quantum probability is that of a \emph{quantum probability space}, sometimes referred to as a noncommutative probability space \citep{space}.

\bigskip

\begin{definition}
A \emph{quantum probability space} is a pair $(A,\rho)$, where $A$ is a ${}^*$-algebra and $\rho$ is a state.
\end{definition}

\bigskip

This definition is a generalization of the definition of a probability space in Kolmogorovian probability theory, in the sense that every (classical) probability space, $\Omega$, gives rise to a quantum probability space if $A$ is chosen as $\mathcal{L}^\infty(\Omega)$, the ${}^*$-algebra of bounded complex-valued measurable functions on it. Indeed every  `quantisation' of classical probability should, ideally, agree with the classical definition if restricted to a commutative subalgebra.

\bigskip

Considered as a research programme, quantum probability is concerned with generalising, where possible, objects in the study of classical probability to quantised objects in the study of quantum probability theory. It is under this programme that this work lies: the study of random walks on finite groups uses classical probability theory --- a study of random walks on \emph{quantum} groups \emph{should} be the corresponding area of study in quantum probability.

\bigskip

Therefore, this work is concerned with a generalisation of a generalisation of card shuffling: generalising, where possible, the ideas and results presented in the MSc thesis to the case of quantum groups. The problem is that while the central object of card shuffling --- the set of shuffles --- is generalised to that of a set of elements of a group, the generalisation to \emph{quantum} groups moves away from a `set of points' interpretation. For those new to the area (such as the author at the beginning of this study), this can cause serious problems --- particularly because this generalisation means a dearth of intuition. Going back to quantum mechanics, the fact that one of the most successful physics theories of our time says frankly unimaginable things about the nature of reality --- space is not as we comprehend and perhaps even incomprehensible --- leads to famous quotes such as those of Niels Bohr:
\begin{quote}
  \emph{If quantum mechanics hasn't profoundly shocked you, you haven't understood it yet.}
\end{quote}
However, in this study of random walks on finite quantum groups at least, the quantum theory generalises so nicely from the classical setup that it can be fruitful to refer to quantum groups and associated  \emph{virtual} objects as if they really exist. This has become more and more common in the quantum group community and is a helpful development in the author's opinion: this approach is utilised as often as possible in this work. As is commented upon later, at the very least this approach gives a most pleasing notation for quantised objects (in fact some papers simply denote a quantum group by $G$ not paying much credence to the fact that it is not actually a `set of points' group). For examples of this approach see recent papers on quantum groups such as by Banica and M\'{e}sz\'{a}ros \citep{not1}, Franz, Kula and Skalski \citep{not2} and Skalski and So{\l}tan \citep{not3}

\bigskip

Starting in the 1980s with the work of Drinfeld, Jimbo and (later) Woronowicz, there are many motivations for and approaches to quantum groups (although in finite dimensions, the majority of approaches are equivalent).  As this study concerns random walks on \emph{finite} quantum groups, this recent history of the motivations for and approaches to quantum groups is largely irrelevant. Briefly, while quantum groups were first spoken about in the 1980s, the objects studied in this thesis can be traced back to work by Heinz Hopf in the 1940s and Kac in the 1960s \citep{Andr}. Please see the introduction by Timmermann \citep{Timm} to learn more about the motivations for and approaches to quantum groups.

\bigskip

Random walks on \emph{finite} quantum groups were first studied by Franz and Gohm \citep{franzgohm}. The random walks of interest in the classical case, largely, are those which converge in distribution to the uniform or \emph{random} distribution, $\pi$. The question that is asked about these classical random walks are as per the shuffling question. Asking this question in a more precise way involves putting a metric on the set of probabilities on a group, and asking, where $\Psi_k$ is the distribution of the product of $k$ group elements (sampled by a fixed probability distribution): for a given $\eps>0$, how large must $k$ be to ensure that $d(\Psi_k,\pi)<\eps$? As far as the author knows, this question has not been asked for random walks, in the sense of Franz and Gohm, on finite quantum groups. The quantisation of a `random walk on a group converging to random' is a random walk on a finite quantum group converging in distribution to the Haar state (which will eventually be denoted by $\pi$ also). This work, in Chapter 4, gives an appropriate metric to measure the distance between the quantised distribution, $\Psi_k$, of the random walk after $k$ transitions, and the random distribution $\pi$.

\bigskip

Returning to quantum mechanics briefly, a classical random walk on a finite group $G$ can be viewed as a quantum mechanical system evolving by transitioning from state to state at discrete times. Assume furthermore that the random walk sits in a lidded black box. The algebra of complex-valued functions on $G$, $F(G)$, is a commutative $\mathrm{C}^*$-algebra and can be concretely realised as the set of diagonal operators on $\C^{|G|}$. Thus, the observables are $|G|\times |G|$ diagonal matrices with real entries i.e. the real-valued functions on $G$. Note that the measurement of an observable $f$ is one of the eigenvalues of the associated linear operator. The eigenvalues of a diagonal operator  are the elements along the diagonal --- in other words the function values $\{f(s):s\in G\}$. The state after $k$ transitions is given by a state $\Psi_k$ on  $F(G)$ i.e. integration against a probability distribution, $\mu_k$. Therefore the expectation of an observable $f$ after $k$ transitions is given by
$$\Psi_k(f)=\int_G f(t)\,d\mu_k(t).$$
Taking the Copenhagen interpretation of quantum mechanics \citep{sten}, after $k$ transitions the wavefunction/state $\Psi_k$ describes the random walk completely and that is all that can be said. However if an observable $f$ is to be measured, the system must be interfered with: the lid must be lifted off. If the result of the measurement of $f$ yields $f(s)$ then the wavefunction has collapsed into the state $s$ (or rather $\delta^s$). A model for a random walk is, of course, a cat\footnote{bearing a collar with the note ``If found please call 01-6140100 and ask for E.S.''} inside a large black room containing a structure modelling the Cayley graph of the group, moving from node to node in a seemingly random manner. Given some observable $f$ --- perhaps the height of the node upon which the cat sits --- before opening the door, all that can said about the result of measuring $f$ after $k$ transitions is the expectation, $\Psi_k(f)$. However, upon opening the door, the observer could see at this $k$th transition of the walk that the cat was on the node labelled $s$ and thus everything was known about the result of the measurement: it would certainly yield the eigenvalue $f(s)$.

\bigskip

Now thinking about a random walk on a finite, classical group converging to random, what can be imagined is that no matter what observable $f$  is considered, as more and more transitions are made, then --- with the lid on ---  less and less is known about the state of the random walk in the  black box. No good prediction  can be made about where the random walk is after $k$ transitions and  the random walk is --- approximately --- uniformly distributed on $G$. If the random walk is uniformly distributed,  the expectation of $f$ is nothing but the mean-average of $f$.

\bigskip

Therefore, although the algebra of functions on a  finite quantum group is defined in this work to be a finite but not-necessarily-commutative $\mathrm{C}^*$-algebra ---  therefore without a `set of points' interpretation to parameterise the states --- the self-adjoint elements of the $\mathrm{C}^*$-algebra can still be considered observables whose eigenvalues are the result of measuring the state of the random walk. If converging to the Haar state is to be considered in the same way as the classical case --- that the Haar state $h$ is integration against the uniform measure and so $h(a)$ is interpreted as the average of $a$ --- then a random walk on a quantum group converging to random shares the property of random walks on classical groups --- that as more and more transitions are made, the expectation of any observable is nothing but the mean-average.

\bigskip

This thesis shows that for given families of random walks on $\Z_n$, $\Z_2^n$, $\widehat{S_n}$ and $\mathbb{KP}_n$, respectively, $\mathcal{O}(n^2)$, $\mathcal{O}(n\ln n)$, $\mathcal{O}(n^n)$ and $\mathcal{O}(n^2)$ transitions are sufficient for convergence to random. The first two random walks have been studied before \citep{cecc}, but as far as the author knows, for a truly quantum group, or rather family of quantum groups, such as $\mathbb{KP}_n$, this is the first time that explicit convergence rates have been obtained.

\bigskip

One of the most exciting and potentially lucrative aspects of quantum probability, or rather more specifically quantum group theory, is that theorems about finite groups may in fact be true for quantum groups also. For example --- and a lot of this thesis hangs upon this --- the finite Peter--Weyl Theorem \ref{PeterW} concerning the matrix elements of representations of classical groups is exactly the same as the finite Peter--Weyl Theorem \ref{QPW} for \emph{quantum} finite groups in the sense that replacing in the classical statement `finite group, $G$' with `finite quantum group, $\G$' yields the quantum statement. What this really means is that the classical finite Peter--Weyl Theorem is actually just a special case of the quantum finite Peter--Weyl Theorem (itself a special case of \emph{the} Peter--Weyl Theorem (for compact quantum groups)).

\bigskip

This is rather comforting on the conceptional level --- these quantum objects behave so much like their `set of points' classical counterparts --- but it is on the pragmatic level of proving results about these quantum objects that this principle really comes to the fore. On the one hand, some theorems concerning the theory of finite groups are just  corollaries to results about quantum groups. On this other, pragmatic, hand, there is a transfer principle: any proof of a classical group theorem, written without regard to any of the points in the `set of points', may be directly translatable into a proof of the corresponding quantum group theorem.

\bigskip

In this work, the hero of this transfer principle is the Haar state, $h$, which frequently allows `sum over points' arguments and statements about elements of an $f:G\raw \C$ in the algebra of functions on a finite group, $G$, to be transferred via:
$$\underbrace{\frac{1}{|G|}\sum_{t\in G}f(t)}_{\text{classical: references points }t\in G}=\underbrace{h(f)}_{\text{quantum: no reference to points}}.$$
Representation Theory and `sum over points' arguments, therefore are transferrable and it is precisely these ideas that play a central role in the representation- theoretic approach of Diaconis to analysing the rate of convergence of random walks on finite groups \citep{PD}. Once the quantised versions of the various objects and maps used by Diaconis are established --- and these were non-trivial tasks --- it was largely straightforward to derive and prove the transferred/quantised central tool of Diaconis' work --- the Diaconis--Shahshahani Upper Bound Lemma. The fact that the quantum Upper Bound Lemma is as similar to the classical Upper Bound Lemma as the quantum Peter--Weyl Theorem is to the classical Peter--Weyl  is \emph{the} triumph of this work.

\bigskip

The restriction to \emph{finite} quantum groups is for two reasons. First of all the classical work that this is building upon is the Diaconis--Shahshahani Theory approach to random walks on \emph{finite}  groups. Secondly, the approach to quantising classical notions in this work --- extolled in Section \ref{category} --- requires the isomorphism $(A\otimes B)^*=A^*\otimes B^*$ for spaces $A$ and $B$ and this holds only when $A$ and $B$ are finite dimensional. However, the classical Diaconis--Shahshahani Theory also applies to compact groups not just finite groups. Section \ref{compact} points the way towards extending this work to the compact quantum case where the algebras are no longer necessarily finite dimensional.

\bigskip

This does indeed make the work modest: however on the other hand the application of the Upper Bound Lemma is made more difficult by the fact that for truly quantum groups there must be at least one representation of dimension greater than one. A quantum group with one dimensional representations only is isomorphic to the group ring of a finite (classical) group.

\bigskip

It would be remiss not to declare a deficiency of this work, namely that the upper bounds generated are probably not very sharp --- and if sharp, they do not come with a complementary sharp lower bound. One could argue that the main aim of this study was to prove a Diaconis--Shahshahani Upper Bound Lemma for quantum groups, and while this aim was successful, a more honest appraisal of the work might paraphrase an idiom of calculus and say, in this context at least, that finding upper bounds is mechanics while procuring lower bounds is art --- and the author has failed to show a creative side. In particular, failing to present a random walk on a truly quantum group exhibiting the cut-off phenomenon, when this was a key emphasis of the MSc thesis, is a definite black mark. The great hope would be that sharpening these bounds, and, critically, coming up with effective lower bounds would be the subject of future, successful, work. This is discussed further in Section \ref{LBCOP}.

\bigskip

The primary references for this work are the author's MSc thesis on random walks on finite groups \citep{MSc} (available on the arXiv), the paper of Franz and Gohm which introduces random walks on finite quantum groups \citep{franzgohm} and the comprehensive book on quantum groups by Timmermann \citep{Timm}.
 \newpage
\section{Summary}
The following sections of this chapter are concerned primarily with discussion of the \emph{Gelfand Philosophy}. This philosophy leads from Gelfand's Theorem \ref{Gelfand}  which states that \emph{commutative} unital $\mathrm{C}^*$-algebras are  algebras of (continuous) functions on compact, Hausdorff spaces. The philosophy is an invitation to think of noncommutative $\mathrm{C}^*$-algebras as algebras of functions on  \emph{quantum spaces}. These quantum spaces do not actually exist --- and are referred to as virtual objects --- yet many questions that can be posed and resolved in the commutative case may also be posed and hopefully resolved in the noncommutative case. It can sometimes be non-trivial to translate classical definitions into the quantised world but in this chapter it is seen that there is a functorial quantisation that often motivates the --- correct and well-established in the literature --- quantised definitions.

\bigskip

Chapter 2 introduces the general theory of finite quantum groups (as defined by the author); and includes a study of the Haar state. Some examples of finite quantum groups are presented; namely classical groups $G$, dual groups of classical groups $\widehat{G}$ (which are virtual when $G$ is non-abelian), the Kac--Paljutkin quantum group $\mathbb{KP}$ as well as the one parameter family of quantum groups of Sekine, $\mathbb{KP}_n$.

\bigskip

Chapter 3 presents the quantisation of discrete-time Markov chains as well as, far more importantly, the quantisation of random walks on finite groups.

\bigskip

In Chapter 4 a distinguished metric, namely the \emph{total variation distance}, is identified as the conventional measure of closeness to random in this study. As far as the author is aware, not only is this the correct quantisation/generalisation of the classical total variation distance --- in that it shares three key features of the classical metric --- it has not been studied previously.

\bigskip

Chapter 5 contains the main result or rather tool of this thesis ---  the Quantum Diaconis--Shahshahani Upper Bound Lemma. This is an extension of the classical result and so returns the same estimates when applied to classical groups viewed as quantum groups. The Upper Bound Lemma is also applied to a family of random walks on the cocommutative quantum group $\widehat{S_n}$, some random walks on the `truly quantum' group of Kac and Paljutkin, as well as a family of random walks on the one-parameter Sekine quantum groups.

\bigskip

Chapter 6 contains some possible questions/avenues for further study such as `what are necessary and sufficient conditions for a random walk to converge to random?', `does the classical spectral analytic approach to Markov chains carry over?' and examples of random walks that deserve analysis.

\bigskip

Most of the original work is concentrated in Chapters 4 and 5. However, it would be hoped that all sections contain new perspectives and points of view on previously studied objects.

\section{The Duality of Algebra and Geometry}
It is a theme of modern mathematics that geometry and algebra are `dual':
$$\text{Geometry }\leftrightarrow \text{ Algebra}$$
Arguably this began when Descartes began to answer questions about synthetic geometry using the (largely) algebraic methods of coordinate geometry. Since then this duality has been extended and refined to consider:
$$\text{Space }\leftrightarrow \text{ Algebra of Functions on the Space}$$
Here a space is a set of points with some additional structure, and the idea is that for a given space, there will be a canonical algebra of functions on the space. For example, given a compact, Hausdorff topological space $X$, the canonical algebra of functions is the continuous functions on $X$, $C(X)$. The algebra of functions on a space encodes many of the properties of that space.  In the example of a compact, Hausdorff space $X$ and its algebra of functions $C(X)$,  the Banach--Stone Theorem says that the algebra of functions determines the topology on $X$.
\subsubsection*{Examples}
\begin{enumerate}
\item \textit{Cardinality}: Let $X=\{a_1,a_2,\dots,a_n\}$ be a set and consider $F(X)$, the space of complex-valued functions on $X$. To define $f\in F(X)$  complex numbers $\lambda_i$ must be chosen:
$$f(a_i)=\lambda_i\text{ ; for }i=1,\dots,n.$$
 \newpage
Define \emph{delta functions} by:

$$\delta_x(y)=\begin{cases}
                1 & \mbox{if } x=y, \\
                0 & \mbox{otherwise.}
              \end{cases}$$
Also define \emph{indicator functions} for each $A\subset X$
$$\ind_A(x)=\begin{cases}
                1 & \mbox{if } x\in A, \\
                0 & \mbox{otherwise.}
              \end{cases}$$
Note that $\delta_x=\ind_{\{x\}}$ and
$$\ind_A=\sum_{x\in A}\delta_x.$$
 Hence every $f\in F(X)$ may be uniquely written in the form:
$$f=\sum_{i=1}^n\lambda_i\delta_{a_i}.$$
That is $\{\delta_{a_i}:1\leq i\leq n\}$ is a basis of $F(X)$ so $\text{dim } F(X)=n$.
It could be argued that the only feature of this space is that $|X|=n$. So, for a finite set $X$ such as this one, with no additional structure at all, the dimension of the algebra of functions $F(X)$ determines $X$ completely.

\item \textit{Connectedness}: Consider the interval $X=[0,1]$. In the usual topology it is connected which means $X$ cannot be represented as a union of non-empty, open disjoint subsets. Consider the continuous functions on $X$, $C(X)$. A map $p\in C(X)$ a projection if $p^2(x)=\overline{p(x)}=p(x)$ for all $x\in X$.
This means that $p$ either takes the value $0$ or the value $1$. Suppose $p$ is a non-zero projection and set
$$A=\{x\in X : p(x)=1\},$$
 so that $p=\mathds{1}_{A}$. It is clear that either $A=\emptyset$ or $X$; otherwise $p$ is not continuous as it would have jump discontinuities on the boundary of $A$. Hence the only continuous projections on the connected set $X$ are the trivial projections $0$ and $\mathds{1}_X$.

\bigskip

Consider $X=[0,1]\cup[2,3]$. This space is certainly disconnected but $\mathds{1}_{[0,1]}$ and $ \mathds{1}_{[2,3]}$ are continuous non-trivial projections. If $X\subset \R$ and if $C(X)$ contains non-trivial projections, then $X$ is disconnected.

\end{enumerate}

The examples above start with a space $X$, `induce' an algebra of functions on the space, $A(X)$ and often there is enough data in the algebra of functions to describe the space completely. Often it is equally valid to look at a commutative algebra, say $A$, and look for an `induced' space $X(A)$ in such a way that there is enough data in the space $X(A)$ to describe the algebra of functions, $A$.

\bigskip

This  can be understood on the level of observable-state duality. For example, consider a point $a_j\in X=\{a_1,\dots,a_n\}$ to be a \emph{state} and a function $f=\sum_i\lambda_i\delta_{a_i}\in F(X)$ to be an \emph{observable}. If the observable $f$ acts on the state $a_i$ then the \emph{measurement} of $f$ produces the result $\lambda_j$:

$$f(a_i)=\sum_{i=1}^n\lambda_i\delta_{a_i}(a_j)=\lambda_j.$$
However could not another party see the same measurement to be a result of the observable $a_j$ acting on the state $f$ producing the same result?

$$a_j(f)=a_j\left(\sum_{i=1}^n\lambda_i \delta_{a_i}\right)=\lambda_j.$$

There are more than a few things that need to be said to make this notion precise but it is useful to loosely introduce the concept of an observable at this point (Majid writes about this both in detail and in context in an essay \citep{Majid2}):

\begin{figure}[ht]\cone\epsfig{figure=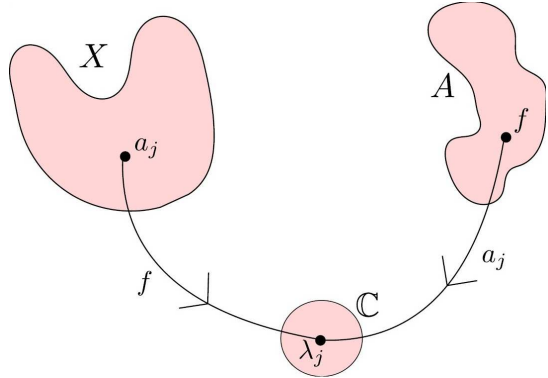,scale=0.5}\ctwo\caption{$f(a_j)=\lambda_j=a_j(f)$: \emph{observable-state duality}:  both $X$ and $A=F(X)$ are at once spaces and algebras of functions on spaces.}\end{figure}

\newpage

Note that in these examples, the  algebra of functions has a common structure:
\begin{enumerate}
\item \emph{Vector Space} --- for any complex valued functions $f$ and $g$ the functions $f+g$ and $\lambda f$  ($\lambda \in \mathbb{C}$) can be defined pointwise.
\item \emph{Normed Space} --- there are various norms that could be put on the algebra of functions. In an appropriate setting, these include the supremum norm, one norm, two norm, etc. In particular, it is convenient if the algebra of functions is a \emph{Banach space}, that is a complete normed vector space.
\item \emph{Algebra} --- a pointwise multiplication can be defined on the algebra of functions.
\item \emph{*-Algebra} --- the algebra of functions takes on an involution, namely the conjugation:
$f^*(x)=\overline{f(x)}$.
\end{enumerate}
Any algebra $A$ which has these four features (with the $\mathrm{C}^*$-equation condition on how the norm interacts with the involution: $\|a^*a\|=\|a\|^2$ for all $a\in A$), is known as a $\mathrm{C}^*$-algebra and, by and large, the canonical algebra of functions on a space will have this structure. As the complex numbers  $\mathbb{C}$ are commutative, the algebra of functions on $X$ is commutative:
$$f(x)g(x)=g(x)f(x)\,;\, \text{for all }x\in X.$$
What is the nature of these seemingly inevitable algebras? The basic features have been outlined  above but here the last two features are explored a little further. An \emph{associative algebra} is a (complex) vector space together with a bilinear map
$$m:A\times A\rightarrow A,\,(a,b)\mapsto ab,$$
such that $m(a,m(b,c))=m(m(a,b),c)$. Using the universal property, the bilinear map $m$ may be extended to a linear map
  $$\nabla:A\otimes A\raw A,\,(a\otimes b)\mapsto ab.$$
    Of course it is natural to refer to this map as the \emph{multiplication} on $A$. If $A$ admits a submultiplicative norm and a \emph{unit} --- an element $1_A\in A$ such that $a1_A=a=1_Aa$ for all $a\in A$ --- such that $\|1_A\|=1$, then $A$ is said to be a \emph{unital normed algebra}. If, further, a unital normed algebra $A$ is complete then $A$ is called a \emph{unital Banach algebra}.  An element $a\in A$ is \emph{invertible} if there is an element $a^{-1}\in A$ such that $aa^{-1}=1_A=a^{-1}a$. The set
$$G(A)=\{a\in A:a\text{ is invertible}\}$$
is an (open) group with the multiplication got from $A$. Define the \emph{spectrum} of an element $a$ to be the set
$$\sigma(a)=\{\lambda\in\C:a-\lambda1_A\not\in G(A)\}.$$
A theorem of Gelfand states that if $a$ is an element of a unital Banach algebra $A$, then the spectrum of $a$ is non-empty. As a corollary, Gelfand and Mazur proved that if $A$ is a unital algebra in which every non-zero element is invertible, then $A$ is isometrically isomorphic to $\C$ (Theorem 10.14 of \citep{Rudin91}).

%

\bigskip

Consider the non-zero linear functionals $\chi:A\rightarrow\C$ that are also homomorphisms. These maps are called \emph{characters} and the set of all such functionals is called the \emph{character space of $A$}, $\Phi(A)$. Suppose that $A$ is an abelian Banach algebra for which the space $\Phi(A)$ is non-empty. For $a\in A$, define  the \emph{evaluation map}:
$$\widehat{a}:\Phi(A)\raw\C,\,\,\chi\mapsto\chi(a).$$
 Endow $\Phi(A)$ with the weakest topology that makes all of these evaluation maps continuous: this coincides with the weak* topology. If $A$ is a unital abelian Banach algebra, then $\Phi(A)$ is a compact Hausdorff space. It can be shown  that the set $\{\chi\in\Phi(A):|\chi(a)|\geq\eps\}$ is weak*-compact. Hence $\widehat{a}\in C_0(\Phi(A))$: it is called the \emph{Gelfand transformation} of $a$.

\bigskip

An \emph{involution} on an algebra is a conjugate-linear map $a\mapsto a^*$ on $A$ such that $a^{**}=a$ and $(ab)^*=b^*a^*$. The pair $(A,*)$ is called a \emph{*-algebra}. An element is said to be \emph{self-adjoint} if $a^*=a$. A \emph{$\mathrm{C}^*$-algebra} is a Banach *-algebra such that
\beq
\|a^*a\|=\|a\|^2\,,\,\text{ for all }a\in A.
\enq
This seemingly mild condition is in fact very strong. In particular, it implies that there is at most one norm on a *-algebra making it a $\mathrm{C}^*$-algebra.
\newpage
\subsubsection*{Example}
Consider the Hilbert space  $H=\C^n$ with the usual inner product on $\C^n$ and the set of bounded operators on $H$, $A=B(H)\cong M_n(\C)$. Using the usual matrix addition and multiplication, $A$ becomes a *-algebra when equipped with the conjugate-transpose for the involution  $a^*= \overline{a^T}$. A quick calculation shows that the operator norm satisfies the $\mathrm{C}^*$-equation and so is the correct norm making $A$ into a $\mathrm{C}^*$-algebra.

\bigskip

Some features of the matrix algebra above extend to general $\mathrm{C}^*$-algebras. For example, the fact that a self-adjoint matrix has real eigenvalues is  a more general result about $\mathrm{C}^*$-algebras --- namely that the spectrum of a self-adjoint element of a $\mathrm{C}^*$-algebra $A$ is real. This analysis culminates in the beautiful theorem of Gelfand that states that every abelian $\mathrm{C}^*$-algebra is isomorphic to an algebra of functions on a space.

\bigskip

\begin{theorem}
\emph{(Gelfand)}\label{Gelfand}
If $A$ is a non-zero commutative $\mathrm{C}^*$-algebra, then the Gelfand representation
$$\varphi:A\raw C_0(\Phi(A))$$
is an isometric *-isomorphism $\bullet$
\end{theorem}

\bigskip

This is a precise realisation of the duality $f(x)=x(f)$ as discussed previously. For more see the introduction to \citep{not2}.

\section{Virtual Objects}
Consider the $\mathrm{C}^*$-subalgebra $A:=\mathcal{D}_{n^2}\subset M_{n^2}(\C)$ of $n^2\times n^2$ diagonal matrices. As it is commutative, by Gelfand's Theorem, $A$ is isomorphic to the algebra of functions on the character space, $\Phi(A)$. There are $n^2$ characters on $A$:
$$\chi_i(a)=a_{ii},$$
and so where $X:=\Phi(A)=\{\chi_1,\chi_2,\cdots,\chi_{n^2}\}$, there is an isomorphism $A\cong C_0(X)=F(X)$.
$$a=\left(\begin{array}{ccc}\lambda_1 & & \\ & \ddots & \\ & & \lambda_{n^2}\end{array}\right)\mapsto \sum_{i=1}^{n^2}\lambda_i\delta_{\chi_i}=\hat{a}$$
Note that the $\mathrm{C}^*$-norm on $A$ is the operator norm while the $\mathrm{C}^*$-norm on $F(X)$ is the supremum norm. The isomorphism $a\mapsto \hat{a}$ is isometric and so
$$\|a\|_{\text{op}}=\|a\|_{\infty}.$$
This will be seen later where the algebra of functions on a finite group  can be viewed as a diagonal subalgebra $F(G)\subset B(H)$ of bounded operators on a Hilbert space, and the operator norm of $f\in F(G)$ will be called by the supremum norm because:
$$\|f\|_{\text{op}}=\sup_{\|x\|_H\leq 1}\|f(x)\|_{H}=\max_{t\in G}|f(t)|=\|f\|_{\infty}.$$

\bigskip

Now consider the $\mathrm{C}^*$-algebra $B:=M_n(\C)\cong B(\C^n)$. Both $A$ and $B$ are $n^2$ dimensional $\mathrm{C}^*$-algebras. However because $B$ is noncommutative, only $A$ may be written as $A=F(X)$ via Gelfand's Theorem. Gelfand's Theorem says that commutative $\mathrm{C}^*$-algebras are nothing but algebras of functions on spaces however it has become fashionable to consider noncommutative $\mathrm{C}^*$-algebras as algebras of functions on \emph{noncommutative} or \emph{quantum} spaces.

\bigskip

So, for example, $B$ can be written as $B=F(\mathbb{X})$ and although $\mathbb{X}$ is a \emph{virtual object}, it can be fruitful to consider $B$ in these terms. At the very worst, this philosophy yields a nice notation. At its very best it can inspire the noncommutative geometer to unshackle macroscopic-earthly chains and employ their imagination.
\subsubsection*{Examples}
\begin{enumerate}
\item Let $(G,\star)$ be a finite group and consider the algebra of complex-valued functions on $G$, $F(G)$.

\bigskip

What kind of relations hold `up' in $F(G)$? Relations in the group --- associativity, identity and inverses --- need to be accounted for. In particular, for all $x,\,y,\,z\in G$ and $f\in F(G)$, where $e$ is the identity:
\begin{eqnarray}
f(x\star(y\star z))=f((x\star y)\star z),
\\ f(x\star e)=f(x)=f(e\star x),
\\ f(x\star x^{-1})=f(e)=f(x^{-1}\star x).
\end{eqnarray}
These relations can be translated into the language of \emph{coalgebras}. Note that with the supremum norm, pointwise multiplication and the involution $f\mapsto \overline{f}$, $F(G)$ has the structure of a $\mathrm{C}^*$-algebra. Note however that $F(G)$ is a \emph{commutative} $\mathrm{C}^*$-algebra. There exist finite-dimensional $\mathrm{C}^*$-algebras that satisfy all of the (coalgebraic) axioms of $F(G)$ \emph{except} commutativity. Through abuse of terminology, these $\mathrm{C}^*$-algebras are sometimes called quantum groups, however it is more appropriate  to refer to such a $\mathrm{C}^*$-algebras as the \emph{algebra of functions} on a quantum group. The quantum group is a virtual object and the noncommutative $\mathrm{C}^*$-algebra can be denoted by $F(\mathbb{G})$ --- the quantum group is this virtual object $\mathbb{G}$. If the algebra $F(\mathbb{G})$ has a unit $1_{F(\mathbb{G})}$ then this is denoted by $\mathds{1}_{\mathbb{G}}$ --- mirroring the fact that the unit in $F(G)$ for a classical $G$ is given by the indicator function on $G$, $\mathds{1}_G$.

\bigskip

\item Consider the unit sphere in $\R^n$:
$$\mathcal{S}^{n-1}:=\{\bld{v}\in\R^n:\|\bld{v}\|_2=1\}.$$
There are $n$ real-valued coordinate functions $\mathcal{S}^{n-1}\raw \R$, $x_1,\,x_2,\,\dots x_n$ which are defined, for $\bld{v}=(a_1,\dots,a_n)$ by:
$$x_i(\bld{v})=a_i.$$
Now consider  the following universal $\mathrm{C}^*$-algebra:
$$C^*_{\text{comm}}\left(x_1,\dots,x_n\,|\,x_i\text{ self-adjoint and }\sum x_i^2=1\right).$$
With the standard topology on $\mathcal{S}^{n-1}$, this $\mathrm{C}^*$-algebra is isomorphic to $C(\mathcal{S}^{n-1})$ --- the algebra of \emph{continuous} functions on $\mathcal{S}^{n-1}$. Now consider the same universal $\mathrm{C}^*$-algebra except that commutativity is not included:
$$C^*\left(x_1,\dots,x_n\,|\,x_i\text{ self-adjoint and }\sum x_i^2=1\right).$$
Banica and Goswami \citep{FS} denotes this $\mathrm{C^*}$-algebra by $C(\mathbb{S}_+^{n-1})$ and call the virtual object $\mathbb{S}^{n-1}_+$ the \emph{free $n$ sphere}. In fact, Banica and Goswami go even further and talk about the action of the virtual group $\mathbb{O}_n^+$ on the virtual space $\mathbb{S}^{n-1}_+$.
\end{enumerate}


Given an arbitrary, not-necessarily commutative $\mathrm{C}^*$-algebra, $A$, the \emph{Gelfand Philosophy} says that $A$ should be considered the algebra of functions on a quantum space, $\mathbb{X}$:
$$A=C_0(\mathbb{X}).$$
Note again that $\mathbb{X}$ is a virtual object: in the notation of the previous section, it corresponds to the space, $X(A)$, `induced' by the algebra of functions $A$. The Gelfand Philosophy suggests some definitions, for example:
$$\mathbb{X}\text{ has cardinality $n$ if }A=C_0(\mathbb{X})\text{ has dimension }n,$$
$$\mathbb{X}\text{ is connected if }A=C_0(\mathbb{X})\text{ contains no non-trivial projections},$$
$$\mathbb{Y}\subset \mathbb{X}\text{ if }C_0(\mathbb{Y})\subset C_0(\mathbb{X}),$$
$$\mathbb{X}\text{ is compact if }A=C_0(\mathbb{X})\text{ is unital.}$$

\bigskip

This suggests one quantisation regime:  you quantise objects, such as Markov chains, by replacing each instance of  a commutative $\mathrm{C}^*$-algebra $C_0(X)$ with  a not-necessarily commutative one $A=C_0(\mathbb{X})$. Such a quantisation is called a \emph{liberation} by Banica and Speicher \citep{Banica}. A feature of any successful quantisation is that if  a restriction to a commutative subalgebra is made, it should be possible to recover a classical version. In many examples, quantisation is achieved just like this --- this thesis will quantise Markov chains in this way. However, for the quantisation of groups there is a slightly different approach that can be taken.

\section{The Quantisation Functor\label{category}}
This section is included as it is used to motivate  the correct notions of (the algebra of functions on) a \emph{quantum group}, a \emph{random walk on a quantum group} as well as (co)\emph{representation of a quantum group}. Note that the   `quantised' objects that are arrived at via this `categorical quantisation' are nothing but the established definitions so this section should be considered as little more than a motivation. The author feels that introductory texts on quantum groups could include these ideas and that is why they are included here. This quantisation is the translation of statements about a finite group, $G$ into statements about the algebra of functions on $G$, $F(G)$.

\bigskip

 This notion of quantisation sits naturally in category theory where two functors  --- the $\C$ functor and the dual functor --- lead towards a satisfactory quantisation. A \emph{category} $\mathfrak{C}$ consists of a class of \emph{objects}, a class of \emph{morphisms} and a composition law for morphisms. Denote the class of objects also by $\mathfrak{C}$ and the class of morphisms by $\operatorname{Mor}(\mathfrak{C})$. Each morphism $f\in\operatorname{Mor}(\mathfrak{C})$ has some \emph{source} $a\in \mathfrak{C}$ and  \emph{target} $b\in\mathfrak{C}$, and thus is denoted by $f:a\raw b$. The class of morphisms with source $a\in\mathfrak{C}$ and target $b\in\mathfrak{C}$ is denoted by $\operatorname{Mor}(a,b)$. Finally there is an associative binary operation:
 $$\circ:\operatorname{Mor}(b,c)\times\operatorname{Mor}(a,b)\raw \operatorname{Mor}(a,c).$$
  Note that for any object $x\in\mathfrak{C}$, there is a morphism $I_x\in\operatorname{Mor}(x,x)$ such that for all $f\in\operatorname{Mor}(a,b)$
  $$I_b\circ f=f=f\circ I_a.$$
  A commutative diagram is a \emph{quiver} with objects for vertices and morphisms for edges such that the composition morphism $a\raw b$ is path independent. That is, it expresses a family of equalities of morphisms of the form:
  $$\underbrace{\varphi_N\circ\cdots \circ \varphi_1}_{=:f}=\underbrace{\psi_M\circ\cdots \circ \psi_1}_{=:g},$$
  where $f$ and $g$ are morphisms from $a\raw b$.

  \bigskip

A \emph{functor} is a map between categories. A functor $F:\mathfrak{C}_1\raw\mathfrak{C}_2$ associates to each $a\in\mathfrak{C}_1$ an object $F(a)\in\mathfrak{C}_2$ and to each morphism $f\in\operatorname{Mor}(\mathfrak{C}_1)$ a morphism $F(f)\in \operatorname{Mor}(\mathfrak{C}_2)$ such that for all $x\in\mathfrak{C}_1$, $F(I_x)=I_{F(x)}$.
 \begin{itemize}
 \item $F$ is called \emph{covariant} if $F$ associates to a morphism $f\in \operatorname{Mor}(a,b)$ a morphism $F(f)\in\operatorname{Mor}(F(a),F(b))$ such that if $g\in\operatorname{Mor}(b,c)$
 $$F(g\circ_1 f)=F(g)\circ_2 F(f).$$
 Therefore the image of a commutative diagram under a covariant functor is another commutative diagram.
 \item $F$ is called \emph{contravariant} if $F$ associates to a morphism $f\in \operatorname{Mor}(a,b)$ a morphism $F(f)\in\operatorname{Mor}(F(b),F(a))$ such that if $g\in\operatorname{Mor}(b,c)$
 $$F(g\circ_1 f)=F(f)\circ_2 F(g).$$
Therefore the image of a commutative diagram under a contravariant functors is another commutative diagram with the morphisms ``turned around''.
 \end{itemize}
\subsubsection*{The $\C$ Functor}
The category of finite sets, $\mathbf{FinSet}$, has the class of all finite sets as objects and functions for morphisms. Also of interest is the category of finite dimensional complex vector spaces, $\mathbf{FinVec}_{\C}$ with linear maps for morphisms.  There is a map, the $\C$ map, $\C :\mathbf{FinSet}\raw \mathbf{FinVec}_{\C}$, that associates to each object $X\in\mathbf{FinSet}$, an object $\C X\in\mathbf{FinVec}_{\C}$ --- the complex vector space with basis $\{\delta^x:x\in X\}$. This map associates to each morphism $f:X\raw Y$ a morphism $\mathbb{C}f:\C X\raw \C Y$, $\delta^x\mapsto \delta^{f(x)}$ and it is not difficult to see that it is a covariant functor.

\bigskip

If $X$ and $Y$ are finite sets then $X\times Y$ is also a finite set. This object is sent to $\C(X\times Y)$ by the $\C$ functor. The following explains how to deal with $\C(X\times Y)$, as well as presenting a number of other useful isomorphisms of vector spaces.

\bigskip

\begin{theorem}
(Tensor Product Isomorphisms)
\begin{enumerate}
\item[(a)] Let $X$ and $Y$ be finite sets. Then, under the isomorphism $\delta^{(x,y)}\leftrightarrow \delta^x\otimes\delta^y$, $\C (X\times Y)\cong \C X\otimes \C Y$.

\bigskip

\item[(b)] Let $V$ be a finite dimensional complex vector space. Then, under the isomorphism $\lambda\otimes\bld{v}\leftrightarrow \lambda\bld{v}\leftrightarrow \bld{v}\otimes\lambda$, $\C \otimes V\cong V\cong V\otimes \C$.

\bigskip

\item[(c)] Let $U$ and $V$ be finite dimensional complex vector spaces. Then $(U\otimes V)^*\cong U^*\otimes V^*$.

\end{enumerate}
\end{theorem}
\begin{proof}
See Wegge-Olsen (Appendix T, \citep{WO}) for these standard results $\bullet$
\end{proof}

  Therefore, a morphism $f:X\times Y\rightarrow Z$ is sent to the linear map $\C f:\C X\otimes \C Y\raw \C Z$:
$$(\mathbb{C}f)(\delta^x\otimes \delta^y)=\delta^{f(x,y)}.$$

\subsubsection*{Dual Functor}
The dual map, $\mathcal{D}$, is a morphism in the category of finite dimensional vector spaces that sends a vector space to its dual and a linear map $T:U\rightarrow V$ to its transpose:
$$\mathcal{D}(T):V^*\rightarrow U^*,\qquad \varphi\mapsto \varphi\circ T.$$
It can be shown that for $T:V_1\raw V_1$ and $S:V_2\raw V_3$ that
$$(S\circ T)^*=T^*\circ S^*.$$
Let $\varphi\in V_3^*$:
\begin{align*}
(T^*\circ S^*)(\varphi)&=T^*\circ (S^*(\varphi))=S^*(\varphi)\circ T
\\&=\varphi\circ S\circ T=\varphi\circ(S\circ T)
\\&=(S\circ T)^*(\varphi).
\end{align*}
With this result, and the fact that $T^*$ is linear, the dual functor is a contravariant functor from $\mathbf{FinVec}_{\C}$ to itself.

\bigskip

Call the composition of these two functors by the quantisation functor:
$$\mathcal{Q}:\textbf{FinSet}\rightarrow \textbf{FinVec}_{\C},\qquad \mathcal{Q}=\mathcal{D}\circ\mathbb{C}.$$
It will be seen that the image of a group under this functor is the algebra of functions on the group. This gives us a routine to quantise groups and related objects: apply the $\mathcal{Q}$ functor to objects, morphism and commutative diagrams in the category of finite sets to get quantised objects, morphisms and commutative diagrams  in the category of finite dimensional vector spaces.

\bigskip

It will be seen that the image of a finite group under this functor, $\mathcal{Q}(G)=F(G)$, has the structure of a \emph{Hopf}-algebra: whose axioms are found simply by quantising the group axioms on $G$. That is $F(G)$ satisfies:
$$\mathcal{Q}(\text{group axioms})\sim\{\mathcal{Q}(\text{associativity}),\mathcal{Q}(\text{identity}),\mathcal{Q}(\text{inverses})\}.$$

There are, however, vector spaces together with morphisms that also satisfy these axioms but are not the algebra of functions on any group --- because the multiplication is no longer commutative. These are the algebras of functions on \emph{quantum groups}:
$$
\begin{CD}
F(G) @>{\mathcal{Q}(\text{group axioms})\text{ but not }ab=ba}>> F(\G)\\
@AA\mathcal{Q}A @AA\mathcal{Q}A\\
G @. \G
\end{CD}
$$
``Algebras of functions'' on quantum groups are algebras, $F(\G)$, that satisfy the quantisations of the group axioms, except letting go of commutativity in the algebra means that $\G$ is a virtual object.

\chapter{Quantum Groups}
In this chapter, the group axioms will be quantised. These quantised axioms imply associativity, identity and inverses in $G$ when $F(G)$ is commutative but do not imply commutativity of the algebra of functions on $F(G)$.
\section{Algebra of Functions on a Group and the Group Ring}
As a group is a `space', it is natural to study the algebra of functions on it. Let $G$ be a finite group and let $F(G)$ be the set of complex-valued functions on $G$. There is a natural $\mathrm{C}^*$-algebra structure on $F(G)$ defined by:
\begin{align*}
(f+g)(x)&=f(x)+g(x)&\qquad(f,\, g\in F(G))
\\ (\lambda f)(x)&=\lambda \,f(x)&\qquad (\lambda\in \C)
\\ M(f\otimes g)(x)&=f(x)g(x)&\qquad
\\ f^*(x)&=\overline{f(x)}&\qquad
\\ \|f\|&:=\max_{x\in G}|f(x)| &\qquad
\end{align*}
The unit is the indicator function $\ind_{G}$. As in the previous discussion, there are  relations that will always hold `up' in $F(G)$ as quantised versions of the relations `down' in $G$. The quantisation functor is used to see exactly what these relations look like in $F(G)$. Note that $F(G)$ is referred to as the \emph{algebra of functions on $G$} and is a commutative $\mathrm{C}^*$-algebra.

\bigskip

Also associated to a finite group   is another canonical algebra: the group ring. For $G$ a finite group, let $\C G$ be a complex vector space with basis elements $\{\delta^s\,:\,s\in G\}$. The scalar multiplication and vector addition are, for $\nu=\sum_t \alpha_t\delta^t$ and $\mu=\sum_t\beta_t\delta^t$, the natural ones:
\begin{align*}
\lambda \nu&=\sum_{t\in G}(\lambda\alpha_t)\delta^t\text{ and }
\\\nu+\mu&=\sum_{t\in G}(\alpha_t+\beta_t)\delta^t.
\end{align*}
The multiplication is given by:
$$\nabla(\delta^s\otimes\delta^t)=\delta^{st}.$$
The vector space $\C G$ together with the multiplication $\nabla$ is a complex associative algebra called the \emph{group ring of }$G$.
 Take an element $\nu$ of $\C G$:
$$\nu=\sum_{t\in G}\alpha_t\delta^t.$$
If the elements of $\C G$ are considered as complex-valued functions on $G$ via the embedding $s\hookrightarrow\delta_s$, $\nu(\delta_s)\hookleftarrow\nu(s)=\alpha_s$, a quick calculation shows that this multiplication $\nabla$ is nothing but the convolution:
$$\nabla(\nu\otimes \mu)(s)=(\nu\star \mu)(s)=\sum_{t\in G}\nu(st^{-1})\mu(t).$$
The unit is\footnote{more on $\widehat{G}$ after Proposition \ref{dualgroup}} $\delta^e=:\mathds{1}_{\widehat{G}}$. There is also an involution:
$$\nu^*=\sum_{t\in G}\overline{\alpha_t}\delta^{t^{-1}},$$
so that $\nu^*(s)=\overline{\nu(s^{-1})}$ under the identification above, turning $\C G$ into a *-algebra. Note that $\C G$ is commutative if and only if $G$ is abelian. Considering $\C G$ as a Hilbert space with an orthonormal basis $\{\delta^s:s\in G\}$, $\C G$ acts on  $\C G$ by left multiplication so  $\C G$ can be seen as an algebra of linear operators on the Hilbert space $\C G$ and thus a $\mathrm{C}^*$-algebra with the operator norm.

Note that $\C G$ can be identified with the algebraic dual of $F(G)$ via
$$\delta^s(\delta_t)=\delta_{s,t},$$
and as $G$ is finite dimensional:
$$F(G)^*=\C G\text{ and }\C G^*=F(G).$$
\section{Quantising Finite Groups\label{Haar}}
In this section the approach of Section \ref{category} is taken to quantising the group axioms. A group is an object in $\mathbf{FinSet}$ together with morphisms $m:G\times G\raw G$, $e:\{\bullet\}\raw G$ and $^{-1}:G\rightarrow G$ that satisfy:
\begin{align*}
m\circ(I_G\times m)&=m\circ (m\times I_G)
\\ m\circ (e\times I_G)&\cong I_G\cong m\circ(I_G\times e)
\\ m\circ(I_G\times{}^{-1})\circ \Delta_G&= e\circ \eps_G=m\circ({}^{-1}\times I_G)\circ \Delta_G.
\end{align*}
The second commutative diagram invokes the isomorphism $\{\bullet\}\times G\cong G\cong G\times\{\bullet\}$ while the third uses the maps $\Delta_G:G\raw G\times G$, $s\mapsto (s,s)$ and $\eps_G:G\raw \{\bullet\}$.

\bigskip

Now apply the covariant $\C$ functor to $G$, the three morphisms and these three commutative diagrams. Firstly the image of $G$ is $\C G$. The image of the group multiplication is the linear multiplication $\C m=:\nabla:\C G\otimes \C G\raw \C G$:
$$(\C m)(\delta^s\otimes \delta^t)=\delta^{m(s,t)}=\delta^{st}=\nabla(\delta^s\otimes \delta^t).$$
Note that $\C\{\bullet\}\cong \C$ and so $(\C e):\C\raw \C G$:
$$(\C e)(1)\cong (\C e)(\delta^{\bullet})=\delta^{e(\bullet)}=\delta^e.$$
Note that $\delta^e$ is the unit of $\C G$ and so denote by $\eta_{\C G}:=\C e$ the \emph{unit map}. The image of ${}^{-1}$ is the linear map $\operatorname{inv}:\C G\raw \C G$, $\delta^s\mapsto \delta^{s^{-1}}$. Note
$$(\C \Delta_G)(\delta^s)=\delta^{\Delta_Gs}=\delta^{(s,s)}\cong \delta^s\otimes\delta^s,$$
and denote $\C\Delta_{G}=:\Delta_{\C G}$. Finally
$$(\C \eps_G)(\delta^s)=\delta^{\eps_G(s)}=\delta^{\bullet}\cong 1,$$
and denote $\C \eps_G=:\eps_{\C G}$.

\bigskip

The image of the commutative diagrams above are therefore given by:
\begin{align}
\nabla\circ(I_{\C G}\otimes \nabla)&=\nabla\circ (\nabla\otimes I_{\C G})\nonumber
\\ \nabla\circ (\eta_{\C G}\otimes I_{\C G})&\cong I_{\C G}\cong \nabla\circ(I_{\C G}\otimes \eta_{\C G})\label{Caxioms}
\\ \nabla\circ(I_{\C G}\otimes\operatorname{inv})\circ \Delta_{\C G}&= \eta_{\C G}\circ \eps_{\C G}=\nabla\circ(\operatorname{inv}\otimes I_{\C G})\circ \Delta_{\C G}.\nonumber
\end{align}
Indeed, the first two commutative diagrams here show that $\C G$ together with $\nabla$ and $\eta_{\C G}$ is  an \emph{algebra}.

\bigskip

To fully quantise the group, the contravariant dual functor must be applied to $\C G$, the morphisms and the commutative diagrams. First note that $(\C G)^*=F(G)$. The multiplication $\nabla:\C G\otimes \C G\rightarrow \C G$ has a dual:
$$\nabla^*:(\C G)^*=F(G)\raw (\C G\otimes\C G)^*\cong (\C G)^*\otimes(\C G)^*=F(G)\otimes F(G)\cong F(G\times G),$$
where the last isomorphism can be seen as a consequence of $\C(X\times Y)^*=F(X\times Y)$. Embed the group $G$ in the group ring $\C G$ via $s\hookrightarrow\delta^s$ and consider for $f\in F(G)$:
$$\nabla^*f(s,t)\hookrightarrow\nabla^*f(\delta^s\otimes\delta^t)=f\circ \nabla(\delta^s\otimes\delta^t)=f(\delta^{st})=f(st).$$
This map $\nabla^*=:\Delta:F(G)\raw F(G)\otimes F(G)$, $\Delta f(s,t)=f(st)$, is the \emph{comultiplication on $F(G)$}. Note that $\Delta(\delta_{s})$ is the indicator function on $m^{-1}(s)$ so that after the identification $F(G\times G)\cong F(G)\otimes F(G)$,
$$\Delta(\delta_{s})=\sum_{t\in G}\delta_{st^{-1}}\otimes\delta_{t}.$$
Now consider the unit map  $\eta_{\C G}:\C\raw \C G\cong \C \otimes\C G$, $\delta^{\bullet}\cong 1\mapsto  \delta^e\cong 1\otimes \delta^e$. The dual of this map is
$$\eta_{\C G}^*:(\C\otimes \C G)^*\cong \C^*\otimes \C G^*\cong \C\otimes F(G)\cong F(G)\raw \C^*\cong \C.$$
Consider an element $f\in F(G)$:
$$\eta_{\C G}^*(f)(\delta^{\bullet})=f\circ \eta_{\C G}(\delta^{\bullet})=f(\delta^e)\hookleftarrow f(e),$$
so that $\eta_{\C G}^*(f)=f(e)$. This map $\eta_{\C G}^*=:\eps$ is called the \emph{counit}.

\bigskip

The inverse map $\operatorname{inv}:\C G\rightarrow \C G$ has a dual $\operatorname{inv}^*:F(G)\raw F(G)$ which (via the embedding) is given by:
$$\operatorname{inv}^*(f)(s)\hookrightarrow \operatorname{inv}^*(f)(\delta^s)=f\circ\operatorname{inv}(\delta^s)=f(\delta^{s^{-1}})=f(s^{-1}).$$
This map $\operatorname{inv}^*=:S$ is called the \emph{antipode}.

\bigskip

These are the most important dualisations of maps but there are two more namely $\Delta_{\C G}$  and $\eps_{\C G}$. Note that $\Delta_{\C G}(\delta^s)=\delta^s\otimes\delta^s$ maps from $\C G$ to $\C G\otimes\C G$ so that
$$\Delta_{\C G}^*:(\C G\otimes \C G)^*\cong F(G)\otimes F(G)\raw F(G).$$
Let $f,\,g\in F(G)$ and $\delta^s\hookleftarrow s\in G$:
$$\Delta^*_{\C G}(f\otimes g)(\delta^s)=(f\otimes g)\Delta_{\C G}(\delta^s)=(f\otimes g)(\delta^s\otimes \delta^s)=f(\delta^s)g(\delta^s)\hookleftarrow f(s)g(s),$$
so that $\Delta_{\C G}^*=:M$ is just the pointwise multiplication on $F(G)$.  Finally consider the map $\eps_{\C G}:\C G\raw \C$, $\delta^s\mapsto 1$. Its dual $\eps_{\C G}^*:\C\raw F(G)$ is the unit map of $F(G)$ ($\mathds{1}_G=\sum_{t}\delta_t$ is the unit of the algebra $F(G)$) as can be seen by taking any $\delta^s\hookleftarrow s\in G$:
$$\eps_{\C G}^*(\lambda)(\delta^s)=\lambda\circ\eps_{\C G}(\delta_s)=\lambda\cdot 1=\lambda,$$
that is $\eps_{\C G}(\lambda)=\lambda \cdot \mathds{1}_G$, i.e. $\eps_{\C G}^*=\eta_{F(G)}$.

\bigskip

Now that the morphisms have been identified:
\begin{align}
\mathcal{Q}(m)&=\Delta\qquad&;&\qquad \delta_s\mapsto \ind_{m^{-1}(s)}
\\ \mathcal{Q}(e)&=\eps \qquad&;&\qquad \delta_s\mapsto \delta_{s,e}
\\ \mathcal{Q}({}^{-1})&=S\qquad&;&\qquad \delta_s\mapsto \delta_{s^{-1}}
\\ \mathcal{Q}(\Delta_G)&=M\qquad&;&\qquad f\otimes g\mapsto fg
\\ \mathcal{Q}(\eps_G)&=\eta_{F(G)}\qquad&;&\qquad \lambda\mapsto \lambda \cdot \mathds{1}_G.
\end{align}

Applying the dual functor to the  commutative diagrams (\ref{Caxioms}) gives \emph{coassociativity}, the $\emph{counital}$ property and the \emph{antipodal} property:
\begin{align*}
(\Delta\otimes I_{F(G)})\circ \Delta&=(I_{F(G)}\otimes \Delta )\circ \Delta
\\ (\eps\otimes I_{F(G)})\circ \Delta&\cong I_{F(G)}\cong (I_{F(G)}\otimes \eps)\circ \Delta
\\ M\circ (S\otimes I_{F(G)})\circ \Delta&=\eta_{F(G)}\circ \eps=M\circ (I_{F(G)}\otimes S)\circ \Delta
\end{align*}
The first two commutative diagrams here show that $F(G)$ together with $\Delta$ and $\eps$ is a \emph{coalgebra}.

\bigskip

Now take an object $H$ in the category of finite vector spaces with morphisms that satisfy these `quantised' axioms. Such an object will be seen to be a finite Hopf-algebra and if it is non-commutative:
$$M_{H}(a\otimes b)\neq M_H(b\otimes a),$$
then it can be considered (if we make a few extra assumptions) the algebra of functions on a quantum group.

\newpage

The preceding quantisation gives, more or less, the correct definition of a finite quantum group. To be more precise, coalgebras, bialgebras and finally Hopf algebras are defined as follows. Assume that all spaces are  complex. To see more about the theory of bialgabras and Hopf algebras see the classic work of Abe \citep{Abe}. The following definitions are following\footnote{note that Timmermann doesn't require coalgebras to be counital. However he does require Hopf algebras to be. This work asks for counits at this point.} Abe.

\bigskip

\begin{definition}
A \emph{coalgebra} is a vector space $C$ together with a coassociative linear \emph{comultiplication} $\Delta:C\raw C\otimes C$ and a \emph{counit} $\eps\in C^*$ which has the counitary property.
\end{definition}

\subsubsection*{Remark: Sweedler Notation}
Let $C$ be a coalgebra, $c\in C$ and consider
$$\Delta(c)=\sum_ic_{1,i}\otimes c_{2,i}=:\sum c_{(1)}\otimes c_{(2)}.$$
This ``$\sum$'' with the subscripts $(1)$ and $(2)$ --- referring to the order of the factors in the tensor product --- is the notation of Sweedler. Since $\Delta$ is coassociative, the elements
$$\left(I_C\otimes \Delta\right)\circ\Delta(c)=\sum c_{(1)}\otimes \Delta(c_{(2)})=\sum c_{(1)}\otimes \left(c_{(2)}\right)_{(1)}\otimes \left(c_{(2)}\right)_{(2)}$$
and
$$\left(\Delta\otimes I_C\right)\circ\Delta(c)=\sum \Delta(c_{(1)})\otimes c_{(2)}=\sum \left(c_{(1)}\right)_{(1)}\otimes \left(c_{(1)}\right)_{(2)}\otimes c_{(2)}$$
are equal and so both may be written unambiguously as
$$\sum c_{(1)}\otimes c_{(2)}\otimes c_{(3)}.$$
More generally, an iterated comultiplication, $\dsp\Delta^{(k)}:C\raw \underset{k+1\text{ copies}}{\bigotimes} C$, can be  defined in various different ways, but all with the same resulting map. Therefore there is no ambiguity in writing
$$\Delta^{(k)}(c)=\sum c_{(1)}\otimes\cdots \otimes c_{(k+1)}.$$

\bigskip

Suppose that a space $A$ carries the structure of a unital (associative) algebra (with multiplication $M:A\otimes A\raw A$) and of a coalgebra. Then there is a \emph{unit map} $\eta:\C \raw A$, $\lambda\mapsto \lambda 1_A$ which satisfies
$$M\circ(I_A\otimes \eta)=I_A=M\circ(\eta\otimes I_A).$$

\begin{theorem}\label{bial}
The following are equivalent:
\begin{enumerate}
\item[(i)] $M,\,\eta$ are coalgebra morphisms,
\item[(ii)] $\Delta,\,\eps$ are algebra morphisms,
\item[(iii)] $\Delta(gh)=\sum g_{(1)}h_{(1)}\otimes g_{(2)}h_{(2)}$, $\Delta(1_A)=1_{A\otimes A}$, $\eps(gh)=\eps(g)\eps(h)$, $\eps(1_A)=1$.
\end{enumerate}
\end{theorem}
\begin{proof}
  See Abe \citep{Abe} for the definitions of algebra/coalgebra morphisms and the proof (Theorem 2.1.1) $\bullet$
\end{proof}

\bigskip

\begin{definition}
A \emph{bialgebra} is a space $A$ that is simultaneously an algebra and a coalgebra such that the two structures relate according to one of the equivalent conditions of Theorem \ref{bial}.
\end{definition}

\bigskip

A bialgebra is the appropriate quantisation of a semigroup. To get closer to a working definition of a quantum group, inverses must be accounted for.

\bigskip

\begin{definition}
A \emph{Hopf-algebra} is a bialgebra $A$ with a linear map $S:A\raw A$ with the antipodal property.
\end{definition}

\bigskip

 The involution must also be accounted for:

\bigskip

\begin{definition}
A \emph{Hopf $*$-algebra} is a Hopf algebra with an involution that satisfies $\Delta(a^*)=\Delta(a)^*$ for all $a\in A$ where the involution on $A\otimes A$ is given by
$$(a\otimes b)^*=a^*\otimes b^*.$$
\end{definition}

Every algebra of the form $F(G)$ for $G$ a finite group satisfies these relations by construction.

\bigskip

\begin{theorem}
Two finite groups $G_1$ and $G_2$ are isomorphic as groups if and only if $F(G_1)$ and $F(G_2)$ are isomorphic as Hopf $*$-algebras \citep{Kuster} $\bullet$
\end{theorem}

\bigskip

Note that all of these algebras are commutative and the question is begged:
\begin{quote}
\emph{Is there an example of a non-commutative Hopf $*$-algebra?}
\end{quote}
The answer is \emph{YES} and such an algebra has been seen already.

\bigskip

\begin{proposition}\label{dualgroup}
For a finite and not-necessarily abelian group $G$, the group ring $\C G$, together with maps $\Delta_{\C G}:=M^*$, $\eps_{\C G}:=\eta_{F(G)}^*$ and $S_{\C G}=\operatorname{inv}$ is a Hopf $*$-algebra.
\end{proposition}
\begin{proof}
Note that
\begin{align*}
\Delta_{\C G}(\delta^s)&=\delta^s\otimes \delta^s,
\\\eps_{\C G}(\delta^s)&=1,
\end{align*}

With these formulae, to show coassociativity and the counitary property is trivial. Once it is recalled that the multiplication on $\C G$ is given by $\nabla$ and the unit map $\eta_{\C G}(\lambda)=\lambda\cdot\delta^e$, the antipodal property is seen to hold $\bullet$\end{proof}

\bigskip

However this is not a truly quantum example because there is still a space underlying $\C G$ --- or rather $\C G$ is the dual of $F(G)$ which lies `above' the space $G$. The group ring $\C G$ has the property of being \emph{cocommutative}. This means that $\Delta_{\C G}=\tau\circ \Delta_{\C G}$ where $\tau:\C G\otimes \C G\raw \C G \otimes \C G$ is the \emph{flip} map $a\otimes b\mapsto b\otimes a$. Another example of a cocommutative Hopf $*$-algebra is the algebra of functions on an abelian group.

\bigskip

To  a finite abelian group $G$, one can associate a group $\widehat{G}$, the \emph{Pontryagin dual} or simply \emph{dual group of $G$} which consists of all characters on $G$: that is group homomorphisms $G\raw \mathbb{T}$. The group multiplication is just given by pointwise- multiplication. The map $T:\C G\raw F(\widehat{G})$, $(T\delta^s)(\chi)=\chi(s)$ for all $\chi\in \widehat{G}$ and $s\in G$ is an isomorphism of Hopf algebras (Example 1.4.3, \citep{Timm}). Taking a Gelfand-philosophical approach to this, the dual of a non-abelian group $G$ --- the virtual object $\widehat{G}$ --- may be given by $\C G=:F(\widehat{G})$. Later on in the work, for similar reasons, the dual of the algebra of functions on a quantum group, $F(\G)$, may be denoted by:
$$F(\G)^*=:\C\G=:F(\widehat{\G}).$$
\newpage
 The question therefore is:

\begin{quote}
\emph{Is there an example of a Hopf $*$-algebra that is neither commutative nor cocommutative?}
\end{quote}

The answer is YES (see Section \ref{KacP}) and using the Gelfand philosophy the notation $A=F(\mathbb{G})$ may be used.

\bigskip

A couple of times in this work, the \emph{Kac assumption} that $S^2=I$ is used. This assumption holds for both $F(G)$ and $\C G$ for $G$ a classical group. One might assume that this Kac assumption --- basically that the inversion map $s\mapsto s^{-1}$ is an involution --- must hold for quantum groups. However there turns out to be Hopf algebras that do not have this property.

\bigskip

The Sweedler algebra \citep{Sweedler} is a four-dimensional Hopf-algebra $A_S$ generated by $g$ and $x$ such that $g^2=1_{A_S}$, $x^2=0$ and $gx=-xg$. The comultiplication is given by $\Delta(g)=g\otimes g$ and $\Delta(x)=1_{A_S}\otimes x+x\otimes g$; the counit by $\eps(g)=1$ and $\eps(x)=0$ and the antipode by $S(g)=g$ and $S(x)=gx$.

  Note that the antipode is antimultiplicative (Proposition 1.3.12, \citep{Timm}) and so
$$S(S(x))=S(gx)=S(x)S(g)=gxg=g(-gx)=-g^2x=-x\neq x,$$
that is $S^2\neq I_{A_S}$.

\bigskip

One could restrict the definition of quantum groups to the setting of Hopf algebras with the Kac assumption but many algebras that morally \emph{should} be considered the algebras of functions on quantum groups --- such as deformed algebras (see e.g. \citep{Wor194}) do not have an involutive antipode. Therefore, although there is no one single definition of `quantum group', all `mainstream' quantum group theories do allow for non-involutive antipodes.

\bigskip

In a Hopf $*$-algebra, where $\tilde{S}(a):=S(a^*)$, it does hold that $\tilde{S}^2=I$. See Timmermann (Proposition 1.3.28, \citep{Timm}) for more.

\bigskip

The algebras studied in this work satisfy the Kac assumption however as a consequence of being finite dimensional (Proposition \ref{propofHaar}). There are definitions of quantum groups such as that of Kusterman and Tuset \citep{Kuster} which call for $a^*a=0\Leftrightarrow a=0$ for all $a\in A$. This is clearly a necessary condition for an algebra to be a $\mathrm{C}^*$-algebra; Franz and Gohm \citep{franzgohm} go further and ask that $A$ be a $\mathrm{C}^*$-Hopf algebra: the algebra carries a $\mathrm{C}^*$-algebra structure. For Gelfand-philosophical reasons,  this work follows Franz and Gohm.

\bigskip

\begin{definition}
An \emph{algebra of functions on a finite quantum group} $\G$ is a finite dimensional $\mathrm{C}^*$-Hopf algebra $A=F(\G)$. The \emph{order} of $\G$ is given by $|\G|:=\dim F(\G)$.
\end{definition}

\bigskip

\begin{theorem}(Classification Theorem)\label{class}
 Let $A$ be the algebra of functions on a finite quantum group $\G$:
  \begin{enumerate}
    \item[(a)] if $A$ is commutative then $\G\cong \Phi(A)$.
    \item[(b)] if $A$ is cocommutative then $A=F(\G)\cong \C \Phi(A)=:F(\widehat{\Phi(A)})$.
  \end{enumerate}
\end{theorem}
\begin{proof}
A well-known result. Theorem 3.3 of Vainerman and Kac an early result whose proof can be adapted for the definition of a quantum group used in this work \citep{class}
\end{proof}
Hence, a truly quantum group must be neither commutative nor cocommutative.

\subsection*{Haar Measure}
There are two critical reasons why a quantisation of Haar measure is required:

\bigskip

The interest in random walks on finite (classical) groups \citep{MSc} lies primarily with those which are \emph{ergodic}. A random walk is ergodic if the driving probability $\nu\in M_p(G)$ is such that the convolution powers, $\nu^{\star k}$, converge to the Haar measure, which for a finite group is the uniform distribution $\pi=\sum_{t\in G}\delta^t/|G|$.

\bigskip

Analytical techniques used in the analysis of the classical case that reference individual points in the space $G$ cannot be adapted to the quantum case. Analytical techniques that use the Haar measure do not fall under this bracket because the Haar measure is a sum over \emph{all} points rather than single points.

\bigskip

A topological group is a group endowed with a topology such that the group multiplication $m:G\times G\raw G$ and inverse ${}^{-1}:G\raw G$ are continuous. A \emph{compact group} $G$ is a compact topological group. Denote by $C(G)$ the continuous complex-valued functions on $G$.

\bigskip

Consider a compact group, $G$. The \emph{Borel} sets, $\mathcal{B}(G)\subset \mathcal{P}(G)$ --- the $\sigma$-algebra generated by the open sets of $G$ --- have a positive measure, $\mu:\mathcal{B}(G)\raw [0,\infty)$ that is invariant under translates:
$$\mu(s\Omega)=\mu(\Omega s)=\mu(\Omega).$$
Here $\Omega\in\mathcal{B}(G)$ and $s\in G$ and the \emph{translates} $s\Omega $ and $\Omega s$ are defined by
$$s\Omega=\{st:t\in \Omega\}\text{ and }\Omega s=\{ts:t\in \Omega\}.$$
This measure is called the \emph{Haar measure on $G$} and may be normalised so that $\mu(G)=1$. This is a classical result ($\S$ 58, Theorem B, \citep{HaarMeasure}).

\bigskip

Using  Lebesgue integration, $\mathcal{B}(G)$-measurable functions may be integrated on $G$:
$$\int f\,d\mu:=\int_G f(t)\,d\mu(t).$$

\begin{proposition}
The map $h:C(G)\raw \C$, $f\mapsto \dsp\int f\,d\mu$ is a state, invariant under translates.
\end{proposition}
\begin{proof}
It is clear from linearity of integration that $h$ is a functional. The Haar measure is positive so that $h$ is positive. The unit on $C(G)$ is the indicator function on $G$, which is simple and so the normalisation of the Haar measure ensures that $h$ is a state.

\bigskip

Let $s\in G$ and consider the \emph{left translate} $L_sf\in C(G)$ defined by
$$L_sf(t)=f(st).$$
Using the invariance of $\mu$ under translations it is possible to show that:

$$\int L_s f\,d\mu=\int f\,d\mu,$$
and a similar result for \emph{right translates} $\bullet$

\end{proof}
The map $h\in C(G)^*$ is called the \emph{Haar state} of the algebra of (continuous) functions on $G$. Its quantisation will yield a Haar state $h$ for the algebra of functions on a quantum group, $F(\G)$. Using the Gelfand philosophy, there is a virtual measure $\mathbb{\mu}$ on the virtual object $\G$ defined by
$$h(a)=:\int_{\G}a\,d\mathbb{\mu}.$$
Using the Gelfand philosophy, and an abuse of terminology, the Haar state is simply referred to as the Haar measure and is simply written
$$h(a)=\int_{\mathbb{G}}a.$$

\bigskip

As the classical Haar measure is a map $\mathcal{B}(G)\rightarrow \mathbb{C}$, a quantisation via the $\mathcal{Q}$-functor is not straightforward.
   On the other hand, with $\ind_{\Omega}=\sum_{t\in \Omega}\delta_t$ for $\Omega\subset G$, the  following can be considered
\begin{align*}
\left(\int_G \otimes I_{F(G)}\right)\circ \Delta(\ind_\Omega)&=\left(\int_G \otimes I_{F(G)}\right)\sum_{s\in G}\ind_{\Omega s}\otimes\delta_{s^{-1}}
\\&=\sum_{s\in G}\int_G\ind_{\Omega s}\otimes \delta_{s^{-1}}
\\&\cong \sum_{s\in G}\mu(\Omega s)\delta_{s^{-1}}=\sum_{s\in G}\mu(\Omega )\delta_{s^{-1}}
\\&=\mu(\Omega )\sum_{s\in G}\delta_{s^{-1}}=\int_G\ind_\Omega \cdot \ind_{G}.
\end{align*}
and therefore the right-invariance of the classical $\dsp\int_G \in F(G)^*$ gets quantised as:
$$\left(\int_{\G}\otimes I_{F(\G)}\right)(\Delta(f))=\int_{\G} f\cdot \mathds{1}_{\G}.$$
Similarly left-invariance is given by
$$\left(I_{F(\G)}\otimes \int_{\G}\right)(\Delta(f))=\int_{\G}f\cdot \mathds{1}_{\G}.$$

%


\bigskip

An element of $\C G$ that is both left- and right-invariant is simply said to be \emph{invariant}. In Section \ref{MC}, the set $M_p(\G)\subset \C\G$ --- the set of states on $F(\G)$ --- will be  defined.

\bigskip

\begin{definition}
The \emph{Haar measure} of a quantum group $\G$ is given by a normalised, invariant state $\dsp h=\int_{\G}\in M_p(\G)$.
\end{definition}

\bigskip

\begin{remark}
In the particular case of a finite classical group, the Haar measure is $\mu(S)=|S|/|G|$ and so it follows that the Haar measure of $f$ is nothing but the mean-average:
$$\int_G f=\frac{1}{|G|}\sum_{t\in G}\delta^t(f)=\frac{1}{|G|}\sum_{t\in G}f(t).$$
\end{remark}

\bigskip

Van Daele (Theorem 1.3, \citep{VDHM}) proves the existence and uniqueness of the Haar measure on  finite quantum groups. The proof of the following may also be found therein.


\begin{theorem}\label{propofHaar}
If $\G$ is a finite quantum group then the antipode $S$ is an involution and the Haar measure $\dsp\int_{\G}$ is tracial:
$$\int_{\G}ab = \int_{\G}ba\text{ \,\, for }a,\,b\in F(\G)\,\,\,\bullet$$
\end{theorem}
Therefore, with the finiteness assumption, all quantum groups in this work have $S^2=I_{F(\G)}$ and a tracial Haar state.

\bigskip

In the richer category of \emph{compact} quantum groups (see Section \ref{compact}), the above theorem is recast.

\bigskip

\begin{theorem}(Woronowicz \citep{202})
If $\G$ is a compact quantum group then the antipode $S$ is an involution if and only if Haar measure $\dsp\int_{\G}$ is tracial $\bullet$
\end{theorem}

\bigskip

\begin{remark}
If the above remark about the Haar measure giving the mean-average of a function is taken to give a definition of the mean-average of a function on a compact quantum group:
$$\overline{a}:=\int_\G a,$$
then the theorem of Woronowicz allows us to remark that when a compact quantum group $\G$ has an involutive antipode   all pairs of  functions $a,\,b$ on $\G$ commute on average:
$$\overline{ab}=\int_\G ab=\int_\G ba=\overline{ba}\Raw \overline{[a,b]}=\int_\G [a,b]=0,$$
where $[a,b]=ab-ba$ is the \emph{commutator} of $a$ and $b$.
\end{remark}

\section{The Kac--Paljutkin Quantum Group\label{KacP}}
 Kac and Paljutkin introduced a truly quantum group \citep{KP6} --- a quantum group $\mathbb{KP}$ of order eight --- and it is the smallest such object.

\bigskip

Franz and Gohm \citep{franzgohm} introduce the quantum group in some detail but here only a flavour is given. The algebra $A:=F(\mathbb{KP})$ may be realised as the direct sum
$$F(\mathbb{KP})=\C\oplus\C\oplus\C\oplus\C\oplus M_2(\C).$$
with the usual matrix multiplication and conjugate-transpose involution. The elements of the standard basis are denoted by $e_i$ for the first four factors and $a_{ij}$ for the fifth factor. The unit, $\mathds{1}_{\mathbb{KP}}$, is canonical. The comultiplication, $\Delta: F(\mathbb{KP})\raw F(\mathbb{KP})\otimes F(\mathbb{KP})$, is detailed in Franz and Gohm. The counit, $\eps:F(\mathbb{KP})\raw \C$, is given by the coefficient of the first factor while the antipode, $S:F(\mathbb{KP})\raw F(\mathbb{KP})$, is just the matrix-transpose. The Haar measure, $\dsp\int_{\mathbb{KP}}\in M_p(\mathbb{KP})$, is given by
$$\int_{\mathbb{KP}}\left(x_1\oplus x_2\oplus x_3\oplus x_4\oplus \left(\begin{array}{cc} c_{11} & c_{12} \\ c_{21} & c_{22}\end{array}\right)\right)=\frac{1}{8}\left(x_1+x_2+x_3+x_4+2c_{11}+2c_{22}\right).$$
It is nothing but a tedious exercise to show that $\mathbb{KP}$ is a quantum group.

\section{The Sekine Quantum Groups}
Sekine \citep{Sekine} introduced a family a finite quantum groups of order $2n^2$ that are neither commutative nor cocommutative.

\bigskip

The following follows the presentation of Franz and Skalski \citep{idempotent} rather than of Sekine. Let $n\geq 3$ be fixed and $\zeta_n=e^{2\pi i/n}$ and
$$\Z_n=\{0,1,\dots,n-1\},$$
 with addition modulo $n$.

 \bigskip

 Consider $n^2$ one-dimensional spaces $\C e_{(i,j)}$ spanned by elements indexed by $\Z_n\times\Z_n$, $\{e_{(i,j)}:i,j\in \Z_n\}$. Together with a copy of $M_n(\C)$, spanned by elements $E_{ij}$ indexed by $\{(i,j)\,:\,i,j=1,\dots, n,\, 0\equiv n\}$, a direct sum of these $n^2+1$ spaces, the $2n^2$ dimensional space
 $$A_n=\left(\bigoplus_{i,j\in\Z_n}\C e_{(i,j)}\right)\oplus M_n(\C),$$
 can be given the structure of the algebra of functions on a finite quantum group denoted by $\mathbb{KP}_n$ (so that $A_n=F(\mathbb{KP}_n)$). On the one dimensional elements the comultiplication is given by, for $i,\,j\in \Z_n$:
 \beq \Delta(e_{(i,j)})=\sum_{\ell,m\in\Z_n}(e_{(\ell,m)}\otimes e_{(i-\ell,j-m)})+\frac{1}{n}\sum_{\ell,m=1}^n\left(\zeta_n^{i(\ell-m)}E_{\ell,m}\otimes E_{\ell+j,m+j}\right).\label{oneD}\enq
 On the matrix elements in the $M_n(\C)$ factor:
 \beq\Delta(E_{i,j})=\sum_{\ell,m\in\Z_n}(e_{(-\ell,-m)}\otimes \zeta_n^{\ell(i-j)}E_{i-m,j-m})+\sum_{\ell,m\in\Z_n}\left(\zeta_n^{\ell(j-i)}E_{i-m,j-m}\otimes e_{(\ell,m)}\right)\label{Mfact}\enq
 The antipode is given by $S(e_{(i,j)})=e_{(-i,-j)}$ on the one dimensional factors and the transpose for the $M_n(\C)$ factor. Sekine does not give the counit but by noting that $u_{(0,0)}=I_{n}$ (where $U\in M_n(M_n(\C))$ is defined in Sekine's original paper),  it can be seen  that the coefficient of the $e_{(0,0)}$ one-dimensional factor satisfies the counital property. The Haar measure $\dsp\int_{\mathbb{KP}_n}\in M_p(\mathbb{KP}_n)$ is given by:
 $$\int_{\mathbb{KP}_n}\left(\sum_{i,j\in \Z_n}x_{(i,j)}e_{(i,j)}+a\right)=\frac{1}{2n^2}\left(\sum_{i,j\in\Z_n}x_{(i,j)}+n\cdot \text{Tr}(a)\right).$$

 \bigskip

 Although Sekine restricts his construction to $n\geq 3$, for $n=1$ and $n=2$ the construction still satisfies the conditions of Kac and Paljutkin \citep{KP6} and so are algebras of functions of quantum groups. Sekine does not clarify but the construction for $n=2$ does not give the celebrated Kac--Paljutkin quantum group of order eight and indeed $\KP_2$ is commonly mistaken for $\KP$ in the literature. Here it is shown that $\KP_1$ is classical and $\KP_2$ a (virtual) dual group.

 \bigskip

 For $n=1$, $\zeta_1=1$ and the construction gives an algebra structure
$$F(\mathbb{KP}_1)=\C\oplus M_1(\C)\cong \C^2,$$
with basis elements $e_1:=e_{(0,0)}$ and $e_2:=E_{11}$. It is straightforward to show that
$$\Delta(e_1)=e_1\otimes e_1+e_2\otimes e_2 \qquad\text{ and }\qquad\Delta(e_2)=e_1\otimes e_2+e_2\otimes e_1,$$
and so via $\varphi(e_1)=\delta_0$ and $\varphi(e_2)=\delta_1$, $\mathbb{KP}_1\cong \Z_2$.

\bigskip

For $n=2$, $\zeta_2=-1$ and so Sekine's construction gives an algebra structure
$$F(\mathbb{KP}_2)=\C^4\oplus M_2(\C).$$
  Define $e_1:=e_{(0,0)}$, $e_2:=e_{(1,1)}$, $e_3:=e_{(0,1)}$ and $e_4:=e_{(1,0)}$ and denote $\{a,b\}:=a\otimes b+b\otimes a$. The comultiplication on $F(\mathbb{KP}_2)$ is given by:
\begin{align*}
\Delta(e_1)&=e_1\otimes e_1+e_2\otimes e_2+e_3\otimes e_3+ e_4\otimes e_4
\\&+\frac12 E_{11}\otimes E_{11}+\frac12 E_{12}\otimes E_{12}+\frac12 E_{21}\otimes E_{21}+\frac12 E_{22}\otimes E_{22},
\\ \Delta(e_2)&=\{e_1,e_2\}+\{e_3,e_4\}+\frac12\{E_{11},E_{22}\}-\frac12\{E_{12},E_{21}\},
\\ \Delta(e_3)&=\{e_1,e_3\}+\{e_2,e_4\}+\frac12\{E_{11},E_{22}\}+\frac12 \{E_{12},E_{21}\},
\\ \Delta(e_4)&=\{e_1,e_4\}+\{e_2,e_3\}
\\&+\frac12 E_{11}\otimes E_{11}-\frac12 E_{12}\otimes E_{12}-\frac12 E_{21}\otimes E_{21}+\frac12 E_{22}\otimes E_{22},
\\ \Delta(E_{11})&=\{e_1,E_{11}\}+\{e_2,E_{22}\}+\{e_3,E_{22}\}+\{e_4,E_{11}\},
\\ \Delta(E_{12})&=\{e_1,E_{12}\}-\{e_2,E_{21}\}+\{e_3,E_{21}\}-\{e_4,E_{12}\},
\\ \Delta(E_{21})&=\{e_1,E_{21}\}-\{e_2,E_{12}\}+\{e_3,E_{12}\}-\{e_4,E_{21}\},
\\ \Delta(E_{22})&=\{e_1,E_{22}\}+\{e_2,E_{11}\}+\{e_3,E_{11}\}+\{e_4,E_{22}\}.
\end{align*}
A very quick inspection shows that, where $\tau$ is the flip map:
$$\tau\circ \Delta=\Delta,$$
and so $F(\mathbb{KP}_2)$ is cocommutative, hence isomorphic to a group ring $\C G$ by the Classification Theorem \ref{class}. As there is an $M_2(\C)$ factor, $\C G$ is noncommutative and so $G$ is non-abelian. The dimension of $F(\mathbb{KP}_2)$ is eight so $G$ is isomorphic to the dihedral group of order four, $D_4$, or the quaternion group. To find the group-like elements, for a general
$$x=ae_1+be_2+ce_3+de_4+a_{11}E_{11}+a_{12}E_{12}+a_{21}E_{21}+a_{22}E_{22},$$
solving $\Delta x=x\otimes x$ gives the following elements of $\C G$:
\begin{align*}
\delta^{s_1}&=e_1+e_2+e_3+e_4+E_{11}+E_{22},
\\ \delta^{s_2}&=e_1+e_2+e_3+e_4-E_{11}-E_{22},
\\\delta^{s_3}&=e_1-e_2-e_3+e_4+E_{11}-E_{22},
\\\delta^{s_4}&=e_1-e_2-e_3+e_4-E_{11}+E_{22},
\\\delta^{s_5}&=e_1-e_2+e_3-e_4+E_{12}+E_{21},
\\\delta^{s_6}&=e_1-e_2+e_3-e_4-E_{11}-E_{12},
\\\delta^{s_7}&=e_1+e_2-e_3-e_4+E_{12}-E_{21},
\\\delta^{s_8}&=e_1+e_2-e_3-e_4-E_{12}+E_{21}.
\end{align*}
A quick calculation shows that there are at least three elements of order two and so $\mathbb{KP}_2$ is equal to the virtual object $\widehat{D_4}$. Further calculations show that if $D_4$ is presented as
$$\langle x,y\,:\,x^2=e,\,y^4=e,\,(xy)^2=e\rangle,$$
that is $x$ is a reflection and $y$ an order four rotation, then $\varphi:G\rightarrow D_4$ is an isomorphism:
$$ s_1\mapsto e,\,s_2\mapsto y^2,\,s_3\mapsto x,\,s_4\mapsto y^2x,\,s_5\mapsto xy,\,s_6\mapsto yx,\,s_7\mapsto y,\,s_8\mapsto y^3.
$$
Note that in \citep{KP6}, to construct the celebrated quantum group of order eight that is neither commutative nor cocommutative, Kac and Paljutkin do not use the same construction as Sekine. In the notation of \citep{KP6}, the Sekine construction of $F(\mathbb{KP}_2)$ would use
$$K=p_{e}=I_2,\quad p_{\alpha}=\left(\begin{array}{cc}-1 & 0 \\ 0 & 1\end{array}\right),\quad p_{\beta}=\left(\begin{array}{cc}0 & 1 \\ 1 & 0\end{array}\right),\quad p_{\gamma}=\left(\begin{array}{cc}0 & -1 \\ 1 & 0\end{array}\right),$$
while the to construct the celebrated quantumn group $\mathbb{KP}$, Kac and Paljutkin use:
$$K=p_{e}=I_2,\quad p_{\alpha}=\left(\begin{array}{cc}0 & i \\ i & 0\end{array}\right),\quad p_{\beta}=\left(\begin{array}{cc}0 & 1 \\ i & 0\end{array}\right),\quad p_{\gamma}=\left(\begin{array}{cc}-1 & 0 \\ 0 & 1\end{array}\right).$$
A good place to see more examples of finite quantum groups --- including the construction of new quantum groups from old --- include the notes of Andruskiewitsch \citep{And} and the work of Banica, Bichon and Natale \citep{BBN}.

\section{The Dual of a Quantum Group --- Quantum Group Rings\label{quantumgroupring}}
Let $F(\G)$ be the algebra of functions on  a finite quantum group with Haar measure $\dsp\int_{\G}:F(\G)\raw \C$. Define $\Ahat$
 as the space of linear functionals on $F(\G)$ of the form
 $$b\mapsto \int_{\G} ba\,,\,\,\qquad (a\in F(\G)).$$
In a relatively natural way, this space can be given the structure of a quantum group and it is called the \emph{dual} of the quantum group $F(\G)$ and, as a nod to Pontryagin duality, using the Gelfand philosophy, $\Ahat$ is denoted by $\C \G$ or  $F(\widehat{\G})$.

 \bigskip

 As $F(\G)$ is finite dimensional, the continuous and algebraic duals coincide. Furthermore,  the Haar measure is faithful (Proposition 2.2.4, \citep{Timm}) and so
 $$\langle a,b\rangle:=\int_{\G} a^*b$$
  defines an inner product making $F(\G)$ into a Hilbert space. Via the Riesz Representation Theorem for Hilbert spaces, for every element  $\varphi\in F(\G)'$, there exists an element $a\in F(\G)$ such that:
  $$\varphi(b)=\langle a,b\rangle=\int_\G a^*b,$$
  so that $F(\G)'=\C\G$.  The dual of the comultiplication $\Delta:F(\G)\raw F(\G)\otimes F(\G)$ defines a multiplication on the dual, $\nabla:\C\G\otimes \C\G\raw \C\G$. In particular, for $\mu,\,\nu\in \C\G$ and $b\in F(\G)$
$$\nabla(\mu\otimes \nu)(b)=\Delta^*(\mu\otimes\nu)(b)=(\mu\otimes \nu)\Delta(b)=\sum \mu\left(b_{(1)}\right)\nu\left(b_{(2)}\right)$$
using Sweedler's notation. This multiplication on $\C\G$ is often called the \emph{convolution} and can be denoted by:
$$\nabla(\mu\otimes\nu)=\mu\star \nu.$$
Similarly, the dual of the multiplication, $M:F(\G)\otimes F(\G)\raw F(\G)$, defines a comultiplication on the dual, $\widehat{\Delta}:\C\G\raw\C\G\otimes\C\G$. In particular, for $\varphi\in\C\G$ and $a,\,b\in F(\G)$:
$$\widehat{\Delta}(\varphi)(a\otimes b)=M^*(\varphi)(a\otimes b)=\varphi\circ M(a\otimes b)=\varphi(ab).$$
The antipode on the dual, $\widehat{S}:\C\G\raw \C\G$, is just the dual of the antipode on $F(\G)$, $S:F(\G)\raw F(\G)$. Namely, for $\varphi\in \C\G$ and $a\in F(\G)$:
$$\widehat{S}(\varphi)(a)=S^*(\varphi)(a)=\varphi(S(a)).$$
The counit on the dual, $\widehat{\eps}:\C\G\raw \C$ is given by evaluation at the unit of $F(\G)$:
 $$\widehat{\eps}(\varphi)=\varphi(\mathds{1}_{\G})\qquad\qquad(\varphi\in \C\G).$$
 This object, $\C\G$, also possesses an involution. For $\varphi\in\C\G$, $a\in F(\G)$, $S$ the antipode on $F(\G)$ and the involution on the right-hand-side the involution on $F(\G)$:
$$\varphi^*(a)=\overline{\varphi\left(S(a)^*\right)}.$$
This gives $\C \G$ the structure of the algebra of functions on a quantum group. Denote this quantum group, the   \emph{dual quantum group of $\G$}, by $\widehat{\G}$ so that $F(\widehat{\G})=\C\G$.

\bigskip

There is also a Haar measure on $\widehat{\G}$. To write it down as a nice formula the bijective map that takes an element $a\in F(\G)$ to the map $b\mapsto \dsp\int_{\G}ba$ will  be denoted by $\mathcal{F}$:
$$\mathcal{F}:F(\G)\raw \C\G,\,\,a\mapsto \int_{\G}(\cdot a).$$
A distinguished functional on the dual, $\dsp \int_{\widehat{\G}}:\C\G\raw \C$, is given by
$$\int_{\widehat{\G}}\varphi=\int_{\widehat{\G}}\mathcal{F}(a)=\eps(a);$$
in other words $\dsp \int_{\widehat{\G}}=\eps\circ \mathcal{F}^{-1}$. Note this is \emph{not} the Haar state on the dual as it is not normalised --- $a\mapsto \int_{\widehat{\G}}a/\int_{\widehat{\G}}\eps$ is the Haar state on $\widehat{\G}$.

\bigskip

%

 Now, rewriting the work of Van Daele \citep{VD}, the basic properties of this map $\mathcal{F}$ are presented. First a lemma that leads to a Plancherel Identity.

 \bigskip

\begin{lemma}\label{lem251}
For $\varphi_1=\F(a_1),\,\varphi_2=\F(a_2)\in\C\G$
$$\int_{\widehat{\G}}\left(\varphi_1\star \varphi_2\right)=\varphi_1(S(a_2)).$$
\end{lemma}
\begin{proof}
See Van Daele \citep{VD2}, Lemma 4.11 for a proof $\bullet$
\end{proof}
This lemma yields a formula for the inverse of $\F$ (that is not used in the sequel).

\bigskip

\begin{theorem}
(Inversion Theorem)
Let $\G$ be a finite quantum group and $\widehat{\G}$ the associated  dual quantum group with Haar measure $\dsp\int_{\widehat{\mathbb{G}}}$. Consider an element $\varphi=\F(a)\in \C\G$. Then $\F^{-1}(\varphi)=a$ is an element of $\C\G^*\cong F(\G)$ whose action on $\mu\in\C\G$ is given by
$$\F^{-1}(\varphi)(\mu)=a(\mu)\cong\mu(a)=\int_{\widehat{\G}}\widehat{S}(\mu)\varphi.$$
\end{theorem}
\begin{proof}
Let $\varphi=\widehat{S}(\mu)$ and apply the above lemma (recalling $S^2=I_{F(\G)}$):
\begin{align*}
\int_{\widehat{\G}}\widehat{S}(\mu)\F(a)&=\widehat{S}(\mu)(S(a))
\\&=\mu\circ S(S(a))=\mu(a)\qquad\bullet
\end{align*}
\end{proof}
\begin{theorem}
(Plancherel Theorem)\label{Planch}
Let $\G$ be a quantum group with Haar measure $\dsp\int_{\G}$ and $\dsp\int_{\widehat{\mathbb{G}}}$ the Haar measure on $\widehat{\G}$. Then for all $a \in F(\G)$:
$$\int_{\widehat{\G}}\left(\F(a)^*\star\F(a)\right)=\int_{\G}a^*a.$$
\end{theorem}
\begin{proof}
Applying Lemma \ref{lem251} to the left-hand side:
\begin{align*}
\int_{\widehat{\G}}(\F(a)^*\star\F(a))&=\F(a)^*(S(a))=\overline{\F(a)\left(S(S(a))^*\right)}
\\&=\overline{\F(a)(a^*)}=\overline{\int_{\G}a^*a}=\int_{\G}a^*a.
\end{align*}
The last equality follows from the fact that $\dsp\int_{\G}$ is a positive linear functional $\bullet$
\end{proof}

\newpage
There is a convolution theorem relating the `convolution product' of $a,\,b\in A=F(\G)$:
 \beq a\star_A b:=\sum b_{(2)}\int_{\G}\left( S(b_{(1)})a\right)\label{conv2}\enq
 (where the Sweedler notation has been used), to the ordinary convolution multiplication in $\C\G$:
 $$\mu\star\nu=(\mu\otimes\nu)\Delta.$$

 \bigskip

 \begin{theorem}
(Van Daele's Convolution Theorem)\label{VDCT}
 For all $a,\,b\in A=F(\G)$
 $$\mathcal{F}\left(a\right)\star\mathcal{F}\left(b\right)=\mathcal{F}(a\star_{A}b),$$
 \end{theorem}
\begin{proof}
See \citep{VD}, Proposition 2.2 for a proof $\bullet$
\end{proof}
%

\chapter{Quantisation of Markov Chains and Random Walks}
\section{Markov Chains\label{MC}}
Consider a particle in a finite space $X=\{x_i:i=1,2,\dots,n\}$. Suppose at time $t=0$ the particle is at the point $x$, and at times $1,2,\dots$ moves to another point in the space chosen `at random'. The probability that the particle moves to a certain point $x_j$ at a time $t$ is dependent only upon its position $x_i$ at the previous time. This is the Markov property. A \emph{time-homogeneous} Markov chain is a mathematical process which models these dynamics. Such a Markov chain can be described by the \emph{transition probabilities} $p(x_i,x_j)$, which give the probability of the particle being at point $x_j$ given that the particle is at the point $x_i$ at the previous time.

\bigskip

To formulate, let $X$ be a finite set. Denote by  $M_p(X)$ the probability measures on $X$.  The \textit{Dirac measures},  $\{\delta^x:x\in X\}$, $\delta^x(\{y\})=\delta_{x,y}$, are the standard basis for $\R^{|X|}\supseteq M_p(X)$.   Denote by $F(X)$ the complex functions on $X$ and $L(V)$ the linear operators on a vector space $V$. The  \textit{Dirac functions}  are the standard basis for $F(X)$.  With respect to this basis $P\in L(F(X))$ has a matrix representation $[p(x,y)]_{xy}$. A map $P\in L(F(X))$ is a \textit{stochastic  operator} if:
\begin{itemize}
\item[(i)] $p(x,y)\geq 0,\,\,\forall x,y\in X$
\item[(ii)]  $\sum_{y\in X}p(x,y)=1$, $\forall x\in X$ \hfill(row sum is unity)
\end{itemize}
Given $\nu\in M_p(X)$, a stochastic operator $P$ acts on $\nu$ as $P^T\nu(x)=\sum_yp(y,x)\nu(y)$. Stochastic operators are readily characterised without using matrix elements  as being $M_p(X)$-stable in the sense that $P^T(M_p(X))\subset M_p(X)$ if and only if $P$ is a stochastic operator. Equivalently, stochastic operators are positive, unital linear maps $F(X)\rightarrow F(X)$. In this context, \emph{positive} means that if $F(X)^+$ is the set of functions with $f(x)\geq0$ for all $x\in X$, then $P(F(X)^+)\subset F(X)^+$. \emph{Unital} means that $P(\mathds{1}_{X})=\mathds{1}_X$.

\begin{definition}
Let $X$ be a finite set and $\nu\in M_p(X)$, $P$ a stochastic operator on $X$, and $(Y,\mathbb{P})$ a probability space. A sequence $\{\xi_k\}_{k=0}^n$ of random variables $\xi_k:Y\raw X$ is a \textit{Markov chain with initial distribution $\nu$ and stochastic operator $P$}, if
\begin{itemize}
\item[(i)] $\mathbb{P}(\xi_0=x_0)=\nu(x_0)$, and for $k\geq 1$
\item[(ii)] $\mathbb{P}(\xi_{k+1}=x_{k+1}\,|\,\xi_0=x_0,\dots,\xi_k=x_k)=p(x_k,x_{k+1})$,\new  assuming $\mathbb{P}(\xi_0=x_0,\dots,\xi_k=x_k)>0$.
\end{itemize}
\end{definition}
Condition (ii) is the \textit{Markov property}. Call $\xi_k$ the position of the Markov chain after $k$ \emph{transitions}. Subsequent references to a Markov
chain $\xi$ refer to a Markov chain $\left(\{\xi_i\}_{i=0}^k,P,\nu\right)$

\bigskip

Many questions may be asked about the local and global behaviour of a  Markov chain $\xi$. One could define local behaviour as the behaviour of the Markov chain with respect to the points of $X=\{x_1,x_2,\dots,x_n\}$, while global behaviour as the behavior of the Markov chain with respect to the whole of $X$ (i.e. no reference is made to distinct points of $X$). Alternatively, imagine a lidded black box  containing an evolving Markov chain. Local questions are questions that would be asked with the lid \emph{off}, while global questions are questions that would be asked with the lid \emph{on}. When a Markov chain is quantised, the notion of a point is now defunct and there can no longer be interest in the local behaviour.

\bigskip

Central questions about the global behaviour of classical Markov chains include:
\begin{itemize}
\item do the random variables, $\{\xi_k\}$, display limiting behaviour as $k\raw \infty$?
\item do stationary distributions exist?
\item how many stationary distributions exist?
\end{itemize}
These are all questions that can be asked in the quantum case.

\subsection*{$\mathrm{C}^*$-algebra Quantisation of a Classical Markov Chain}
So where is the $\mathrm{C}^*$-algebra in a Markov chain? Well let $\xi$ be a Markov chain on a set $X=\{x_1,\dots,x_n\}$ with initial distribution $\nu$ and transition probabilities $ \mathbb{P}[\xi_{k+1}=x_j|\xi_{k}=x_i]=p(x_i,x_j)=p_{ij}$. The probability distribution of this walk, after $k$ transitions, is given by $(P^k)^T\nu $. However, the probability measures on $X$, $M_p(X)$, lie in the dual of the $\mathrm{C}^*$-algebra $F(X)$ ($M_p(X)\subset \mathbb{R}^{n}$ is equipped with the 1-norm while $F(X)$ is equipped with the supremum norm). In fact, the probability measures comprise the states (defined below) of the $\mathrm{C}^*$-algebra $A$ as for any $\theta\in M_p(X)\subset F(X)^*$, $\theta$ is a positive linear functional of norm one.

\bigskip

Actually in this specific case ($X$ is a finite set) the positivity of the functional $\theta$ has two equivalent definitions (the second is the same as saying $\theta(\delta_x)\geq 0$ for all $x\in X$):
\begin{enumerate}
\item $\theta(f)\in\C^+= \mathbb{R}^+$ for all positive functions $f\in F(X)^+$.
\item In the basis of Dirac measures, $(\delta^{x_1},\dots,\delta^{x_n})$ --- the dual basis to the Dirac functions, $(\delta_{x_1},\dots,\delta_{x_n})$ --- the coefficients of $\theta$ are all positive.
\end{enumerate}
Usually talking about functionals on $\mathrm{C}^*$-algebras being positive refers to the first definition: i.e. a linear map $\varphi:C_0(\mathbb{X})\rightarrow C_0(\mathbb{Y})$ between $\mathrm{C}^*$-algebras is said to be positive if $\varphi(C_0(\mathbb{X})^+)\subset C_0(\mathbb{Y})^+$. The positive elements of a $\mathrm{C}^*$-algebra $C_0(\mathbb{X})$ are given by:
$$C_0(\mathbb{X})^+=\{a\in C_0(\mathbb{X})\,:\,a=b^*b\,\text{ for some }b\in C_0(\mathbb{X})\}.$$
The states of a general $\mathrm{C}^*$-algebra are given by:
$$S(C_0(\mathbb{X}))=\{\varphi\in C_0(\mathbb{X})^*:\|\varphi\|=1,\,\varphi\geq 0\}.$$
States correspond in the classical case to probability measures on $X$. Therefore the following notation is used:
$$M_p(\mathbb{X})=S(C_0(\mathbb{X})).$$

In this global picture of a classical Markov chain --- which looks at the deterministic evolution of $\{(P^i)^T\nu:i=0,1,\dots,k\}$ --- rather than the random variable picture of the $\xi_k:(Y,\mathbb{P})\rightarrow X$, there is thus an initial distribution $\nu\in M_p(X)$, a stochastic operator:
$$C(X)^*\rightarrow C(X)^*\,,\,\theta\mapsto P^T\theta$$
which is $M_p(X)$-stable, and the set of distributions $\{(P^k)^T\nu:k=1,\dots,n\}$ can be looked at to
 tell everything  about the Markov chain. For example, if $(P^k)^T(\nu)$ is convergent then  the walk converges and a fixed point of the stochastic operator is a stationary distribution. The only thing left to do to complete the \emph{liberation} is to put some conditions on a stochastic operator being $M_p(\mathbb{X})$-stable --- $P^T$ being isometric and positive is certainly enough (although \emph{serious} references on quantum Markov chains via this approach --- such as Accardi \citep{Accardi} --- usually ask that $P^T$ be \emph{completely positive}).

\bigskip

The $\mathrm{C}^*$-algebra quantisation is then as follows. Let $C_0(\mathbb{X})$ be a $\mathrm{C}^*$-algebra with dual $C_0(\mathbb{X})^*$. Choose an element $ \psi\in M_p(\mathbb{X})$ and a positive linear isometry $P^T:C_0(\mathbb{X})^*\rightarrow C_0(\mathbb{X})^*$ (which is automatically $M_p(\mathbb{X})$-stable). The distribution of the quantum Markov chain generated by $\psi$ and $P$ after $k$ transitions would then be given by $(P^k)^T\psi$. Good references for quantum Markov chains may be found in the introduction to the paper of Franz and Gohm \citep{franzgohm}.

%
%
%
%
%
%

%

\bigskip

 This construction is leaning towards the fact that the deterministic evolution of the $(P^k)^T\psi$ can tell us all about the global behaviour --- and this is desirable for quantisation. Alternatively, note that a stochastic operator is a unital, positive operator $P:F(X)\rightarrow F(X)$ and work from there.

\bigskip

However, it is equally valid (and indeed far more common), to examine a classical Markov chain, not as a deterministic evolution $\{(P^k)^T\nu\}_{k\geq 0}$, but rather as a random variable. Therefore, instead of a $\mathrm{C}^*$-algebra quantisation up in $F(X)^*=\C X$, a category theory quantisation could translate the random variable picture of random variables $\xi_k:Y\raw X$ down in $X$ up into $F(X)$. Subsequently, there could be a `lifting' to a deterministic evolution using the associated quantised stochastic operators and distributions.

 In the below diagram, the right arrows are `liftings' from the random variable picture to the deterministic picture. The $\mathrm{C}^*$-algebra quantisation --- simply liberating from commutative $\mathrm{C}^*$-algebras to noncommutative $\mathrm{C}^*$-algebras --- is denoted by $\mathcal{Q}_{\mathrm{C}^*}$ while the category theory quantisation is denoted as before by $\mathcal{Q}$:

$$
\begin{CD}
\{j_k\} @>\mathcal{L}>>(P^k)^T (M_p(\mathbb{X}))\\
@AA\mathcal{Q}A @AA\mathcal{Q}_{\mathrm{C}^*}A\\
\{\xi_k\} @>\mathcal{L}>> (P^k)^T (M_p(X))
\end{CD}
$$

To quantise a random walk on a group the category theory approach is inevitable --- if the structure of the group acting on itself is to be encoded.

\bigskip

To quantise in the random variable picture, the sequence of space-valued random variables, $\{\xi_k\}$, must be replaced by a sequence of function-valued random variables and, essentially, it will be seen that this is done by defining a sequence $\{j_i\}_{i=0}^k$ of random variables $F(Y)\raw F(X)$ by
\beq
j_i:=f\circ \xi_i.
\enq
For those more interested in Markov chains rather than random walks on quantum groups specifically, Diaconis, Pang and Ram \citep{DPR} use the \emph{Hopf square} $M\Delta:H\raw H$ (for $H$ a Hopf algebra) to generate Markov chains on different structures.
\section{Random Walks}
A particularly nice class of Markov chain is that of a \textit{random
walk on a group}. The particle moves from group point to group
point by choosing a point $h$ of the group `at random' and
moving to the product of $h$ and the present position $s$, i.e. the particle moves from $s$ to $hs$.

\bigskip

To formulate, let $G$ be a finite group, $\nu\in M_p(G)$ and $(G,\mathbb{P})$  a probability space. Let
$\{\zeta_i\}_{i=0}^k: (G^{k+1},\mathbb{P})\raw G$ be a sequence of
random variables
$$\zeta_i(g_0,g_1,\dots,g_k)=g_i,$$
 with distributions
$$\delta^e\star \underbrace{\nu\star\cdots \star \nu}_{k\text{ times}}.$$ The sequence of random variables
$\{\xi_i\}_{i=0}^k:(G^{k+1},\mathbb{P})\raw G$ \beq
\xi_i=\zeta_i\zeta_{i-1}\cdots\zeta_1\zeta_0 \enq is a
\textit{right-invariant random walk on $G$}.

\bigskip

Consider the category theory quantisation of $$\xi_1:G\times G\raw G\,, \,\,\,(\zeta_1,\zeta_0)\mapsto m(\zeta_1,\zeta_0)=\zeta_1\zeta_0.$$
Under the $\mathcal{Q}$ functor, $\mathcal{Q}(\xi_1)=\Delta$.

\bigskip

Considering
$$\xi_2(\zeta_2,\zeta_1,\zeta_0)=m(\zeta_2,\zeta_1\zeta_0)=\zeta_2\zeta_1\zeta_0,$$
shows that
$$\xi_2=m\circ(I_G\times m),$$
and so
$$\mathcal{Q}(\xi_2)=(I_{F(G)}\otimes \Delta)\circ \Delta=\Delta^{(2)}.$$

Inductively, the quantisation of a random walk on a group simply replaces the random variables $\{\xi_k\}$ by the random variables $\{j_k\}$, where
$$j_k=\mathcal{Q}(\xi_k)=\Delta^{(k)}:F(G)\rightarrow \bigotimes_{k+1\text{ copies}} F(G)$$
is the iterated comultiplication. Making the appropriate identifications of tensor copies of $F(G)$ with the algebra of functions on cartesian products of $G$, $F(G^{k})$, it can be seen that for an $f\in F(G)$ the quantisation implies that
  $$j_k(f)=f\circ \delta^{\xi_k}.$$
Franz and Gohm \citep{franzgohm} also describe the $j_i$ in terms of random variables $z_i$ in the same way that the classical $\xi_i$ can be described in terms of the classical $\zeta_i$. Recall that the $\{\zeta_i\}$ are a family of random variables $\zeta_i:(G^{k+1},\mathbb{P})\rightarrow G$, $(g_j)\mapsto g_i$, and the random walk $\xi$ is given by (where $m^{(k)}$ is the group multiplication on $G^{k+1}$)
$$\xi_i=m^{(k)}(\zeta_i, \zeta_{i-1},\cdots ,\zeta_1,\zeta_0).$$
Define $z_i:F(G)\raw F(G^{k+1})$ by $z_i:=\mathcal{Q}(\zeta_i)$ by $f\mapsto f\circ \delta^{\zeta_i}$. Then
applying the quantisation functor to
$$\xi_i=m^{(k)}(\zeta_i, \zeta_{i-1},\cdots ,\zeta_1,\zeta_0),$$
yields
$$j_k=(z_k\otimes z_{k-1}\otimes\cdots\otimes z_1\otimes z_0)\Delta^{(k)}.$$
Franz and Gohm show how to extend this viewpoint to the quantum case.
Define $z_i:F(\G)\raw F(\G^{k+1})$ by
$$a\mapsto  \ind_{\G}\otimes\cdots\otimes \ind_{\G}\otimes a\otimes\ind_{\G}\otimes\cdots\otimes \ind_{\G},$$
where $a$ is inserted in the $i$th copy from the right; e.g.
$$z_2(a)=\ind_{\G}\otimes\cdots\otimes\ind_{\G}\otimes a\otimes\ind_{\G}\otimes\ind_{\G}.$$
Now note
\begin{align*}
(z_k\otimes\cdots\otimes z_0)\Delta^{(k)}(a)&=(z_k\otimes \cdots\otimes z_0)\sum a_{(1)}\otimes\cdots\otimes a_{(k+1)}
\\&=\sum z_k(a_{(1)})\otimes\cdots \otimes z_0(a_{k+1})
\\&=\sum (a_{(1)}\otimes\ind_{\G}\otimes\cdots\otimes\ind_{\G})\cdots(\ind_{\G}\otimes\cdots\otimes \ind_{\G}\otimes a_{(k+1)})
\\&=\sum a_{(1)}\otimes\cdots\otimes a_{(k+1)}
\\&=\Delta^{(k)}(a)=j_k(a).
\end{align*}
Using the natural embedding,
$$z_i:F(G)\raw F(G)\underset{i\text{th entry from the right}}{\hookrightarrow} F(G^{k+1}),$$
the classical $z_i(f)=f\circ \delta^{\zeta_i}$ fits into this framework:
$$z_i(f)=(\ind_G\otimes\cdots\otimes\ind_{\G}\otimes f\otimes \ind_G \otimes\cdots\otimes \ind_G)(\eps\otimes\cdots\otimes \eps\otimes \delta^{\zeta_i}\otimes \eps\otimes\cdots\otimes\eps)\cong f\circ \delta^{\zeta_i}.$$

 To understand how studying the random variables $\{\Delta^{(k)}\}$ --- or rather the algebra of functions $F(G)$ --- gives  an insight into the random variables $\{\xi_k\}$ --- or rather a random walk on a group --- consider a random walk on a finite group $G$. The random variables $\{\xi_i\}$ are just a sequence of points in $G$:
$$\xi_0,\xi_1,\xi_2,\xi_3,\dots$$
The quantisation regime above says that if one takes an element
$$f=\sum_{t\in G}\alpha_t\delta_t\in F(G),$$ and apply it at each transition of the random walk (ignore wave function collapse and other quantum mechanical concerns), then the random variables $\{j_i\}$ can use a function --- in this case $f$ --- to \emph{measure} the states of the random walk and the sequence $\{j_i(f)\}$ can be considered:
$$\alpha_{\xi_0},\alpha_{\xi_{1}},\alpha_{\xi_2},\alpha_{\xi_3}\dots$$
Of course, if the distribution of the $\{\xi_k\}$ converges as $k\raw\infty$ --- say for example to the uniform distribution on $G$ --- then the distribution of the $\{j_{k}(f)\}$ also converges --- to the average of $f$:
$$\overline{f}=\frac{1}{|G|}\sum_{i\in G}\alpha_i.$$
Of course this is nothing but $\dsp \int_G f$.

\bigskip

\bigskip

Of course, there is no need for the algebra $F(G)$ to be the algebra of functions on a classical group $G$: instead given the algebra of functions on a quantum group $\G$, the random variables $\{\Delta^{(k)}\}$ can be studied.

\bigskip

Let $\nu$ and $\mu\in M_p(G)$. The
\textit{convolution} of $\nu$ and $\mu$ is the probability \beq
\nu\star \mu(\delta_s):=\sum_{t\in G}\nu(\delta_{st^{-1}})\mu(\delta_t). \enq The distribution of a random walk after one transition is given by $\nu$. If $s\in G$, then the walk can go to $s$ in two transitions by going to \emph{some} $t\in G$ after one transition and going from there to $s$ in the next. The probability of going from $t$ to $s$ is given by the probability of choosing $st^{-1}$, i.e. $\nu(\delta_{st^{-1}})$. By summing over all intermediate transitions $t\in G$, and noting that $\nu\star\delta^e=\nu$, it is seen that if $\{\xi_i\}_{i=0}^k$ is a
random walk on $G$ driven by $\nu$, then $\nu^{\star k}$ --- defined inductively --- is the
probability distribution of $\xi_k$. In terms of the stochastic operator $P$ induced by $\nu\in M_p(G)$ --- $p(t,s)=\nu(\delta_{st^{-1}})$ --- given any $\mu\in M_p(G)$, $P^T\mu=\nu\star \mu$.

\bigskip

To study distributions, probability theory must be quantised --- probabilities, conditional expectations, independence, etc. Let $\G$ be a finite quantum group. As noted previously, the quantisation of probability measures on a finite classical group, $G$, are states on $F(\G)$, denoted by $M_p(\G)$. Form the  tensor product
$$F(\G^{k+1}):=\underbrace{F(\G)\otimes\cdots\otimes F(\G)}_{k+1\text{ copies}}.$$
Now consider probabilities, $\psi$, $\nu\in M_p(\G)$ and form product states:
\begin{align*}
\Psi_k=&\bigotimes_{i=1}^k\nu\otimes\psi
\end{align*}
With care, an infinite tensor product, $F(\G^\infty)$, and infinite product state, $\Psi_\infty$, can be defined. For the purposes of this work, everything can be studied in `finite time' and so these constructions are not included.

\begin{example}
For $n\geq 5$, consider $f\in F(\Z_n)$ given by
$$f=\sum_{i=0}^{n-1} \alpha_i\delta_i.$$
Note that the comultiplication $\Delta:F(\Z_n)\raw F(\Z_n)\otimes F(\Z_n)$ is given by
$$\Delta(\delta_{i})=\sum_{j=0}^{n-1}\delta_{ij^{-1}}\otimes \delta_{j}=\sum_{j=0}^{n-1}\delta_{i-j}\otimes \delta_{j}$$
so that
$$\Delta (f)=\sum_{i=0}^{n-1}\alpha_i\left(\sum_{j=0}^{n-1}\delta_{i-j}\otimes \delta_{j}\right)=\sum_{i,j=0}^{n-1}\alpha_i\delta_{i-j}\otimes\delta_{j}.$$
Now $j_0=I_{F(\Z_n)}$, $j_1=\Delta$ and $j_2=\Delta^{(2)}$ so
$$j_2(f)=\sum_{i,j,k=0}^{n-1}\alpha_i\delta_{i-j-k}\otimes \delta_{k}\otimes\delta_{j}.$$
Now suppose that the initial state is given by $\varepsilon=\delta^0$ and the transition state is given by $\nu=(\delta^{1}+\delta^{-1})/2$. Consider
\begin{align*}
\Psi_0(j_0)(f)&=\sum_{i=0}^{n-1}\alpha_i\varepsilon(\delta_{i})=\alpha_0.
\\ \Psi_1(j_1)(f)&=\sum_{i,j=0}^{n-1}\alpha_i\nu(\delta_{i-j})\eps(\delta_{j})
\\ &=\sum_{j=0}^{n-1}\alpha_i\phi(\delta_{i})
\\ & =\frac{1}{2}\alpha_1+\frac{1}{2}\alpha_{-1}.
\\ \Psi_2(j_2)(f)&=\sum_{i,j,k=0}^{n-1}\alpha_i\nu(\delta_{i-j-k})\nu(\delta_{k})\eps(\delta_{j})
\\ & =\sum_{i,k=0}^{n-1}\alpha_i\nu(\delta_{i-k})\nu(\delta_k)
\\ &=\frac12 \sum_{i=0}^{n-1}\alpha_i\nu(\delta_{i-1})+\frac12 \sum_{i=0}^{n-1}\alpha_i\nu(\delta_{i+1})
\\ &=\frac12\left(\frac12 \alpha_0+\frac12 \alpha_{-2}\right)+\frac12\left(\frac12 \alpha_2+\frac12 \alpha_0\right)
\\ & =\frac12 \alpha_0+\frac14 \alpha_2+\frac14 \alpha_{-2}
\end{align*}

Therefore the $\{j_k\}$ can be thought of as quantum random variables with distributions $\Psi_k\circ j_k$, and $\{j_k\}_{k\geq0}$ as a quantum stochastic process. Call $\psi$ the \emph{initial state} and $\nu$ the \emph{transition state}.
\end{example}

\bigskip

This can be considered in more generality. A series of calculations, for a general $f=\sum_t\alpha_{t}\delta_t\in F(G)$, leads  towards:
$$j_{k}(f)=\sum_{t_k,t_{k-1},\dots,t_1,t_0\in G}\alpha_{t_k}\delta_{t_kt_{k-1}^{-1}}\otimes\delta_{t_{k-1}t_{k-2}^{-1}}\otimes \cdots\otimes \delta_{t_1t_0^{-1}}\otimes \delta_{t_0}.$$
If the initial distribution is given by the counit, then looking at
\begin{align*}
\Psi_k(j_{k}(f))&=\sum_{t_k,\dots,t_0}\alpha_{t_k}\prod_{i=1}^k\nu(\delta_{t_it_{i-1}^{-1}})\eps(\delta_{t_0})
\\&=\sum_{t_k,\dots,t_0}\alpha_{t_k}\mathbb{P}[\xi_k=t_k]
\\&=\sum_{t_k,\dots,t_0}f(t_k)\mathbb{P}[\xi_k=t_k],
\end{align*}
these calculations yield an expectation so that $\Psi_k\circ j_k=\mathbb{E}_{\Psi_k}$.

\bigskip

\begin{definition}
Let $\G$ be a finite quantum group. If there exists a $\nu\in M_p(\G)$ such that the distribution of the random variables
$$j_k=\Delta^{(k)}:F(\G)\rightarrow \bigotimes_{k+1\text{ copies}}F(\G)$$
are given by
$$\Psi_k=\left(\bigotimes_{k\text{ copies }}\nu\right)\otimes\eps,$$
then the family $\{j_i\}_{i=0}^k$ is called  the \emph{right-invariant random walk on $\G$ driven by $\nu$.}
\end{definition}

\begin{example}
\emph{(Card Shuffling)}
Card shuffling provides a motivation  for the study of random walks on groups and remains a key example. Everyday shuffles such as the overhand shuffle or the riffle shuffle, as well as simpler but more tractable examples such as top-to-random or random transpositions all have the structure of a random walk on $S_{52}$. Each shuffle may be realised as sampling from a probability distribution $\nu\in M_p(S_{52})$. Let $\sigma\in S_{52}$ be any arrangement of the deck:
$$j_{k}(\delta_\sigma)=\sum_{\underset{\sigma_k\cdots \sigma_0=\sigma}{\sigma_i\in S_{52}}}\delta_{\sigma_k}\otimes\delta_{\sigma_1}\otimes\cdots \otimes\delta_{\sigma_0}.$$

\bigskip

 For example, consider the case of repeated random transpositions. A \emph{random transposition} consists chooses two cards at random (with replacement) from the deck and swapping the positions of these two cards. Suppose without loss of generality that the first card chosen is the ace of spades. The probability of choosing the ace of spaces again is 1/52. Swapping the ace the spades with itself leaves the deck unchanged. The choice of the first card is independent hence the probability that the shuffle leaves the deck unchanged is 1/52. What is the probability of transposing two given (distinct) cards?  Consider, again without loss of generality, the probability of transposing  the ace of spades and the ace of hearts. There are two ways this may be achieved: choose $A\spadesuit$-$A\heartsuit$ or choose $A\heartsuit$-$A\spadesuit$. Both of these have probability of 1/52$^2$. Any other given shuffle (not leaving the deck unchanged or transposing two cards) is impossible. Hence repeated shuffles may be modelled as repeatedly sampling by
\beqa
\nu(\delta_s):=\begin{cases}
1/52 & \text{ if }s=e,
\\ 2/52^2 & \text{ if $s$ is a transposition,}
\\ 0 & \text{ otherwise.}
\end{cases}
\enqa
If $f$ is any real-valued function on $S_{52}$, then the distribution of $f$ applied after $k$ transitions is given by $\Psi_k\circ j_k(f)$. For example, consider the function $A\spadesuit:S_{52}\raw \R$, which, if a starting order is specified with the $A\spadesuit$ on the bottom of the deck, is given by
$$A\spadesuit(\sigma)=\sigma(52).$$
The distribution of the position of the $A\spadesuit$ is given by $\Psi_k\circ j_k(A\spadesuit)$.
\end{example}

\bigskip

\begin{example}(Random Walks on the Dual Group $\widehat{G}$)
Let $G$ be a finite group. The dual group $\widehat{G}$, a virtual object when $G$ is non-abelian, is defined by $\C G=:F(\widehat{G})$. Let $\mu\in F(\widehat{G})$ be given by
$$\mu=\sum_{t\in G}a_t\delta^t.$$
The comultiplication is given by $\Delta(\delta^s)=\delta^s\otimes\delta^s$ and so
$$j_k(\mu)=\sum_{t\in G}a_t\left(\bigotimes_{k+1\text{ copies}}\delta^t\right).$$
The counit on $F(\widehat{G})$ is given by $\eps_{\C G}=\ind_{G}$. Using this, and supposing that the random walk is driven by $\nu\in M_p(\widehat{G})$, then
$$\Psi_k(j_k)(\mu)=\sum_{t\in G}a_t\nu(\delta^t)^k.$$
The $\mathrm{C}^*$-algebra $F(\widehat{G})$  is unital with unit $\ind_{\widehat{G}}=\delta^e$ and therefore (Murphy \citep{Murphy} Corollary 3.3.4) $\nu(\ind_{\widehat{G}})=1$ and thus $a_e=1$ and so for $\dsp \Psi_k\raw \delta_e$ it is necessary that $|\nu(\delta^s)|<1$ for all $s\in G\bs \{e\}$.
\end{example}

\bigskip

\subsection*{Stochastic Operators}
At the end of Section \ref{MC}, the stochastic operator approach to quantisation was abandoned in favour of the random variable approach. Given a random walk on a quantum group, it is straightforward to write down the associated stochastic operator. Let $\{j_i\}_{i=0}^k$ be a random walk on a finite quantum group $\G$ driven by $\nu\in M_p(\G)$. The distribution of $j_k$ is given by $\Psi_k$. Consider in particular the distribution of $j_1$:
$$\Psi_1(j_1)=(\nu\otimes \eps)\Delta=\nu\star \eps=\nu,$$
as $\eps$ is the unit for the convolution algebra $\C \G$:
$$\nu\star \eps=(\nu\otimes\eps)\Delta=(\nu\otimes I_{\C})(I_{F(\G)}\otimes\eps )\Delta=(\nu\otimes I_{\C})I_{F(\G)}\cong \nu.$$
Also
\begin{align*}
\Psi_2(j_2)&=(\nu\otimes\nu\otimes\eps)\Delta^{(2)}
\\&=(\nu\otimes\nu\otimes I_{\C}) (I_{F(\G)}\otimes I_{F(\G)}\otimes\eps)(I_{F(\G)}\otimes \Delta)\Delta
\\&\cong(\nu\otimes\nu) (I_{F(\G)}\otimes I_{F(\G)})\Delta
\\&=(\nu\otimes\nu)\Delta=\nu^{\star 2}.
\end{align*}
Similarly it can be shown that
$$\Psi_k(j_k)=\nu^{\star k}.$$

\bigskip

For $\nu\in M_p(\G)$, define $P_\nu \in L(F(\G))$
$$P_\nu=(\nu\otimes I_{F(\G)})\Delta.$$
\newpage
\begin{proposition}\label{propertiesofP} Let $\G$ be a finite quantum group and $\nu\in M_p(\G)$. Then the following hold:
\begin{enumerate}
\item[i.] $P_\nu^T\mu=\nu\star\mu$. Hence, in particular, $P_{\nu^{\star k}}=P_\nu^k$.
\item[ii.] $P_\nu$ is unital and positive.
\item[iii.] $M_p(\G)$ is stable under $P_\nu^T$.
\item[iv.] The map $\nu\raw P_{\nu}^T$ is an algebra homomorphism from $\C \G$ to $L(\C\G)$.
\item[v.] $\dsp P_\nu^T\int_{\G}=\int_{\G}$.
\item[vi.] $P_{\F(a)}b=S(a)\star_A b$ for all $b\in F(\G)=A$.
\end{enumerate}
\begin{proof}
\begin{enumerate}
\item[i.] Let $f\in F(\G)$:
\begin{align*}
P_\nu^T\mu(f)&=\mu(\nu\otimes I_{F(\G)})\Delta(f)
\\&=\mu\left(\sum \nu(f_{(1)})\otimes f_{(2)}\right)
\\&\cong \mu\left(\sum \nu(f_{(1)})f_{(2)}\right)
\\&=\sum \nu(f_{(1)})\mu(f_{(2)})
\\&=(\nu\star\mu)(f).
\end{align*}
\item[ii.] Note that $\Delta$ is unital and $\nu(\mathds{1}_{\G})=1$:
\begin{align*}
P_\nu(\mathds{1}_{\G})&=(\nu\otimes I_{F(\G)})\Delta(\ind_{\G})
\\&=(\nu\otimes I_{F(\G)})(\ind_{\G}\otimes \ind_{\G})
\\&=\nu(\ind_{\G})\ind_{\G}=\ind_{\G}.
\end{align*}
Note that as $\Delta$ is a *-homomorphism --- and $F(\G)^+$ is a convex cone:
\begin{align*}
P_\nu(f^*f)&=(\nu\otimes I_{F(\G)})\Delta(f^*f)
\\&=(\nu\otimes I_{F(\G)})\Delta(f)^*\Delta(f)
\\&=(\nu\otimes I_{F(\G)})\sum f_{(1)}^*f_{(1)}\otimes f_{(2)}^*f_{(2)}
\\&=\sum \nu\left(f_{(1)}^*f_{(1)}\right)f_{(2)}^*f_{(2)}\in F(\G)^+,
\end{align*}
Alternatively note that $P_{\nu}$ is positive as the composition of positive maps.
\item[iii.] This follows from the fact that $M_p(\G)$ is closed under convolution --- a consequence of $\Delta$ being a unital --- and i.
\item[iv.] Let $\varphi\in \C \G$. From i.:
\begin{align*}
P_\nu^TP^T_{\mu}\varphi&=P_\nu^T(\mu\star\varphi)
\\&=\nu\star\mu\star\varphi
\\&=P_{\nu\star \mu}^T\varphi.
\end{align*}
\item[v.] This follows from the fact that $\dsp \varphi\star \int_{\G}=\int_{\G}$ for all $\varphi\in\C\G$ and i.
\item[vi.]  Note that
 \begin{align*}
P_{\F(a)}b&=(\F(a)\otimes I_{F(\G)})\Delta(b)
\\&=(\mathcal{F}(a)\otimes I_{F(\G)})\sum b_{(1)}\otimes b_{(2)}
\\&=\sum b_{(2)}\int_{\G}b_{(1)}a
\\&=\sum b_{(2)}\int_{\G}S(b_{(1)})S(a),
\end{align*}
via $\dsp \int_{\G}\circ S=\int_{\G}$ (Theorem 2.2.6, \citep{Timm}) and the traciality of the Haar measure (Theorem \ref{propofHaar}, this work). Looking at (\ref{conv2}), note this is nothing other than $S(a)\star_A b$\qquad$\bullet$
\end{enumerate}
\end{proof}
\end{proposition}

\bigskip

Amongst other results, this shows that
$$(P_\nu^T)^k\eps=\nu^{\star k}\qquad\text{and}\qquad P_\nu^k=(\nu^{\star k}\otimes I_{F(\G)})\Delta,$$
so that, as Franz and Gohm state \citep{franzgohm}, the semigroup of stochastic operators $\{P_{\nu}^k\}$ and the semigroup $\{\nu^{\star k}\}$ of convolution powers of the driving probability are essentially the same thing. For connections to quantum mechanics see Majid \citep{Majid1,Majid3}.

\chapter{Distance to Random\label{dist}}
\wxn

\section{Introduction}
In the classical case, under mild conditions a random
walk on a group converges to the uniform distribution. Therefore, initially
the walk is `far' from random and eventually the walk is `close' to random.
An appropriate question therefore is, given $\eps > 0$, how large should
$k$ be so that the walk is $\eps$-close to random after $k$ transitions? The first problem
here is to have a measure of `close to random'. This chapter introduces
a measure of `closeness to random' for measures on a finite quantum group.

\bigskip

Note that in the classical case $\Psi_k(j_k(f))$ is nothing but $\nu^{\star k}(f)$. In the ergodic case, $\nu^{\star k}$ converges to the uniform distribution $\pi$ and so the elements of the vector $\nu^{\star k}$ all converge to $1/|G|$ so that we have
$$\Psi_k(j_k(f))\underset{\text{`}k\raw\infty\text{'}}{\longrightarrow} \sum_{t\in G}f(\delta^t)\frac{1}{|G|}=\int_G f.$$
In other words the distribution $\nu^{\star k}=\Psi_k\circ j_k$ `converges' to the Haar measure on $G$.  In the quantum case, given a random walk on a quantum group, $\G$, and an appropriately chosen driving probability $\nu\in M_p(\G)$, the distribution of the $j_k$, also given by $\nu^{\star k}$, can also `converge' to the Haar measure.

\newpage

\section{Measures of Randomness}
The preceding remarks indicate that when $\nu^{\star k}\raw \dsp\int_{\G}$  a measure of closeness to random can be defined by defining a metric on $M_p(\G)$ or putting a norm on $\C\G\supseteq M_p(\G)$. Then a precise mathematical question may be asked: given $\eps>0$, how large should $k$ be so that $d\left(\nu^{\star k},\dsp\int_{\G}\right)<\eps$ or  $\left\|\nu^{\star k}-\dsp\int_{\G}\right\|<\eps$? In mirroring the classical notation, also denote the Haar measure on $\G$ by $\pi:=\dsp\int_{\G}$ and refer to it as the \emph{random distribution}.

\bigskip

In the classical case, the norm used is the \emph{total variation distance} and, for $\nu,\,\mu\in M_p(G)$, it comes in three equivalent guises:
$$\|\nu-\mu\|_{\text{TV}}=\sup_{S\subset G}|\nu(S)-\mu(S)|=\frac{1}{2}\sup_{\|f\|_{\infty}\leq 1}|\nu(f)-\mu(f)|=\frac12 \|\nu-\mu\|_{1}.$$
Although the first `$\sup_{S\subset G}$' is popular among the classical theorists, a na\"{i}ve translation/quantisation, `$\sup_{S\subset \G}$', needs work and indeed it is not immediately obvious how to define a quantum total variation distance.

\bigskip

It will be seen, however, that the $\sup_{S\subset G}$ presentation can be salvaged as follows. Consider a `subset' $\mathbb{B}\subset \G$ given by a subspace $F(\mathbb{B})\subset F(\G)$. It will be seen that, where the `indicator function' on $\mathbb{B}$ given by $\ind_{\mathbb{B}}:=p_{F(\mathbb{B})}$ --- the projection onto $F(\mathbb{B})$ --- that $\phi=2\ind_{\mathbb{B}}-\ind_{\G}$ is a suitable test function as $\|\phi\|_{F(\G)}^\infty=1$. That is for any subspace $F(\mathbb{B})\subset F(\G)$
\begin{align*}
\|\nu-\mu\|&\geq \frac12 |\nu(2\ind_{\mathbb{B}}-\ind_{\G})-\mu(2\ind_{\mathbb{B}}-\ind_{\G})|
\\&=\frac12|2\nu(\ind_{\mathbb{B}})-1-2\mu(\ind_{\mathbb{B}})+1|
\\&= |\nu(\ind_{\mathbb{B}})-\mu(\ind_{\mathbb{B}})|.
\end{align*}
An interesting question is, for a given $\nu\in M_p(\G)$, does there exist a $\mathbb{B}\subset \G$ such that
$$\|\nu-\mu\|=|\nu(\ind_{\mathbb{B}})-\mu(\ind_{\mathbb{B}})|\,\text{?}$$
The answer is yes in the classical case. Take $F(B)=\langle \delta_s\in G:\nu(\delta_s)>\mu(\delta_s)\rangle$. An answer in the quantum case is not given in this work.

 \newpage
 To measure $\nu^{\star k}-\pi$, three features that such a norm must have include
 \begin{enumerate}
 \item Agreement in the classical case:
 $$\|\mu\|_{\text{QTV}}=\|\mu\|_{\text{TV}}.$$
 \item A Cauchy--Schwarz-type inequality:
$$\|\mu\|_{\text{QTV}}\lessapprox \|\mu\|_2,$$
as the Diaconis--Shahshahani theory generates upper bounds for $\|\nu^{\star k}-\pi\|_2$.
\item  A presentation as a supremum
$$\|\mu\|_\text{QTV}=\sup_{s\in S}F(s,\mu).$$
This allows for the generation of lower bounds via `test elements' $s_0\in S$:
$$\|\nu-\pi\|\geq F(s_0,\nu-\pi).$$
 \end{enumerate}
A closer analysis of the classical case reveals the correct norm to use. On the one hand the conclusion is unsatisfactory because the quantum total variation distance is a norm on functions $F(\G)$ rather than on probability measures $M_p(\G)$. On the other it satisfies all of the three conditions and in particular is identical to the second `guise':
$$\|\nu-\pi\|_{\text{QTV}}=\frac{1}{2}\sup_{\phi\in F(G)\,:\,\|\phi\|_{\infty}\leq 1}|\nu(\phi)-\pi(\phi)|,$$
although identifying when $\|\phi\|_\infty\leq 1$ may be a non-trivial task.

\bigskip

Let $V$ be a finite-dimensional vector space and denote by $\|\cdot\|_{(\mathcal{L}^p,\mathcal{B})}$ the $p$-norm with respect to the basis $\mathcal{B}:=\{e_i\}_{i=1}^n$:
$$v=\sum_{i=1}^na_ie_i\Rightarrow \|v\|_{(\mathcal{L}^p,\mathcal{B})}=\left(\sum_{i=1}^n|a_i|^p\right)^{1/p}.$$
For example, if $\mathcal{B}=\{\delta_t\,:\,t\in G\}\subset F(\Z_3)$ is the standard basis then
$$\|3\delta_0+4\delta_1\|_{(\mathcal{L}^2,\mathcal{B})}=5.$$
In the richer category of von Neumann algebras with a normal, faithful trace $\tau$, for each $a\in A$ a  von Neumann algebra and $1\leq p<\infty$,
$$\|a\|_p^A=\left(\tau|a|^p\right)^{1/p}$$
defines a norm on $A$ \citep{PX03}. Set the infinity norm equal to the operator norm:
$$\|a\|_\infty^A=\|a\|.$$
In the case of a classical $A=F(G)$, with the standard basis $\mathcal{B}$, and with the normal, faithful trace given by the Haar measure;
\begin{align*}
\|f\|_1^{F(G)}&=\int_G |f|=\frac{1}{|G|}\sum_{t\in G}\delta^t(f^*f)^{1/2}=\frac{1}{|G|}\|f\|_{(\mathcal{L}^1,\mathcal{B})}\text{ and}
\\ \|f\|_2^{F(G)}&=\left(\int_G f^*f\right)^{1/2}=\left(\frac{1}{|G|}\sum_{t\in G}\delta^t(f^*f)\right)^{1/2}=\frac{1}{\sqrt{|G|}}\|f\|_{(\mathcal{L}^2,\mathcal{B})}
\end{align*}
Sections 2.A, 2.B, 2.C and especially 3.B of Diaconis \citep{PD} (covered in Sections 2.1, 2.2 and 3.2 of the MSc thesis \citep{MSc}) involve a blurring of the lines between elements of $F(G)$ and elements of $\C G$. Consider the vector space $\C G$ with basis $\widehat{\mathcal{B}}=\{\delta^t\,:\,t\in G\}$. The classical total variation norm is equal to
$$\|\nu-\pi\|_{\text{TV}}=\frac12 \|\nu-\pi\|_{\left(\mathcal{L}^1,\widehat{\mathcal{B}}\right)}.$$
However this $(\mathcal{L}^1,\widehat{\mathcal{B}})$-norm is not easily related to
$$\|\nu-\pi\|_1^{\C G}=\int_{\widehat{G}}\left(\left((\nu-\pi)^*\star(\nu-\pi)\right)^{1/2}\right).$$
For example, take
$$ \mu=\frac{1}{2}\delta^0+\frac13 \delta^1+\frac{1}{6}\delta^2\in M_p(\Z_3)\subset\C \Z_3.$$
While $\|\mu\|_{\left(\mathcal{L}^1,\widehat{\mathcal{B}}\right)}=1$,
\begin{align*}
\|\mu\|_1^{\C \Z_3}&=\int_{\widehat{\Z_3}}(\mu^*\star\mu)^{1/2}
\\&=\int_{\widehat{\Z_3}}\left(\frac{7}{18}\delta^0+\frac{11}{36}\delta^1+\frac{11}{36}\delta^2\right)^{1/2},
\end{align*}
where the fact that
$$\left(\sum_{t\in G}\alpha_t\delta^t\right)^*=\sum_{t\in G}\overline{\alpha_t}\delta^{t^{-1}}$$
was used. There is a potential confusion now because while all elements of $M_p(G)$ are positive functionals not all of them are positive in the $\mathrm{C}^*$-algebra $\C G$. For $\varphi=\sum_{t\in G}\alpha_t\delta^t\in\C G$ to be positive in the $\mathrm{C}^*$-algebra $\C G$ there must a $\phi=\sum_{t}\beta_t\delta^t\in\C G$ such that $\varphi=\phi^*\star \phi$. That is there must exist complex constants $\{\beta_t\,:\,t\in G\}$ such that
$$\varphi=\sum_{s\in G}\left(\sum_{t\in G} \overline{\beta_t}\beta_{ts}\right)\delta^s.$$
In particular, $\mu$ given above is not positive. Note that
$$\varphi=\left(\frac{1}{3}+\frac{1}{9}\sqrt{3}\right)\delta^0+\left(\frac{1}{3}-\frac{1}{18}\sqrt{3}\right)\delta^1+\left(\frac{1}{3}-\frac{1}{18}\sqrt{3}\right)\delta^2$$
is a square root of $\mu^*\star\mu$. Also $\varphi$ is positive because it is equal to $\phi^*\star \phi$ where
$$\phi=\left(\frac{1}{2}-\frac{1}{6}i\sqrt{3+2\sqrt{3}}\right)\delta^0+\left(\frac{1}{2}+\frac{1}{6}i\sqrt{3+2\sqrt{3}}\right)\delta^1.$$

\bigskip

The positive $\varphi=|\mu|$ can be found by concretely realising $\C\Z_3$ via ($\omega=e^{2\pi i/3}$)
$$\delta^0=\left(\begin{array}{ccc} 1 & 0 & 0 \\ 0 & 1 & 0 \\ 0 & 0 & 1 \end{array}\right),\quad \delta^1=\left(\begin{array}{ccc} 1 & 0 & 0 \\ 0 & \omega & 0 \\ 0 & 0 & \omega^2 \end{array}\right),\quad \delta^2=\left(\begin{array}{ccc} 1 & 0 & 0 \\ 0 & \omega^2 & 0 \\ 0 & 0 & \omega \end{array}\right),$$
writing $\mu$ in this basis and finding that
$$|\mu|=\left(\begin{array}{ccc} 1 & 0 & 0 \\ 0 & \sqrt{\frac{1}{12}} & 0 \\ 0 & 0 & \sqrt{\frac{1}{12}} \end{array}\right).$$
When written in the standard basis this is positive and the same as $\varphi$ above.

\bigskip

Therefore
\begin{align*}
\|\mu\|_{1}^{\C \Z_3}=\int_{\widehat{\Z_3}}\varphi &=\eps(\mathcal{F}^{-1}(\varphi))
\\&=\eps\left(\left(1+\frac{1}{\sqrt{3}}\right)\delta_0+\left(1-\frac{1}{6}\sqrt{3}\right)\delta^1+\left(1-\frac{1}{6}\sqrt{3}\right)\delta^2\right)
\\&=1+\frac{1}{\sqrt{3}},
\end{align*}
as $$\mathcal{F}(\delta_s)=\delta^s/|G|\Leftrightarrow \mathcal{F}^{-1}(\delta^s)=|G|\,\delta_s$$ for classical groups. Therefore, with the $(\mathcal{L}^1,\widehat{\mathcal{B}})$-norm equal to one and this norm giving $1+1/\sqrt{3}$, it is clear that the `one-norm' as used by the classical theorists is not  a scalar multiple of $\|\cdot\|_{1}^{\C G}$. Furthermore, this rules out using a multiple of $\|\cdot\|_1^{\C \G}$ to define quantum total variation distance, as although such a norm does satisfy a useful Cauchy--Schwarz inequality\footnote{$\|\eps\|_2^{\C\G}=\sqrt{|\G|}$ can be shown using the Diaconis--Van Daele Inversion Theorem (later)}
$$\|\nu\|_1^{\C \G}=\|\eps\star \nu\|_1^{\C \G}\leq \|\eps\|_2^{\C \G}\|\nu\|_2^{\C \G}\leq\sqrt{|\G|}\|\nu\|_2^{\C\G},$$
and has a presentation as a supremum
$$\|\nu-\pi\|_1^{\C \G}=\sup_{\mu\in \C \G\,:\,\|\mu\|_{\infty}^{\C \G}\leq 1}\left|\int_{\widehat{\G}}\left((\nu-\pi)\star \mu\right)\right|,$$
the fact that, in the classical case,
$$\|\nu-\pi\|_1^{\C G}\neq k\cdot \|\nu-\pi\|_{\text{TV}},$$
means that this norm is not going to be the preferred option.

\bigskip

 The total variation distance \emph{is} however related to the norm $\|\cdot\|_{1}^{F(G)}$. Consider a $\nu=\sum_t \nu(\delta_t)\delta^t\in M_p(G)$ and consider $\tilde{\nu}\in F(G)$ given by
$$\tilde{\nu}=\sum_{t\in G}\nu(\delta_t)\delta_t,$$
i.e. viewing the function $\tilde{\nu}$ as the probability measure $\nu$. Where $\mathcal{B}$ is the standard basis of $F(G)$, note that \begin{align*}
\|\nu\|_{\left(\mathcal{L}^1,\widehat{\mathcal{B}}\right)}&=\|\tilde{\nu}\|_{\left(\mathcal{L}^1,\mathcal{B}\right)}
\\&=\sum_{t\in G}\delta^t|\tilde{\nu}|
\\&=|G|\int_G|\tilde{\nu}|=|G|\|\tilde{\nu}\|_1^{F(G)},
\end{align*}
Therefore, again in the classical case,
\begin{align*}
\|\nu-\pi\|_{\text{TV}}&=\frac{1}{2}\|\nu-\pi\|_{(\mathcal{L^1},\widehat{\mathcal{B}})}
\\&=\frac{|G|}{2}\|\tilde{\nu}-\tilde{\pi}\|_1^{F(G)}
\\&=\frac{1}{2}\||G|\tilde{\nu}-|G|\tilde{\pi}\|_1^{F(G)}
\\&=\frac{1}{2}\|\mathcal{F}^{-1}(\nu)-\mathcal{F}^{-1}(\pi)\|_1^{F(G)}
\\&=\frac{1}{2}\|\mathcal{F}^{-1}(\nu-\pi)\|_1^{F(G)}.
\end{align*}
In the classical case, the relationship between the four norms $\|\cdot\|_2^{\C G}$, $\|\cdot\|_{\left(\mathcal{L}^2,\widehat{\mathcal{B}}\right)}$, $\|\cdot\|_2^{F(G)}$ and $\|\cdot\|_{\left(\mathcal{L}^2,\mathcal{B}\right)}$ can be examined. Indeed considering
$$\nu=\sum_{t\in G}\alpha_t\delta^t\in \C G$$
and then
$$\tilde{\nu}=\sum_{t\in G}\alpha_t \delta_t=\mathcal{F}^{-1}(\nu/ |G|)\in F(G),$$
and defining
$$\|\nu\|_{\mathcal{L}^2}=\|\nu\|_{\left(\mathcal{L}^2,\widehat{\mathcal{B}}\right)}=\|\tilde{\nu}\|_{\left(\mathcal{L}^2,\mathcal{B}\right)},$$
\begin{align*}
\left(\|\nu\|_2^{\C G}\right)^2&=\int_{\widehat{G}}(\nu^*\star\nu)
\\&=\eps\circ \left(\mathcal{F}^{-1}\left(\sum_{s\in G}\left(\sum_{t\in G} \overline{\alpha_t}\alpha_{ts}\right)\delta^s\right)\right)
\\ &=|G|\eps\left(\sum_{s\in G}\left(\sum_{t\in G} \overline{\alpha_t}\alpha_{ts}\right)\delta_s\right)
\\&=|G|\sum_{t\in G}\overline{\alpha_t}\alpha_t
\\&=|G|\sum_{t\in G}|\alpha_t|^2
\\&=|G|\|\nu\|_{\mathcal{L}^2}^2.
\end{align*}
Consider also
\begin{align*}
\left(\|\tilde{\nu}\|_2^{F(G)}\right)^2&=\int_G\tilde{\nu}^*\tilde{\nu}=\int_G|\tilde{\nu}|^2
\\&=\frac{1}{|G|}\sum_{t\in G}\delta^t|\tilde{\nu}|^2
\\&=\frac{1}{|G|}\sum_{t\in G}|\alpha_t|^2=\frac{1}{|G|}\|\nu\|_{\mathcal{L}^2}^2
\\ \Raw \|\nu\|_2^{\C G}&=|G|\|\tilde{\nu}\|_2^{F(G)}=\|\mathcal{F}^{-1}(\nu)\|_2^{F(G)}.
\end{align*}
This is nothing but a Plancherel theorem and as has been seen in Section \ref{quantumgroupring}, the Plancherel theorem is also true in the quantum setting:
$$\|a\|_2^{F(\G)}=\|\mathcal{F}(a)\|_2^{\C \G}.$$
In fact, analysing the classical case in more detail it is clear that
$$\|\nu-\pi\|_{\text{TV}}=\frac{1}{2}\|\mathcal{F}^{-1}(\nu)-\mathcal{F}^{-1}(\pi)\|_1^{F(G)}$$
is equivalent to the definition used by the classical theorists. Indeed the full chain of inequalities and equalities is:
\begin{align*}
\|\nu-\pi\|_{\text{TV}}&:=\frac{1}{2}\|\mathcal{F}^{-1}(\nu)-\mathcal{F}^{-1}(\pi)\|_1^{F(G)}
\\&=\frac{1}{2}\|\mathds{1}_{G}\mathcal{F}^{-1}(\nu-\pi)\|_1^{F(G)}
\\&\underset{\text{C--S}}{\leq }\frac{1}{2}\|\mathds{1}_{G}\|_2^{F(G)}\cdot \|\mathcal{F}^{-1}(\nu-\pi)\|_2^{F(G)}
\\&\underset{\text{Planch.}}{=}\frac{1}{2}\|\nu-\pi\|_2^{\C G},
\end{align*}
and then representation theory is used to write down an explicit formula for $\|\nu-\pi\|_2^{\C G}$.

Consider again the norm $\|\cdot\|_p^A$. Noting that $\C \G$ and $F(\G)$ are indeed von Neumann algebras, with the trace given by their Haar measure, the following result may be used.

\begin{theorem}
(Properties of von Neumann $p$-norms\label{props})
Let $A$ be a finite von Neumann algebra equipped with a normal, faithful trace $\tau$. Denote the completion of $(A,\|\cdot\|_p^{A})$ by $\mathcal{L}^p(A)$. If $x,\,y\in\mathcal{L}^2(A)$ then $xy\in\mathcal{L}^1(A)$ and the Cauchy--Schwarz inequality holds:
$$\|xy\|^{A}_1\leq \|x\|^{A}_2\|y\|^{A}_2.$$
Furthermore, if $\|\cdot\|^{A}_{\infty}$ is defined as the operator norm of the von Neumann algebra, then the following supremum-presentations hold:
$$\|b\|^{A}_{\infty}:=\sup_{\|a\|_1^{A}\leq 1}|\tau(ba)|,$$
$$\|a\|_1^{A}=\sup_{\|b\|_{\infty}^{A}\leq1}|\tau(ab)|.$$
\end{theorem}

\begin{proof}
Standard results of non-commutative $\mathcal{L}^p$-spaces: see \citep{TAK02} and \citep{PX03} \,\,\,$\bullet$
\end{proof}

\bigskip

Define for $\nu\in \C G$ and $\pi$ the Haar measure on $\G$
$$\|\nu-\pi\|_{\text{QTV}}:=\frac{1}{2}\|\mathcal{F}^{-1}(\nu-\pi)\|_1^{F(\G)}.$$
This definition satisfies the three properties that a quantum total variation distance must have.

\bigskip

Earlier calculations show, in the classical case with $\nu\in F(G)$, that:
$$\|\nu-\pi\|_{QTV}=\|\nu-\pi\|_{\text{TV}}.$$
Secondly, using the Cauchy--Schwarz inequality for $\|\cdot\|_1^{F(\G)}$ and the Plancherel Theorem \ref{Planch} note
\begin{align*}
\|\nu-\pi\|_{\text{QTV}}&=\frac{1}{2}\|\mathcal{F}^{-1}(\nu-\pi)\|_1^{F(\G)}
\\&=\frac{1}{2}\|\mathds{1}_{\G}\mathcal{F}^{-1}(\nu-\pi)\|_1^{F(\G)}
\\&\underset{\text{C--S}}{\leq }\frac12 \|\mathds{1}_{\G}\|_2^{F(\G)}\|\mathcal{F}^{-1}(\nu-\pi)\|_2^{F(\G)}
\\ &\underset{\text{Planch.}}{=}\frac12 \|\nu-\pi\|_2^{\C\G}.
\end{align*}
Finally,
\begin{align}
\|\nu-\pi\|_{\text{QTV}}&=\frac{1}{2}\|\mathcal{F}^{-1}(\nu-\pi)\|_1^{F(\G)}\nonumber
\\ &\underset{\ref{props}}{=}\frac12 \sup_{\phi\in F(\G)\,:\,\|\phi\|^{F(\G)}_{\infty}\leq 1}\left|\int_{\G}\left(\left(\mathcal{F}^{-1}(\nu-\pi)\right)\phi\right)\right|\nonumber
\end{align}
As the Haar measure is tracial (Theorem \ref{propofHaar})
\begin{align}
\Raw \|\nu-\pi\|_{\text{QTV}}&=\frac12 \sup_{\phi\in F(\G)\,:\,\|\phi\|^{F(\G)}_{\infty}\leq 1}\left|\mathcal{F}\left(\mathcal{F}^{-1}(\nu-\pi)\right)\phi\right|\nonumber
\\&=\frac12 \sup_{\phi\in F(\G)\,:\,\|\phi\|^{F(\G)}_{\infty}\leq 1}\left|(\nu-\pi)\phi\right|\nonumber
\\&=\frac12 \sup_{\phi\in F(\G)\,:\,\|\phi\|^{F(\G)}_{\infty}\leq 1}\left|\nu(\phi)-\pi(\phi)\right|.\label{lb}
\end{align}
In particular, if $\phi$ has zero expectation under the Haar measure,
$$\|\nu-\pi\|_{\text{QTV}}\geq \frac{1}{2}|\nu(\phi)|.$$
There is potentially a problem in easily identifying when $\|\phi\|_{\infty}^{F(\G)}\leq 1$. Later it will be seen that the matrix elements of one dimensional representations are particularly nice for generating lower bounds.

\bigskip

Altogether then, $\|\cdot\|_{\text{QTV}}$ satisfies all the desirable properties of a quantum total variation distance and from now on it will simply be denoted by $\|\cdot\|$ and called the total variation distance.

\bigskip

In the classical case, standard results about the norms of matrices can be used to show that the total variation distance is decreasing in $k$. In the truly quantum case things are not as straightforward. Using $\mathrm{C}^*$-algebraic machinery, it is not difficult to show that a \emph{quantum separation distance} is decreasing in $k$.

\bigskip

Define a norm on $\C\G$ by $\|\mu\|_{\infty}=\|\mathcal{F}^{-1}(\mu)\|_\infty^{F(\G)}$. Recall that in the classical case, commutativity of $F(G)$ means that $\|\cdot\|_\infty^{F(G)}$, the operator norm, is nothing but the supremum norm. Let $\nu=\sum_{t\in G}\nu(\delta_t)\delta^t\in\C G$ and consider:
\begin{align*}
\|\nu-\pi\|_{\infty}&=\|\mathcal{F}^{-1}(\nu-\pi)\|_{\infty}^{F(G)}
\\&=\left\||G|\left(\sum_{t\in G}\nu(\delta_t)\delta_t-\frac{1}{|G|}\ind_{G}\right)\right\|_{\infty}^{F(G)}
\\&=|G|\max_{t\in G}\left|\frac{1}{|G|}-\nu(\delta_t)\right|.
\end{align*}
This is precisely the classical separation `distance'\footnote{it is not actually a metric} used by e.g. Aldous and Diaconis \citep{chenthree} except for the absolute value. Not worrying about this slight difference  (as it will not be used in the sequel), for a fixed $\nu\in M_p(\G)$, call by the quantum separation distance the quantity $s(k):=\|\nu^{\star k}-\pi \|_{\infty}$.

\bigskip

\begin{theorem} The quantum separation distance is decreasing in $k$.
\begin{proof}
Suppose that $\nu=\F(a)$. Note by Proposition \ref{propertiesofP} that
$$P_{\F(S(a))}a=a\star_A a$$
so that $P_{{\F(S(a))}}a^{\star k}=a^{\star k+1}$. Recalling that $P_{\F(S(a))}(\ind_{\G})=\ind_{\G}$,
\begin{align*}
s(k+1)=\|\nu^{\star k+1}-\pi\|_\infty&=\|\F^{-1}(\nu^{\star k+1}-\pi)\|_\infty^{F(\G)}
\\&=\|a^{\star {k+1}}-\ind_{\G}\|_\infty^{F(\G)}
\\&=\|P_{{\F(S(a))}}(a^{\star k}-\ind_{\G})\|_\infty^{F(\G)}
\\&\leq \|P_{{\F(S(a))}}\|_{\infty}^{F(\G)}\|a^{\star k}-\ind_{\G}\|_{\infty}^{F(\G)}
\end{align*}
Note that $F(\G)$ together with $\|\cdot\|_\infty^{F(\G)}$ is a $\mathrm{C}^*$-algebra. Therefore as $P_{\F(S(a))}:F(\G)\raw F(\G)$ is a positive map (Proposition \ref{propertiesofP}) between unital $\mathrm{C}^*$-algebras, it satisfies the hypotheses of Corollary 2.9 of Paulsen \citep{Paul}. This gives
$$\|P_{{\F(S(a))}}\|_{\infty}^{F(\G)}=\|P_{{\F(S(a))}}\ind_{\G}\|_{\infty}^{F(\G)}=\|\ind_{\G}\|_{\infty}^{F(\G)}=1.$$
Therefore $s(k+1)\leq s(k)$ $\bullet$
\end{proof}
\end{theorem}

\bigskip

Proving the corresponding result for total variation distance does not seem so straightforward.

\bigskip

\begin{theorem}
For a random walk on a classical group, the total variation distance is decreasing.
\end{theorem}
\begin{proof}
In the same notation as before, consider
\begin{align*}
\|\nu^{\star k+1}-\pi\|&=\frac12 \|\F^{-1}(\nu^{\star k+1}-\pi)\|_1^{F(G)}
\\&=\frac12 \|a^{\star {k+1}}-\ind_{G}\|_1^{F(G)}
\\&=\frac12 \|P_{{\F(S(a))}}(a^{\star k}-\ind_{G})\|_1^{F(G)}
\\&\leq \frac12\|P_{{\F(S(a))}}\|_{1}^{F(G)}\|a^{\star k}-\ind_{G}\|_{1}^{F(G)}
\end{align*}
If $\mu,\phi\in M_p(G)$ then $\mathcal{F}^{-1}(\mu-\phi)\in F(G)_{\text{sa}}$: that is $\mathcal{F}^{-1}(\mu-\phi)$ is a real-valued function.
The algebra of real-valued functions on $G$, $F(G)_{\text{sa}}$ is a real sub-*-algebra of the *-algebra $F(G)$. Positivity gives a partial order on $F(G)_{\text{sa}}\subset F(G)$:
$$f\geq g\Leftrightarrow f-g\geq 0.$$
Furthermore, $F(G)_{\text{sa}}$ is a Riesz space with
$$(f\vee g)(t)=\max\{f(t),g(t)\}.$$
Furthermore the norm $\|\cdot\|_1^{F(G)}$ is a \emph{Riesz} norm (Example 1.3.3, Batty and Robinson \citep{robinson}) and Robinson shows (Lemma 3.3, \citep{robin}) that for such a space, the norms of positive operators are determined by their behaviour on the positive cone and so
$$\|P_{\F(S(a))}\|_{F(G)_{\text{sa}}\raw F(G)_{\text{sa}}}=\sup_{\underset{f\geq 0}{\|f\|_1^{F(G)_{\text{sa}}}\leq 1}}\|P_{\F(S(a))}(f)\|_1^{F(G)_{\text{sa}}}.$$

\bigskip

Therefore consider an $f\geq 0$ with $\|f\|_1^{F(G)_{\text{sa}}}\leq 1$. Note that as $P_{\F(S(a))}$ is positive, $|P_{\F(S(a))}(f)|=P_{\F(S(a))}(f)$ and by Proposition \ref{propertiesofP} v., $\dsp\int_G=\int_GP_{\F(S(a))}$:
\begin{align*}\|P_{\F(S(a))}(f)\|_1^{F(G)_{\text{sa}}}&=\int_G|P_{\F(S(a))}(f)|=\int_GP_{\F(S(a))}(f)
\\&=\int_Gf=\int_G|f|
\\&=\|f\|_1^{F(G)_{\text{sa}}}\leq 1.\end{align*}
Note, however, that $\|P_{{\F(S(a))}}(\ind_G)\|_1^{F(G)_{\text{sa}}}=1$. Clearly $\|\cdot \|_1^{F(G)}$ and $\|\cdot\|_1^{F(G)_{\text{sa}}}$ coincide for elements of $F(G)_{\text{sa}}$ and so
$$\|\nu^{\star k+1}-\pi\|\leq \frac12\|P_{{\F(S(a))}}\|_{1}^{F(G)}\|a^{\star k}-\ind_{G}\|_{1}^{F(G)}\leq \frac{1}{2}\|a^{\star k}-\ind_{G}\|_{1}^{F(G)}=\|\nu^{\star k}-\pi\|\,\,\,\bullet$$

\bigskip

This approach does not work for truly quantum groups because of Sherman's Theorem \citep{sherman} which says that the self-adjoint elements of a  $\mathrm{C}^*$-algebra form a Riesz space in this way if and only if the algebra is commutative. Thus, in the truly quantum case, the results of Robinson may not be used.

\end{proof}

\wxo
\chapter{Diaconis--Shahshahani Theory}
In a seminal monograph \citep{PD}, Diaconis shows how to exploit representation theory to produce upper bounds for the distance to random of a random walk on a finite classical group. As this work extensively uses the algebra of functions and `sum over points' arguments rather than points of the space $G$, it is ripe for exploitation via the transfer principle spoken about in the introduction. The foundation --- the representation theory of finite quantum groups --- has been set by Woronowicz \citep{193} and \citep{202} in his development of the corepresentation theory of compact quantum groups. For a presentation of the classical Diaconis--Shahshahani theory see Chapter Three of the author's MSc thesis \citep{MSc}. In this chapter, a brief introduction to classical representation theory is presented  followed by the quantisation of this theory. The Fourier theory for finite groups is not quantised  by the quantisation functor but the necessary generalisation, which leans very strongly on the work of Van Daele \citep{VD}, is presented. Finally the Quantum Diaconis--Shahshahani Upper Bound Lemma is presented. It is  applied to two commutative examples and a cocommutative example. Finally the formula is used to analyse all symmetric random walks on $\mathbb{KP}$ as well as a family of random walks on $\mathbb{KP}_n$.

\section{Basics of Classical Representation Theory\label{Basics}}
When it comes to developing the theory of group representations, \emph{Group Representations in Probability and Statistics} --- the seminal monograph of Diaconis \citep{PD} --- follows Serre \citep{Serre} quite closely. A more random-walk focussed summary of this material may be found in Section 3.1 of the author's MSc thesis \citep{MSc}.

\bigskip

A \textit{representation} $\rho$
of a finite group $G$ is a group homomorphism from $G$ into  $GL(V)$ for some vector space $V$. The
dimension of the vector space\footnote{at this point the underlying vector space may be infinite dimensional but it can be shown that the only representations of any interest are of finite dimension. Also the underlying field is unspecified at this point but  it can be shown that the only representations of any interest \emph{for this work} are over complex vector spaces.} is called the \textit{dimension of
$\rho$} and is denoted by $d_\rho$.   If $W$ is a subspace of $V$
invariant under $\rho(G)$, then
$\rho_{|W}$ is called a \textit{subrepresentation}. It can be shown that every representation splits into a direct sum of subrepresentations. Both
$\{\bld{0}\}$ and $V$ itself yield trivial subrepresentations in the obvious way. A
representation $\rho$ that admits no non-trivial subrepresentations
is called \textit{irreducible}.  Inductively, therefore, every representation is a direct sum of irreducible
representations. Given representations $\rho$ acting on $V$ and $\varrho$ acting on $W$, a linear map $f\in L(V,W)$ is said to \emph{interwine} $\varrho$ and $\rho$ (and be an \emph{intertwiner}) if  $\varrho\circ f=f\circ \rho$. If there is an invertible intertwiner between $\varrho$ and $\rho$ they are said to be
 \textit{equivalent as representations}, denoted $\rho\equiv \varrho$. Furthermore the operators $\rho(s)$  can be assumed to be unitary as every irreducible representation is equivalent to a unitary one.

\bigskip

Note that when a basis of $V$ is fixed, the representation $\rho$ maps from $G$ into $M_{d_\rho}(\C)$:
$$\rho(s)=\left(\begin{array}{ccc}\rho_{11}(s) & \cdots &\rho_{1d_{\rho}(s)}\\ \vdots & & \vdots \\ \rho_{d_\rho1}(s)&\cdots &\rho_{d_{\rho}d_{\rho}}(s)\end{array}\right).$$

Note that the matrix coefficients $\{\rho_{ij}\,:\,i,j=1\dots d_\rho\}$ are functions on $G$, $s\mapsto \rho_{ij}(s)$ and so $\rho_{ij}\in F(G)$. Schur's Lemma, a vital result in the area, says that:
\begin{itemize}
\item if two representations are inequivalent then their only intertwiner is the zero map
\item if two representations on a vector space, $V$, are equivalent then all their intertwiners are scalars.
\end{itemize}
Associated to the Haar measure on a group, $\dsp\int_G:F(G)\raw \C$,  there is an inner product $\langle f,g\rangle=\dsp\int_G f^*g$. By considering a certain intertwiner between two representations in the context of Schur's Lemma, it can be seen that the matrix elements of irreducible representations are orthogonal. In fact, if $\rho^\alpha$ and $\rho^\beta$ are two unitary irreducible representations:

$$\langle \rho^\alpha_{ij},\rho^\beta_{k\ell}\rangle=\begin{cases}\dsp\frac{\delta_{ik}\delta_{j\ell}}{d_{\rho}} & \text{ if }\rho^\alpha\equiv \rho^\beta,
\\ 0 & \text{ otherwise }.\end{cases}$$
The category of representations of a finite group (with intertwiners as morphisms) is a \emph{monoidal category} and considering the \textit{regular representation}, defined with respect to a complex
vector space with basis $\{e_s\}$ indexed by $s\in G$ via $r(s)(e_t):=e_{st}$, in light of this fact, gives us the following theorem.

\bigskip

\begin{theorem}\label{PeterW}
(Finite Peter--Weyl Theorem)
Let $\mathcal{I}=\operatorname{Irr}(G)$ be an index set for a family of pairwise-inequivalent irreducible representations of $G$. Where $d_{\alpha}$ is the dimension of the vector space on which $\rho^\alpha$ acts ($\alpha\in\mathcal{I}$),
$$\left\{\rho_{ij}^\alpha\,|\,i,j=1\dots d_\alpha,\,\alpha\in \mathcal{I}\right\},$$
the set of matrix elements of $G$, is an orthogonal basis of $F(G)$.
\end{theorem}
\begin{proof}
See the discussion on P.35-36 of \citep{MSc} $\bullet$
\end{proof}
Once the notion of a representation has been quantised, it will be seen that there is also a Finite Peter--Weyl Theorem for quantum groups.

\section{Representations of Quantum Groups}
Employing a process very similar to that of using the quantisation functor, it is possible to quantise the representation theory of finite groups. Let $V$ be a vector space. A \emph{representation} of $G$ on $V$ is a right linear group action $\Phi:V\times G\raw V$. We can define a \emph{representation matrix} $\rho:G\raw GL(V)$ by
$$\Phi(v,s)=\rho(s^{-1})v,$$
which is a group homomorphism $\rho:G\raw GL(V)$. Linearly extending $\Phi$ to $V\times \C G$ (and using the embedding of $G$ into $\C G$, $s\hookrightarrow \delta^s$) gives the bilinear map:
$$\overline{\Phi}:V\times\C G\raw V\,,\,\,\,(v,\delta^s)\mapsto \rho(\delta^{s^{-1}})v$$
and then the linear map
$$\tilde{\Phi}:V\otimes\C G\raw V\,,\,\,\,v\otimes\delta^s\mapsto \rho(\delta^{s^{-1}})v.$$
The properties that make $\Phi$ an action are encoded by two relations involving $\tilde{\Phi}$, $\nabla$ and $\eta_{\C G}$. The first is compatibility:
$$\tilde{\Phi}\circ(I_{V}\otimes \nabla)(v\otimes \delta^s\otimes \delta^h)=\tilde{\Phi}\circ (\tilde{\Phi}\otimes I_{\C G})(v\otimes\delta^s\otimes \delta^h).$$
The relation for identity emanates from
$$\tilde{\Phi}\circ(I_V\otimes \eta_{\C G})v\cong\tilde{\Phi}\circ(I_V\otimes \eta_{\C G})(v\otimes 1_{\C})=\tilde{\Phi}(v\otimes\delta^e)=I_{V}v,$$
i.e. $\tilde{\Phi}\circ (I_V\otimes \eta_{\C G})=I_V$. Fix a basis $\{e_i\}$ of $V$ and let $\{e^i\}$ be the basis of $V^*$ dual to this basis. Now apply the dual functor to this:
$$\kappa_{\rho}:=\mathcal{D}(\tilde{\Phi}):V^*\raw V^*\otimes F(G)\,,\,\,\,\kappa_{\rho}(e^i)=e^i\circ \tilde{\Phi}.$$
Together with the dual statements for compatibility (mpatibility!) and identity:
\begin{eqnarray}
(I_{V^*}\otimes \Delta)\circ \kappa_\rho=(\kappa_{\rho}\otimes I_{F(G)})\circ \kappa_\rho\label{5.1}
\\ (I_{V^*}\otimes\eps)\circ \kappa_\rho=I_{V^*},
\end{eqnarray}
this motivates the definition of a corepresentation of the algebra of functions on a quantum group on a complex vector space.
\begin{definition}
A \emph{corepresentation} of the algebra of functions on a quantum group $\mathbb{G}$ on a complex vector space is a linear map $\kappa:V\raw V\otimes F(\mathbb{G})$ that satisfies:
$$(\kappa\otimes I_{A})\circ \kappa=(I_V\otimes\Delta)\circ\kappa\qquad\text{and}\qquad(I_V\otimes \eps)\circ \kappa=I_V.$$
\end{definition}

\begin{proposition}
For any group homomorphism $\rho:G\rightarrow GL(V)$ on a finite group, the map $\kappa_\rho$, \emph{induced by the representation $\rho$}, given by
$$\kappa_\rho(v)=\sum_{t\in G}\rho(t)v\otimes \delta_t$$
is a corepresentation of $F(G)$ on $V$.

\bigskip

On the other hand, suppose that $\kappa_\alpha$ is a corepresentation of $F(G)$ on $V_\alpha$ (with basis $\{e_i\}$) given by
$$\kappa_\alpha(e_j)=\sum_{i=1}^{d_\alpha}e_i\otimes\rho_{ij}.$$
Then the map $\rho_{\alpha}:G\raw GL(V_\alpha)$, $s\mapsto (\rho_{ij}(s))$ is a group homomorphism.
\end{proposition}

\begin{proof}
Clearly $\kappa_\rho$ is a linear map. Let $v\in V$:
\begin{align*}
(\kappa_\rho\otimes I_{F(G)})\circ\kappa_\rho(v)&=\sum_{s\in G}\kappa_\rho(\rho(s)v)\otimes \delta_s=\sum_{t,s\in G}\rho(t)\rho(s)v\otimes \delta_t\otimes\delta_s
\\&=\sum_{t,s\in G}\rho(ts)v\otimes \delta_t\otimes \delta_s\underset{\ell=ts}{=}\sum_{\ell,t\in G}\rho(\ell)v\otimes\delta_{\ell s^{-1}}\otimes\delta_s
\\&=\sum_{\ell\in G}\rho(\ell)v\otimes\Delta(\delta_{\ell})=(I_V\otimes \Delta)\circ\kappa_\rho(v).
\end{align*}
Also
\begin{align*}
(I_V\otimes\eps)\circ\kappa_\rho(v)&=\sum_{t\in G}\rho(t)v\otimes \eps(\delta_t)
\\&=\rho(e)v=v=I_V(v).
\end{align*}
The converse statement is easily seen to be a consequence of hitting $\delta^s\otimes\delta^t$ with the relation:
$$\Delta \rho_{ij}=\sum_{k}\rho_{ik}\otimes\rho_{kj}.$$
This relation, a consequence of (\ref{5.1}) will be proved in the sequel (Proposition \ref{matrixel})\,\,\,$\bullet$
\end{proof}

Therefore, using the Gelfand philosophy, a corepresentation of the algebra of functions on a quantum group $F(\mathbb{G})$ may be called a \emph{representation} of the quantum group $\mathbb{G}$.

Recall that $\Phi(u,s)=\rho(s^{-1})u$ and so
$$\kappa_\rho(v)=\sum_{t\in F}\Phi(v,t)\otimes\delta_{t^{-1}}.$$


\begin{example}
\emph{(Examples from Classical Groups)}
\begin{enumerate}
\item Define a representation $\Phi$ of  the quaternion group on $\C^2$, via the group homomorphism $\rho:\mathcal{Q}\raw GL(\C^2)$ by
\beq
\begin{array}{cc}
\rho(i)=\left(\begin{array}{cc}i & 0\\ 0 & -i\end{array}\right), & \rho(k)=\left(\begin{array}{cc}0 & -1\\ 1 & 0\end{array}\right),
\\ \rho(j)=\left(\begin{array}{cc}0 & i\\ i & 0\end{array}\right), & \rho(1)=I,\,\,\rho(-s)=-\rho(s)
\end{array}
\enq
This yields a corepresentation of $F(\mathcal{Q})$ on $\C^2$:
\begin{align*}
\kappa_{\rho}\left(\begin{array}{c}z_1\\ z_2\end{array}\right)&=\left(\begin{array}{c}z_1\\ z_2\end{array}\right)\otimes\delta_1+\left(\begin{array}{c}-z_1\\ -z_2\end{array}\right)\otimes\delta_{-1}+\cdots +\left(\begin{array}{c}z_2\\ -z_1\end{array}\right)\otimes\delta_{-k}
\end{align*}
\item Recall the \emph{trivial representation} $\tau$, defined for any group $G$ on $\C$ by $\tau(s)\lambda=\lambda$ for all $\lambda\in \C$. As a corepresentation
    \begin{align*}
    \kappa_\tau(\lambda)&=\sum_{t\in G}\tau(t)\lambda\otimes\delta_t\cong\sum_{t\in G}\lambda\otimes\delta_t
    \\&\cong\lambda\,\sum_{t\in G}\delta_t=\lambda\otimes\ind_{G}.
    \end{align*}
Exactly analogously, the map $\tau:\C\raw\C\otimes F(\G)$,
$$\tau(\lambda)=\lambda\otimes\ind_{\G}$$
will be called the \emph{trivial corepresentation} of a quantum group $\G$ on $\C$. Note that $\tau\cong \eta_{F(\G)}$.

\bigskip

\item Let $G$ be a finite group. The \emph{regular representation} $R:\C G\times G\raw \C G$  is defined by $R(\delta^s,t)=\delta^{st}$. Applying the quantisation routine to this regular representation  gives the comultiplication on $F(G)$. Therefore the comultiplication is a corepresentation of $F(G)$.

\end{enumerate}

\end{example}

By looking at how statements about a group representation $\rho:G\raw \operatorname{GL}(V)$ are translated into statements about the induced corepresentation $\kappa_\rho$, some quantum analogues of classical definitions may be motivated.

\bigskip

Let $\kappa_\rho$ be a corepresentation on a finite group with an invariant subspace $W=\text{ span}\{w_1,\dots,w_m\}$. Then
$$\text{ span}\left\{\kappa_\rho(w_i)\right\}=\text{ span}\left\{\sum_{t\in G}\underbrace{\rho(t)w_i}_{\in W}\otimes\delta_t \right\}\subset W\otimes F(G)$$
If $V$ is an inner product space and $A$ a $\mathrm{C}^*$-algebra, an $A$-valued sesquilinear inner product on $V\otimes A$ can be defined  by
$$\la v\otimes a,w\otimes b\ra_A:=\la v,w\ra a^*b.$$
If $\rho$ is unitary then
\begin{align*}
\left\la \kappa_\rho(u),\kappa_\rho(v)\right\ra_{F(G)}&=\left\la\sum_{t\in G} \rho(t)u\otimes\delta_t,\sum_{s\in G}\rho(s)v\otimes\delta_s\right\ra_{F(G)}
\\&=\sum_{t,s\in G}\left\la\rho(t)u,\rho(s)v\right\ra\delta_t^*\delta_s
\\&\underset{s=t}{=}\sum_{t\in G}\la\rho(t)u,\rho(t)v\ra\,\delta_t
\\&\underset{\rho\text{ unit.}}{=}\sum_{t\in G}\la u,v\ra\,\delta_t=\la u,v\ra\,\sum_{t\in G}\delta_t=\la u,v\ra\,\mathds{1}_{G}.
\end{align*}
Suppose that $T$ intertwines representations of $\rho$  and $\varrho$ of $G$:
\begin{align*}\kappa_\rho(t)\left(T(u)\right)&=\sum_{t\in G}\rho(t)T(u)\otimes\delta_t
\\&\underset{\rho\circ T=T\circ\varrho}{=}\sum_{t\in G}T(\varrho(t)u)\otimes\delta_t
\\&=(T\otimes I_{F(G)})\kappa_{\varrho}(u).
\end{align*}
\emph{Liberating} these relations gives the following series of definitions.
\begin{definition}
Let $\G$ be a finite quantum group with a representation $\kappa$ on $V$. A subspace $W\subset V$ is \emph{invariant} with respect to $\kappa$ if $\kappa(W)\subset W\otimes F(\G)$. If $V$ contains no non-trivial subspace, $\kappa$ is said to be \emph{irreducible}. If for all $v,\,u\in V$
$$\left\la\kappa(v),\kappa(u)\right\ra_{F(\G)}=\la v,u\ra\,\mathds{1}_{\G},$$
the representation $\kappa$ is said to be \emph{unitary}. When a linear map $T:V\raw V_0$ satisfies
$$\kappa_0\circ T=(T\otimes I_{F(\G)})\circ \kappa,$$
for a representation $\kappa_0$ of $\G$ on a vector space $V_0$, it is said to \emph{intertwine} $\kappa$ and $\kappa_0$ and be an \emph{intertwiner}. Furthermore if $T$ is invertible then $\kappa$ and $\kappa_0$ are \emph{equivalent}.

\bigskip

It can be seen that $V$ can be chosen to be finite dimensional (Theorem 3.2.1, \citep{Timm}). Letting $d_{\kappa}$ denote the dimension of $V$, the linearity of $\kappa$ implies the existence of $d_{\kappa}^2$ elements $\rho_{ij}$ of $F(\G)$:
$$\kappa(e_j)=\sum_{i=1}^{d_\kappa}e_i\otimes\rho_{ij}.$$
These are the \emph{matrix elements} of the representation $\kappa$.

\end{definition}

\bigskip

Note that for the corepresentation $\kappa_\rho$ induced by a representation $\rho$ on a classical group
\begin{align*}
\kappa_{\rho}(e_j)&=\sum_{t\in G}\rho(t)e_j\otimes \delta_t=\sum_{\underset{i}{t\in G}}\rho_{ij}(t)e_i\otimes\delta_t
\\&=\sum_{\underset{i}{t\in G}}e_i\otimes \rho_{ij}(t)\delta_t=\sum_{i}e_i\otimes\rho_{ij},
\end{align*}
justifying the notation and name for these elements of $F(\G)$.

\bigskip

The rest of this section will be concerned with outlining some key results used in the proving the quantum version of the Finite Peter--Weyl Theorem and the approach follows very closely that of Section 3.1.2 of Timmermann \citep{Timm}.

\bigskip

\begin{proposition}\label{matrixel}
For any matrix element $\rho_{ij}\in F(\G)$
$$\Delta(\rho_{ij})=\sum_k \rho_{ik}\otimes\rho_{kj}\text{ and }\eps(\rho_{ij})=\delta_{i,j}.$$
\end{proposition}
\begin{proof}
The first result follows after calculating
\begin{align*}
(\kappa\otimes I_{F(\G)})\circ \kappa(e_j)&=(\kappa\otimes I_{F(\G)})\left(\sum_k e_k\otimes \rho_{kj}\right)
\\&=\sum_{i,k}e_i\otimes\rho_{ik}\otimes\rho_{kj}=\sum_ie_i\otimes\left(\sum_k\rho_{ik}\otimes\rho_{kj}\right),
\end{align*}
and
\begin{align*}
(I_V\otimes\Delta)\circ\kappa(e_j)&=(I_V\otimes\Delta)\left(\sum_i e_i\otimes \rho_{ij}\right)
\\&=\sum_i e_i\otimes\Delta(\rho_{ij}),
\end{align*}
and noting that $(\kappa\otimes I_{F(\G)})\circ\kappa=(I_V\otimes\Delta)\circ\kappa$.

\bigskip

The second result follows because with $(I_V\otimes\eps)\circ\kappa=I_V$
$$(I_V\otimes\eps)\kappa(e_j)=\sum_ie_i\otimes \eps(\rho_{ij})=\sum_i \eps(\rho_{ij})e_i\overset{!}{=}e_j\,\,\,\bullet$$
\end{proof}
Let $\kappa$ be a corepresentation on a vector space $V$. Denote by $\overline{V}$ the conjugate vector space of $V$ and by $v\mapsto \overline{v}$ the canonical conjugate-linear isomorphism. Since $\Delta$ and $\eps$ are $*$-homomorphisms, the map
$$\overline{\kappa}:\overline{V}\raw \overline{V}\otimes F(\G)\,,\,\,\,\overline{e_j}\mapsto \sum \overline{e_i}\otimes \rho_{ij}^*,$$
is a representation again, called the \emph{conjugate} of $\kappa$.

\bigskip

For each $\phi\in \C\G$ define the map\footnote{this map will appear again in the next section} $\widehat{\phi}\in L(\overline{V})$:
$$\widehat{\phi}(\kappa)=(I_{\overline{V}}\otimes \phi)\circ \overline{\kappa}.$$
\begin{proposition}
\label{conv1}
For $\phi,\,\varphi\in \C\G$ and $\kappa$ a representation of $\G$ on $V$
$$\widehat{\phi\star\varphi}(\kappa)=\widehat{\phi}(\kappa)\circ \widehat{\varphi}(\kappa).$$
\end{proposition}
\begin{proof}
Taking the approach of Timmermann (proof of Proposition 3.1.7 ii., \citep{Timm}, start with
\begin{align*}
\widehat{\phi\star\varphi}(\kappa)&=\left(I_{\overline{V}}\otimes(\phi\star \varphi)\right)\circ \overline{\kappa}
\\&=\left(I_{\overline{V}}\otimes(\phi\otimes\varphi)\circ \Delta\right)\circ\overline{\kappa}
\\&=(I_{\overline{V}}\otimes\phi\otimes\varphi)\circ(I_{\overline{V}}\otimes \Delta)\circ\overline{\kappa}.
\end{align*}
Using compatibility,
\begin{align*}
\widehat{\phi\star\varphi}(\kappa)&=(I_{\overline{V}}\otimes\phi\otimes\varphi)\circ(\overline{\kappa}\otimes I_{F(\G)})\circ \overline{\kappa}
\\&=(\widehat{\phi}(\overline{\kappa})\otimes\varphi)\circ\overline{\kappa}
\\&=\left(\widehat{\phi}(\kappa)\otimes I_{\C}\right)\circ(I_{\overline{V}}\otimes \varphi)\circ\overline{\kappa}
\\&\cong \widehat{\phi}(\kappa)\circ\widehat{\varphi}(\kappa)\,\,\,\bullet
\end{align*}
\end{proof}
The following results are presented in Timmermann \citep{Timm} Sections 3.2.1 to 3.2.4.

\bigskip

\begin{theorem}\label{directsum}
Every representation of a finite quantum group $\G$ is equivalent to a direct sum of finite-dimensional irreducible unitary representations $\bullet$
\end{theorem}

\bigskip

\begin{proposition}\label{ortho1}
Let $\G$ be a finite quantum group with Haar measure $\dsp\int_{\mathbb{G}}$ and let $\kappa_\alpha$ and $\kappa_\beta$ be inequivalent irreducible representations of $\G$ on vector spaces $V_\alpha$ and $V_\beta$ with matrix elements $\rho^{\alpha}_{ij}$ and $\rho^{\beta}_{ij}$. Then
$$\int_{\G}\left(\rho^\beta_{ij}\right)^*\rho^\alpha_{kl}=0=\int_{\G}{\rho_{ij}}^\beta\left(\rho_{kl}^\alpha\right)^*.\qquad\bullet$$
\end{proposition}

\bigskip

\begin{proposition}\label{ortho}
Let $\G$ be a finite quantum group with Haar measure $\dsp\int_{\mathbb{G}}$ and let $\kappa$ be an irreducible unitary representation of $\G$ on $V$. Then for all $i,\,j,\,k,\,l$:
\begin{align*}
\int_{\G}\rho_{ij}^*\rho_{kl}=\frac{\delta_{i,k}\delta_{j,l}}{d_{\kappa}}=\int_{\G}\rho_{ij}\rho_{kl}^*.
\end{align*}
Furthermore, the elements $\{\rho_{ij}\}$ are linearly independent $\bullet$
\end{proposition}

\bigskip

\begin{theorem}
\emph{(Quantum Finite Peter--Weyl Theorem)}\label{QPW}
Let $\mathcal{I}=\operatorname{Irr}(\G)$ be an index set for a family of pairwise-inequivalent irreducible representations of $\G$. If $d_{\alpha}$ is the dimension of the vector space on which $\rho^\alpha$ acts ($\alpha\in\mathcal{I}$),
$$\left\{\rho_{ij}^\alpha\,|\,i,j=1\dots d_\alpha,\,\alpha\in \mathcal{I}\right\},$$
the set of matrix elements of $\mathcal{I}$, is an orthogonal basis of $F(\G)$ with respect to the inner product
\beq
\langle a,b\rangle:=\int_{\G}a^*b\,,\,\,\,a,b\,\in F(\G)\qquad\bullet
\enq

\end{theorem}

\bigskip

Note this reads exactly   as the classical version. In particular it also means that there is a finite number of inequivalent irreducible representations.

\section{Diaconis--Van Daele Theory}
The following definition is similar to that of Simeng Wang (formula (2.5), \citep{Wang}) save for a choice of left-right. As remarked upon by Simeng Wang, his definition is similar to earlier definitions of Kahng and also Caspers save for the presence of the conjugate representation $\overline{\kappa}_\alpha$ rather than $\kappa_{\alpha}$ itself. As Wang explains, the conjugate representation is used to be compatible with standard definitions in classical analysis on compact groups and hence most welcome for this work.
\begin{definition}
(The Fourier Transform)
Let $\G$ be a quantum group with representation notation as before. Then the \emph{Fourier transform} is a map:
$$F(\G)\rightarrow \bigoplus_{\alpha\in\operatorname{Irr}(\G)}L(\overline{V}_{\alpha})\,,\,\,a\mapsto \widehat{a},$$
defined, with some abuse of notation, for each $\alpha\in \operatorname{Irr}(\G)$:
$$\widehat{a}(\alpha)=\left(I_{\overline{V_\alpha}}\otimes \F(a)\right)\circ \overline{\kappa}_\alpha.$$
Each $\widehat{a}(\alpha)$ is called the \emph{Fourier transform of $a$ at the representation $\alpha$}. For $\varphi\in\C\G$, as has been seen:
$$\widehat{\varphi}:=(I_{\overline{V_\alpha}}\otimes \varphi)\circ\overline{\kappa_\alpha}.$$
\end{definition}
The maps $\{\widehat{a}(\alpha)\,:\,\alpha\in\operatorname{Irr}(\G)\}$ play a key role in the sequel.

\bigskip

\begin{theorem}
(Diaconis--Van Daele Inversion Theorem)\label{DVDIT}
Let $\varepsilon$ be the counit of a quantum group $\G$ and $a\in F(\G)$. Then
\beq\int_{\widehat{\G}}\mathcal{F}(a)=\eps\left(a\right)=\sum_{\alpha\in\text{Irr}(\G)} d_\alpha \text{Tr}\left(\widehat{a}\left(\alpha\right)\right).\enq
where the sum is over the irreducible representations of $F(\G)$.
\end{theorem}
\begin{proof}
Both sides are linear in $a$ so it suffices to check $a=\rho_{kl}^{\beta}$ for $\beta\in\text{Irr}(\G)$. The left-hand side reads
$$\eps\left(\rho_{kl}^\beta\right)=\delta_{k,l}.$$

To calculate the right-hand-side, calculate for a given representation the trace of $\widehat{a}\left(\alpha\right)$. Let $\overline{e_j}\in \overline{V_\alpha}$ and calculate
\begin{align*}
\widehat{\rho_{kl}^{\beta}}\left(\alpha\right)\overline{e_j}&=\left(I_{\overline{V_\alpha}}\otimes\mathcal{F}\left(\rho_{kl}^\beta\right)\right)\sum_{i} \overline{e_i}\otimes\left(\rho_{ij}^\alpha\right)^*
\\&=\sum_i \mathcal{F}\left(\rho_{kl}^\beta\right)\left(\rho_{ij}^{\alpha}\right)^* \overline{e_i}
\\&=\sum_i \int_{\G}\left(\left(\left(\rho_{ij}^{\alpha}\right)^*\rho_{kl}^\beta\right)\right) \overline{e_i}
\end{align*}
This is zero unless $\alpha\equiv \beta$. If $\alpha\equiv\beta$ then we have
\begin{align*}
\widehat{\rho_{kl}^{\beta}}\left(\alpha\right)\overline{e_j}&=\int_{\G}\left(\left(\rho_{kj}^\beta\right)^*\rho_{kl}^\beta\right)\cdot\overline{e_k}
\\&=\frac{1}{d_\beta}\delta_{j,l}\overline{e_k}
\\ \Raw \operatorname{Tr}\left(\widehat{\rho_{kl}^\beta}(\beta)\right)&=\sum_j\langle \overline{e_j},\widehat{\rho_{kl}^\beta}(\beta)\overline{e_j}\rangle_{\overline{V_\beta}}
\\&=\sum_j \left\langle \overline{e_j},\frac{1}{d_\beta}\delta_{j,l}\overline{e_k}\right\rangle
\\&=\frac{1}{\delta_\beta}\delta_{k,l}.
\end{align*}

Multiply this by $d_\beta$ to get $\delta_{k,l}$ $\bullet$
 \end{proof}

 \bigskip

 \begin{theorem}
(Diaconis--Van Daele Convolution Theorem)\label{DVDCT}
For a representation $\kappa_\alpha$ of $\G$ and $a,\,b\in F(\G)$
 $$\widehat{a\star_{A}b}\left(\alpha\right)=\widehat{a}\left(\alpha\right)\circ\widehat{b}\left(\alpha\right).$$
\end{theorem}
 \begin{proof}
 \begin{align*}
 \widehat{\left(a\star_A b\right)}\left(\alpha\right)&=\left(I_{\overline{V_\alpha}}\otimes\mathcal{F}\left(a\star_A b\right)\right)\overline{\kappa_\alpha}
 \\&=\left(I_{\overline{V_\alpha}}\otimes \mathcal{F}\left(a\right)\star\mathcal{F}\left(b\right)\right)\overline{\kappa_\alpha}
 \\&=\widehat{\F(a)\star \F(b)}(\kappa_\alpha)
\\&=\widehat{\F(a)}(\kappa_\alpha)\circ\widehat{\F(b)}(\kappa_\alpha)
 \\&=\widehat{a}\left(\alpha\right)\circ\widehat{b}\left(\alpha\right).
 \end{align*} The fourth line is Proposition \ref{conv1} and  Van Daele's Convolution Theorem \ref{VDCT} was also used $\bullet$
 \end{proof}

\bigskip

\begin{lemma}\label{lemmax}
Where the sum is over unitary irreducible representations,
$$\int_{\widehat{\G}}\left(\mathcal{F}\left(a\right)\star\mathcal{F}\left(b\right)\right)=\sum_{\alpha\in\text{Irr}(\G)}{d_\alpha}\text{ Tr}\left(\widehat{a}\left(\alpha\right)\widehat{b}\left(\alpha\right)\right).$$
\end{lemma}
\begin{proof}
The proof uses the convolution theorem of Van Daele and the definition of the Haar measure on $\widehat{\G}$ to find
$$\int_{\widehat{\G}}\left(\mathcal{F}\left(a\right)\star\mathcal{F}\left(b\right)\right)=\int_{\widehat{\G}}\mathcal{F}\left(a\star_{A}b\right)=\eps\left(a\star_A b\right).$$
Now use the Diaconis--Van Daele Inversion Theorem \ref{DVDIT}
\begin{align*}
\eps\left(a\star_A b\right)&=\sum_{\alpha\in\text{Irr}(\G)}d_\alpha\text{ Tr}\left(\widehat{a\star_A b}\left(\alpha\right)\right)
\\&=\sum_{\alpha\in\text{Irr}(\G)}d_\alpha\text{ Tr}\left(\widehat{a}\left({\alpha}\right)\widehat{b}\left({\alpha}\right)\right)\qquad\bullet
\end{align*}
\end{proof}
%
%

\begin{proposition}
Suppose that $\mathcal{F}\left(a\right)$ is a state. If $\kappa_\tau$ is the trivial representation, $\lambda\mapsto \lambda\otimes \mathds{1}_{\G}$,  then $\widehat{a}\left(\tau\right)=I_\C$.
\end{proposition}
\begin{proof}
\begin{align*}
\widehat{a}\left(\tau\right)\lambda&=\left(I_{\overline{\C}}\otimes\mathcal{F}\left(a\right)\right)\overline{\kappa_{\tau}}\left(\lambda\right)
\\&=\left(I_{\overline{\C}}\otimes \mathcal{F}\left(a\right)\right)(\lambda\otimes \mathds{1}_{\G}^*)=\lambda\otimes1=\lambda\qquad\bullet
\end{align*}
\end{proof}

\begin{proposition}
Suppose that $\kappa_{\alpha}$ is a non-trivial and irreducible representation, then $\widehat{\mathds{1}_{\G}}\left(\alpha\right)=0$.\label{vanish}
\end{proposition}
\begin{proof}
A calculation:
\begin{align*}
\widehat{\mathds{1}_{\G}}\left(\alpha\right)\overline{e_j}&=\left(I\otimes\mathcal{F}\left(\mathds{1}_{\G}\right)\right)\sum_i\overline{e_i}\otimes\left(\rho_{ij}^\alpha\right)^*
\\&=\sum_i\overline{e_i}\mathcal{F}\left(\mathds{1}_{\G}\right)\left(\rho_{ij}^\alpha\right)^*
\\&=\sum_{i}\int_{\G}\left(\left(\rho_{ij}^\alpha\right)^*\mathds{1}_{\G}\right)\cdot \overline{e_i}.
\end{align*}
Note that $\mathds{1}_{\G}$ is the matrix element of the trivial representation and $\alpha$ is not equivalent to the trivial representation. Therefore, by the first orthogonality relation $\widehat{\mathds{1}_{\G}}\left(\alpha\right)=0$ as required $\bullet$
\end{proof}

\bigskip

Note in particular that $\mathcal{F}\left(\mathds{1}_{\G}\right)=\dsp\int_{\G}=\pi$.

\bigskip

\begin{proposition}
Let $\G$ be a quantum group. Then $\mathcal{F}(a)\in M_p(\G)$ if and only if $a\in F(\G)^+$ such that $\dsp\int_{\G}a=1$.
\begin{proof}
If $a$ is positive write it as $b^*b$ and note that for $f\in F(\G)$
$$\F(a)(f^*f)=\int_{\G}(f^*fb^*b)=\int_{\G}(bf^*fb^*)=\int_{\G}((fb^*)^*(fb^*))\geq 0.$$
Clearly $\F(a)\ind_{\G}=1$ if $\dsp \int_{\G}a=1$.

\bigskip

Considering the other direction; note that if $\dsp\int_{\G} a\neq 1$ then $\F(a)\ind_{\G}\neq 1$ so assume $\dsp\int_{\G}a=1$.
Suppose that $a$ is self-adjoint so that $a=a_+-a_-$ where $a_+,a_-\in F(\G)^+$ and  $a_+a_-=0$. If $a$ is not positive $a_-\neq0$. Noting that, consider:
$$\F(a)(a_-)=\int_{\G}a_-(a_+-a_-)=-\int_{\G}a_-^2<0.$$
Therefore, as $\int_{\G}$ is faithful, $\F(a)$ cannot be a state if $a$ is not positive. Suppose that $a$ is not self-adjoint but equal to $c+id=c+i(d_+-d_-)$ ($d\neq 0$) then:
\begin{align*}
\F(c+id)(\pm d_{\pm})&=\pm\int_{\G}d_-(c+id)
\\&=\pm\int_{\G}d_{\pm}(c+i(d_+-d_-))
\\ &=\pm \int_{\G}d_{\pm}c\pm \int_{\G}id_{\pm}(d_+-d_-)
\\ &=\pm \int_{\G}d_{\pm}c+ i\int_{\G}d_{\pm}^2\not\in\R.
\end{align*}
 Therefore $a$ must be positive $\bullet$
\end{proof}
\end{proposition}

\begin{proposition}\label{star}
If $\kappa_\alpha$ is unitary and $\F(a)\in M_p(\G)$ then
$$\widehat{\mathcal{F}(a)^*}(\kappa_\alpha)=\widehat{a}(\kappa_{\alpha})^*.$$
Note the first involution is in $\C\G$ and the second is in $L(\overline{V_\alpha})$.\end{proposition}
\begin{proof}
First take $\overline{e_i}$ (and using $S(\rho_{ij}^{\alpha})=\left(\rho_{ji}^\alpha\right)^*$ (Proposition 3.1.7 v., \citep{Timm}) and $S^2=I_{F(\G)}$)
\begin{align*}
\widehat{\mathcal{F}(a)^*}(\kappa_\alpha)\overline{e_i}&=\left(I_{\overline{V_\alpha}}\otimes \mathcal{F}(a)^*\right)\overline{\kappa_\alpha}(\overline{e_i})
\\&=\left(I_{\overline{V_\alpha}}\otimes \mathcal{F}(a)^*\right)\sum_j \overline{e_j}\otimes\left(\rho_{ij}^\alpha\right)^*
\\&=\sum_j\overline{\mathcal{F}(a)S\left(\left(\rho_{ij}^\alpha\right)^*\right)^*}\overline{e_j}
\\&=\sum_j\overline{\int_{\G}\left(\left(\rho_{ji}^\alpha\right)^* a\right)}\cdot \overline{e_j}
\\&=\sum_j \int_{\G}\left(\left(\rho_{ji}^\alpha\right) a^*\right)\cdot \overline{e_j}.
\end{align*}
Now looking at the right-hand side:
\begin{align*}
\widehat{a}(\alpha)\overline{e_i}&=(I_{\overline{V_\alpha}}\otimes \mathcal{F}(a))\overline{\kappa_\alpha}(\overline{e_i})
\\&=(I_{\overline{V_\alpha}}\otimes \mathcal{F}(a))\sum_j\overline{e_j}\otimes \left(\rho_{ij}^\alpha\right)^*
\\&=\sum_j\int_{\G}\left(\left(\rho_{ij}^\alpha\right)^*a\right)\cdot \overline{e_j}.
\end{align*}
Considering that the involution in $L(\overline{V_\alpha})$ is the conjugate-transpose, this is enough to show the result $\bullet$
\end{proof}

\bigskip

Recall, from Chapter \ref{dist}, that the \emph{total variation distance} is defined as
$$\|\nu^{\star k}-\pi\|:=\frac12 \|\mathcal{F}^{-1}(\nu^{\star k}-\pi)\|_1^{F(\G)}$$
and that
$$\|\nu^{\star k}-\pi\|\leq \frac12 \|\nu^{\star k}-\pi\|_2^{\C \G},$$

where
$$\|\mu\|_2^{\C G}=\left(\int_{\widehat{\G}}(\mu^*\star \mu)\right)^{1/2}.$$

\bigskip

\begin{lemma}
(Quantum Diaconis--Shahshahani Upper Bound Lemma)
Let $\nu\in M_p(\G)$. Then
\beq
\|\nu^{\star k}-\pi\|^2\leq \frac14\sum_{\alpha\in\text{Irr}(\G)\backslash\{\tau\}}d_\alpha\text{Tr} \left[\left(\hat{\nu}\left(\alpha\right)^*\right)^k\left(\hat{\nu}\left(\alpha\right)\right)^k\right],
\enq
where the sum is over all non-trivial, irreducible representations.
\end{lemma}
\begin{proof}
Let $\nu=\mathcal{F}(a)$, and recalling that $\F(\ind_{\G})=\pi$ write
\begin{align*}
\left(\|\mathcal{F}\left(a\right)-\pi\|_2^{\C G}\right)^2&=\int_{\widehat{\G}}\left(\left(\mathcal{F}\left(a\right)-\mathcal{F}\left(\mathds{1}_{\G}\right)\right)^*\star\left(\mathcal{F}\left(a\right)-\mathcal{F}\left(\mathds{1}_{\G}\right)\right)\right)
\\&=\int_{\widehat{\G}}\mathcal{F}\left(a-\mathds{1}_{\G}\right)^*\star\mathcal{F}\left(a-\mathds{1}_{\G}\right).
\end{align*}
Now using Lemma \ref{lemmax} and Proposition \ref{star}, this is equal to
\begin{align*}
\left(\|\mathcal{F}\left(a\right)-\pi\|_2^{\C G}\right)^2&=\sum_{\alpha\in\operatorname{Irr(\G)}}d_{\alpha}\text{Tr} \left[\widehat{\left(a-\mathds{1}_{\G}\right)}\left(\alpha\right)^*\widehat{\left(a-\mathds{1}_{\G}\right)}\left(\alpha\right)\right].
\end{align*}
Now note that
\begin{align*}\widehat{\left(a-\mathds{1}_{\G}\right)}\left(\alpha\right)&=\widehat{a}\left(\alpha\right)-\widehat{\mathds{1}_{\G}}\left(\alpha\right).
\end{align*}
If $\alpha=\tau$, the trivial representation, then this yields zero as both terms are the identity on $\C$. If $\alpha$ is non-trivial, then $\widehat{\mathds{1}_{\G}}\left(\alpha\right)=0$ and thus (using the notation $\widehat{a}(\alpha)=\widehat{\nu}(\alpha)$:
$$\|\nu-\pi\|^2\leq \frac14\left(\|\mathcal{F}\left(a\right)-\pi\|_2^{\C G}\right)^2 =\frac14\sum_{\alpha\in\text{Irr}(\G)\backslash\{\tau\}}d_\alpha\text{Tr} \left[\widehat{\nu}\left(\alpha\right)^*\widehat{\nu}\left(\alpha\right)\right].$$
Apply the Diaconis--Van Daele Convolution Theorem \ref{DVDCT} $k$ times   $\bullet$
\end{proof}

\bigskip

Note that this is \emph{exactly} the same as the classical Diaconis--Shahshahani Upper Bound Lemma \citep{PD}, save for replacing $G$ with $\G$.

\bigskip

\begin{lemma}\label{lbl}
(Lower Bound Lemma)
Suppose that $\nu\in M_p(\G)$ and $\rho$ the matrix element of a non-trivial one dimensional representation. Then
\beq
\|\nu^{\star k}-\pi \|\geq \frac12 |\nu(\rho)|^k.
\enq

\end{lemma}
\begin{proof}
Starting with (\ref{lb}), note that the argument from Proposition \ref{vanish} shows that $\rho$ has zero expectation under the random distribution. Note also that $\rho$ is unitary (Proposition 3.1.7 v., \citep{Timm}), thus norm one and thus  a suitable test function.

\bigskip

Note that for a one dimensional representation, by the Convolution Theorem \ref{DVDCT}:
$$|\nu^{\star k}(\rho)|=\left|\overline{\nu^{\star k}(\rho)}\right|=\left|\widehat{\nu^{\star k}}(\rho)\right|=|\widehat{\nu}(\rho)^k|=\left|\overline{\nu(\rho)}\right|^k=|\nu(\rho)|^{k}\,\,\,\bullet$$
\end{proof}

\section{Commutative Examples}
\subsection*{Simple Walk on the Circle}
Consider the random walk on\footnote{note that $\Z_n$ is the circle group of order $n$ and  \emph{not} some virtual object!} $F(\Z_n)$  driven by $\nu_n\in M_p(\Z_n)$:
\beq
\nu_n(\delta_s):=\begin{cases}
\half & \text{ if }s=\pm1,
\\[1ex] 0 & \text{ otherwise}.
\end{cases}
\enq

Note that $\Z_n$ is an abelian group, so all irreducible representations have degree 1. By the classical theory, each $\alpha=0,1,2,\dots,n-1$, gives a representation
$$\kappa_\alpha(\lambda)=\sum_{s\in \Z_n}\lambda e^{2\pi i\alpha s/n}\otimes\delta_s=\sum_{s\in \Z_n}\lambda\otimes e^{2\pi i\alpha s/n}\delta_s.$$
The family of groups $\Z_n$ has the random distribution $\pi_n:F(\Z_n)\raw\C$:
$$\pi_n(f)=\frac{1}{n}\sum_{t=0}^{n-1}\delta^t(f).$$
\subsubsection*{Upper Bounds\label{circ}}
\emph{For $k\geq n^2/40$, with $n$ odd,}
\beq
\|\nu_n^{\star k}-\pi_n \|\leq e^{-\pi^2 k/2n^2}
\enq

\begin{proof}
The Fourier transform of $\nu_n$ at $\kappa_\alpha$ is:
\begin{align*}
\widehat{\nu_n}(\alpha)(\lambda)&=\left(I_{\overline{\C}}\otimes\nu_n\right)\overline{\kappa}_{\alpha}(\lambda)
\\&=\left(I_{\overline{\C}}\otimes\nu_n\right)\sum_{s\in \Z_n}\lambda\otimes e^{-2\pi i\alpha s/n}\delta_s
\\&=\lambda\left(\frac12 e^{-2\pi i \alpha/n}+\frac12 e^{+2\pi i \alpha/n}\right)
\\&=\lambda\cos\left(2\pi \alpha/n\right).
\end{align*}
At this point the classical and quantum  theories coincide as can be seen by consulting Diaconis \citep{PD} (Section 3.C, Theorem 2) or Ceccherini-Silberstein \citep{cecc} (Theorem 2.2.1) $\bullet$
\end{proof}
\subsubsection*{Lower Bounds\label{circlb}}
\emph{For $n\geq 7$, and \text{any} $k$}
\beq
\|\nu_n^{\star k}-\pi_n\|\geq\half e^{-\pi^2 k/2n^2-\pi^4 k/2n^4}.
\enq

\begin{proof}
Use Lemma \ref{lbl} with $\kappa_{\left(\frac{n-1}{2}\right)}$. See Ceccherini-Silberstein \citep{cecc} (Theorem 2.2.1) for details $\bullet$
\end{proof}

\subsection*{Nearest Neighbour Walk on the $n$-Cube}
Consider the walk on $F(\Z_2^n)$, $n>1$,  driven by
\beq
\nu_n(\delta_s):=\begin{cases}
\frac{1}{n+1} & \text{ if }w(s)=0\text{ or }1,
\\[2ex] 0 & \text{ otherwise}
\end{cases}
\enq
where $w(s)$, the weight of $s=(s_1,s_2,\dots,s_n)$, is given by the sum in $\N$:
\beq
w(s)=\sum_{i=1}^ns_i
\enq

\subsubsection*{Upper Bounds}
\emph{For $k=(n+1)(\log n+c)/4$, $c>0$:}
\beq
\|\nu_n^{\star k}-\pi_n \|^2\leq \frac12 \left(e^{e^{-c}}-1\right)
\enq
\begin{proof}
A similar story to the above. See, for example, Diaconis \citep{PD} (Section 3.C, Theorem 3)\,\,$\bullet$
\end{proof}

\subsubsection*{Lower Bounds}
 Along with the upper bound extracted from the Diaconis--Fourier theory, tedious but elementary calculations bound the variation distance away from 0 for $k=(n+1)(\log n-c)/4$ for $n$ large and $c>0$. Define $\phi\in F(\Z_2^n)$ by $\phi(s)=n-2w(s)$. A set $A_\beta\subset \Z_2^n$ is  defined as the elements whose weight is sufficiently close to $n/2$ for some $\beta$:
\beqa
A_\beta:=\{s\in\Z_2^d:|\phi(s)|<\beta \sqrt{n}\}
\enqa
Using the norm-one function $2\ind_{A_\beta}-\ind_{\Z_2^d}\in F(\Z_2^n)$, and the same calculation as before:
$$\|\nu_n^{\star k}-\pi_n\|\geq |\nu_n^{\star k}(\ind_{A_\beta})-\pi_n(\ind_{A_{\beta}})|.$$
Careful calculations, referenced in the MSc, yield the desired lower bound. A more precise definition of $\beta$ in terms of $c$ makes this lower bound useful\footnote{if $\beta=e^{c/2}/2$ then the lower bound is $1-20/e^c$, which clearly tends to $1$ as $c$ increases}. Hence it follows that the random walk has a cut-off at time $t_n=n\log n/4$ --- for times sufficiently smaller than $t_n$ the variation distance is close to one, while for times sufficiently larger than $t_n$ the variation distance is close to zero.
\section{Cocommutative Examples: Random Walks on the Dual Group}
Let $\C G=F(\widehat{G})$ be the algebra of functions on the dual group of a finite group $G$. It is not immediately straightforward to recognise a probability on $\widehat{G}$. Elements of $M_p(\widehat{G})$ --- states on $\C G$ --- lie in $(\mathbb{C}G)'=F(G)$ and must be positive and have $u(1_{\widehat{G}})=1$.
 Let $\varphi=\sum_{t\in G}\alpha_t\delta^t\in F(\widehat{G})$ so that using the involution and multiplication in $F(\widehat{G})$,
\begin{align*}\varphi^*\star\varphi&=\left(\sum_{s\in G}\overline{\alpha_s}\delta^{s^{-1}}\right)\star\left(\sum_{t\in G}\alpha_t\delta^t\right)
\\ &=\sum_{s,t\in G}\overline{\alpha_s}\alpha_t \delta^{s^{-1}t}.
\end{align*}
Let $u=\sum_{t\in G}u(t)\delta_t\in F(\widehat{G})'$. For $u$ to be a positive functional:
$$u(\varphi^*\star \varphi)=\sum_{s,t\in G}\overline{\alpha_s}\alpha_tu(\delta^{s^{-1}t})\geq 0.$$
Such a function --- that has this property for all $\{\alpha_s:s\in G\}\subset \C$ --- is called \emph{positive definite}. It was noted earlier that for $u\in M_p(\widehat{G})$ it is required that $u(e)=1$. Also, $|u(s)|\leq u(e)$ (and so $|u(s)|\leq 1$) and $u(s^{-1})=\overline{u(s)}$ for all positive definite functions. See Bekka, de la Harpe and Valette (Proposition C.4.2., \citep{BdHV}) for a proof.

\bigskip

Furthermore, there is a correspondence between positive definite functions and unitary representations on $G$ together with a vector. In particular, for each positive definite function $u$ there exists a unitary representation $\rho:G\raw \operatorname{GL}(V)$ and a  vector $\xi\in V$ such that
\beq u(s)=\langle\rho(s)\xi,\xi\rangle,\label{posdef}\enq
and for each unitary representation and  vector (\ref{posdef}) defines a positive definite function on $G$.

\bigskip

For $u$ to be a state it is required that $u(e)=1$ and so $\langle \xi,\xi\rangle=1$; i.e. $\xi$ is a unit vector. Therefore probabilities on $\hat{G}$ can be chosen by selecting a given representation and unit vector.

\bigskip

Since $\Delta(\delta^s)=\delta^s\otimes\delta^s$, it follows that $\kappa_s(\lambda)=\lambda\otimes \delta^s$ defines a (co)representation (with $\tau=\kappa_e$), and thus all irreducible representations are of this form by counting. This makes the application of the upper bound lemma straightforward. Let $u\in M_p(\widehat{G})$ so that $\widehat{u}(\kappa_s)=\overline{u(s)}$
and so
$$\widehat{u}(\kappa_s)^*\widehat{u}(\kappa_s)=|u(s)|^2.$$
Therefore the upper bound lemma yields:
$$\|u^{\star k}-\pi\|^2\leq \frac{1}{4}\sum_{t\in G\backslash \{e\}}|u(t)|^{2k}.$$
\subsection*{A Walk on $\widehat{S_n}$}
Consider, for $n\geq 4$, the quantum group $\widehat{S_n}$ (given by $F(\widehat{S_n}):=\C S_n$) with a state $u\in M_p(\widehat{S_n})$ given by the permutation representation on $\C^n$ given by $\pi(\sigma)(e_i)=e_{\sigma(i)}$ together with the unit vector $\xi$ with components
$$\alpha_i=\sqrt{n^{n-i}\frac{n-1}{n^n-1}}.$$
For large $n$, this vector is approximately given by:
$$\xi\approx \left(1,\frac{1}{\sqrt{n}},\frac{1}{\sqrt{n^2}},\cdots \right)\approx (1,0,0,\cdots).$$
Following this through
$$u(\sigma)\approx\begin{cases}
                    1 & \mbox{if } \sigma(1)=1, \\
                    0 & \mbox{otherwise}.
                  \end{cases}$$
\subsubsection*{Upper Bounds}
\emph{For $k=\alpha n^n$ and $n\geq 4$}
$$\|u^{\star k}-\pi\|^2\leq \frac{e^2}{4} \frac{(n-1)^{n-\frac12}}{e^n}e^{-2(n-1)(\sqrt{n}-1)^2\alpha}\left[1+(n-1)e^{-(n-4)n^{n-1}\alpha}\right]$$
\begin{proof}
Note that
\begin{align*}
u(\sigma)&=\langle\xi,\pi(\sigma)\xi\rangle
\\ &=\left\langle \sum_{i=1}^n\alpha_ie_{\sigma(i)},\sum_{j=1}^n\alpha_je_j\right\rangle
\\ &=\sum_{i=1}^n\alpha_i\alpha_{\sigma(i)}
\\ &=\frac{n^{n+1}-n^n}{n^n-1}\sum_{i=1}^n\frac{1}{\sqrt{n^{i+\sigma(i)}}}.
\end{align*}
Therefore, using the Upper Bound Lemma,
\begin{align*}
\|u^{\star k}-\pi\|^2&\leq\frac{1}{4}\left(\frac{n^{n+1}-n^n}{n^n-1}\right)^{2k}\sum_{\sigma\in S_n\backslash \{e\}}\left[\sum_{i=1}^n\frac{1}{\sqrt{n^{i+\sigma(i)}}}\right]^{2k}.
\end{align*}
Define for $a_i=1/\sqrt{n^i}$
$$S(\sigma)=\sum_{i=1}^na_ia_{\sigma(i)}.$$
An \emph{inversion} is an ordered pair $(j,k)$ with $j,k\in \{1,\dots,n\}$ with $j<k$ and $\sigma(j)>\sigma(k)$. All non-identity permutations have at least one inversion.

 \newpage

Take any inversion $(j,k)$   and define a new permutation by

$$\tau_1(i):=(j\quad k)\sigma(i)=\begin{cases}\sigma(i)&\text{ if }i\neq j,k,\\ \sigma(j)&\text{ if } i=k,\\ \sigma(k)&\text{ if }i=j.\end{cases}$$

The calculation on P.79 of \citep{perm} shows that
$$S(\sigma)\leq S(\tau_1).$$
That is, multiplying by $(j\quad k)$, whenever $(j,k)$ is an inversion,  in this fashion, increases the number of fixed points, reduces the number of inversions and increases $S$. This can always be done until $\tau_r=e$.  If the maximising $\sigma\in S_{n}/\{e\}$ were not a transposition, then it would be the product of at least two transpositions. Take one of the transpositions $(j\quad k)$: it is certainly an inversion. By the referenced calculation, $\tau_i:=(j\quad k)\sigma$ has $S(\tau_i)\geq S(\sigma)$.  Therefore, no matter what the starting permutation $\sigma$, $\displaystyle\tau_{r-1}$ is a transposition and therefore to maximise $S$ on $S_n\backslash\{e\}$ one just maximises over transpositions.

\bigskip

Now again consider the decreasing sequence $\displaystyle a_i=\frac{1}{\sqrt{n^i}}$.
Define $$f(x)=\frac{1}{\sqrt{n^x}}-\frac{1}{\sqrt{n^{x+1}}}.$$
Note $f(x)$ is positive for $x\geq 1$. Furthermore
$$f'(x)=-\frac{1}{2}\ln n\cdot f(x),$$
and so $f(x)$ --- the one-step differences between the $a_i$ --- is decreasing and so the smallest one-step difference is between $a_{n-1}$ and $a_n$.

\bigskip

Let $(j\quad k)$ be a transposition with $j<k$. Note from the result about the one-step difference and $a_k\leq a_{j+1}$. Note
\begin{align*}
S((n-1\quad n))-S((j\quad k))&=a_j^2+a_k^2+2a_{n-1}a_n
\\&-a_{n-1}^2-a_n^2-2a_ja_k
\\&=(a_j-a_k)^2-(a_n-a_{n-1})^2\geq 0,
\end{align*}
so that $S$ is maximised at $(n-1\quad n)$.

\bigskip

Partition $S_n\bs \{e\}$ into $F_1$ and $F_1^C$ where $F_1$ is the set of permutations with $\sigma(1)=1$. On $F_1$,
\begin{align*}
S(\sigma)\leq S((n-1\quad n))&=\sum_{i=1}^{n-2}\frac{1}{n^i}+\frac{1}{\sqrt{n^{n-1+n}}}+\frac{1}{\sqrt{n^{n+n-1}}}
\\&=\frac{1}{n^n}\frac{n^n-n^2}{n-1}+\frac{2\sqrt{n}}{n^n}=:f_0.
\end{align*}
Now consider the maximum of $S$ on $F_1^C$. From \citep{perm}, it is known that strictly increasing the number of fixed points (by multiplying by suitably chosen transpositions), increases $S$.  Also, if written in the disjoint cycle notation, elements of $F_1^C$ must contain a cycle of the form $(1\quad i_2 \quad \dots \quad i_N)$. By multiplying by suitably chosen transpositions, any disjoint cycle not containing $1$ may be factored out whilst increasing $S$. Then write
$$(1\quad i_2 \quad \dots \quad i_N)=(1\quad i_N)(1\quad i_{N-1})\cdots (1\quad i_2),$$
so that the maximum of $S$ on $F_1^C$ occurs at an element of the form
$$\sigma=\prod_{k=N}^2(1\quad i_k).$$
All transpositions are inversions therefore can be removed --- all the time increasing $S$ --- until one gets a transposition of the form $(1\quad i)$. The maximum must occur at such a transposition. Note that
\begin{align*}
S((1\quad 2))-S((1\quad i))&=2a_1a_2+a_i^2-2a_1a_i-a_2^2
\\&=(a_1-a_i)^2-(a_1-a_2)^2\geq 0.
\end{align*}
Therefore $S(1\quad i)\leq S(1\quad 2)$ and for any $\sigma\in F_1^C$:
\begin{align*}
S(\sigma)&\leq S(1\quad 2)
\\ &=\frac{1}{\sqrt{n^{1+2}}}+\frac{1}{\sqrt{n^{2+1}}}+\sum_{i=3}^n\frac{1}{n^{i}}
\\&=\frac{2\sqrt{n}}{n^2}+\frac{1}{n^2}\frac{1}{n^n}\frac{n^n-n^2}{n-1}=:f_1.
\end{align*}
For $n\geq 4$, $f_1\leq 2f_0/\sqrt{n}$ as
\begin{align*}
\frac{2}{\sqrt{n}}f_0-f_1&=\frac{1}{n^2n^n(n-1)}[n^3\sqrt{n}((n^{n-3}-2)+n^{n-4}(n-\sqrt{n}))+n^2(4n-3)]\geq0.
\end{align*}

\bigskip

Therefore the Upper Bound Lemma yields:
\begin{align*}
\|u^{\star k}-\pi\|^2 &\leq \frac{1}{4}\left(\frac{n^{n+1}-n^n}{n^n-1}\right)^{2k}\left(\sum_{\sigma\in F_1}S(\sigma)^{2k}+\sum_{\sigma\in F_1^C}S(\sigma)^{2k}\right)
\\&\leq \frac{1}{4}\left(\frac{n^{n+1}-n^n}{n^n-1}\right)^{2k}\left[((n-1)!-1)f_0^{2k}+(n!-(n-1)!)\left(\frac{2}{\sqrt{n}}f_0\right)^{2k}\right]
\\&=\frac{1}{4}\left(\frac{n^{n+1}-n^n}{n^n-1}f_0\right)^{2k}\left[((n-1)!-1)+\left(\frac{4}{n}\right)^k(n!-(n-1)!)\right]
\\&\leq\frac{1}{4}\left(\frac{n^n-(n^2-2n\sqrt{n}+2\sqrt{n})}{n^n-1}\right)^{2k}\left[(n-1)!+\left(\frac{4}{n}\right)^k(n!-(n-1)!)\right]
\\&\leq\frac{1}{4}\left(1-\frac{(n-1)(\sqrt{n}-1)^2}{n^n-1}\right)^{2k}\left[\sqrt{n-1}(n-1)^{n-1}e^{2-n}\right.
\\ &\qquad+\left.\left(\frac{4}{n}\right)^k(n!-(n-1)!)\right],
\end{align*}
where the Stirling approximation upper bound was used.

\bigskip

Let $k=\alpha n^n$ so that
\begin{align*}
\left(\frac{4}{n}\right)^{k}&=\left(1-\frac{n-4}{n}\right)^{\alpha n^n}
\\&=\left[\left(1-\frac{n-4}{n}\right)^n\right]^{\alpha n^{n-1}}
\\&\leq e^{-(n-4)n^{n-1}\alpha},
\end{align*}
where $(1-x/n)^n\leq e^{-x}$ for $x<n$ was used. Also
\begin{align*}
\left(1-\frac{(n-1)(\sqrt{n}-1)^2}{n^n-1}\right)^{2k}&=\left(1-\frac{(n-1)(\sqrt{n}-1)^2}{n^n-1}\right)^{2\alpha n^n}
\\&\leq \left(1-\frac{(n-1)(\sqrt{n}-1)^2}{n^n-1}\right)^{2\alpha n^n-2}
\\&=\left[\left(1-\frac{(n-1)(\sqrt{n}-1)^2}{n^n-1}\right)^{n^n-1}\right]^{2\alpha}
\\&\leq e^{-2(n-1)(\sqrt{n}-1)^2\alpha}.
\end{align*}
Also note
$$n!-(n-1)!=(n-1)(n-1)!\leq\sqrt{n-1}(n-1)^{n}e^{2-n}.$$
using Stirling again. Putting these altogether
\begin{align*}
\|u^{\star k}-\pi\|^2&\leq \frac{e^2}{4} e^{-2(n-1)(\sqrt{n}-1)^2\alpha}\left[\frac{(n-1)^{n-1}\sqrt{n-1}}{e^n}+\right.
\\ &\qquad\left.\frac{\sqrt{n-1}(n-1)^{n-1}}{e^n}(n-1)e^{-(n-4)n^{n-1}\alpha}\right]\,\,\,\bullet
\end{align*}
\end{proof}
\subsubsection*{Lower Bounds}
\emph{For $k=\beta(n^n-1)$}
$$\|u^{\star k}-\pi\|^2\geq \frac14 \exp\left[-2\left(\frac{(n-1)^2(\sqrt{n}-1)^4}{n^n-1}+(n-1)(\sqrt{n}-1)^2\right)\beta\right].$$

\begin{proof}
First, a lemma:
\begin{lemma}\label{lem4}
  For $x>0$ and $n>2x$
  \beq
  \left(1-\frac{x}{n}\right)^n\geq e^{-x^2/n-x}.
  \enq
\end{lemma}
\begin{proof}
  Consider the function $v:(-1/2,1/2)\raw \R$ given by
  \begin{align*}
  v(t)&=t-t^2-\ln(1+t)
  \\ \Raw v'(t)&=1-2t-\frac{1}{1+t}
  \\&=-\frac{t(2t+1)}{1+t}.
  \end{align*}
  This is positive for $t<0$ and negative for $t>0$ and so $v(0)=0$ is the absolute max. Therefore
  $$\ln(1+t)\geq t-t^2.$$
  Note that $x\mapsto e^{nx}$ is increasing on the same domain so that:
  $$(1+t)^n\geq e^{nt-nt^2}.$$
  Let $t=y/n$. If $|t|=|y/n|<1/2\Raw n>2|y|$ then
  $$\left(1+\frac{y}{n}\right)^n\geq e^{y-y^2/n}.$$
  To complete the proof let $y=-x$ $\bullet$
\end{proof}
Now using the Lower Bound Lemma \ref{lbl} with the matrix element $\delta^{(n-1\quad n)}$
\begin{align*}
\|u^{\star k}-\pi\|&\geq \frac12 |u(\delta^{(n-1\quad n)})|^k
\\ \Raw \|u^{\star k}-\pi\|^2&\geq \frac14 \left(u(n-1\quad n)\right)^{2\beta(n^n-1)}
\\&=\frac14\left[\left(1-\frac{(n-1)(\sqrt{n}-1)^2}{n^n-1}\right)^{n^n-1}\right]^{2\beta}.
\end{align*}
An application of  Lemma \ref{lem4} completes the proof $\bullet$
\end{proof}
\section{Random Walks on the Kac--Paljutkin Quantum Group}
Let $A=F(\mathbb{KP})$ be the algebra of functions of the Kac--Paljutkin Quantum Group $\mathbb{KP}$ as described in \citep{franzgohm}.

\subsection*{Example: A Periodic Random Walk on $\mathbb{KP}$}
Let $\nu$ be the state $e^2$ (dual to $e_2$). It can be shown \citep{franzgohm} that
$$\nu^{\star k}=\left\{\begin{array}{ccc}e^2 & \text{ if } & k \text{ is odd,}\\ \eps & \text{ if }& k\text{ is even.}
\end{array}\right.$$
Therefore
$$\|\nu^{\star k}-\pi\|=\left\{\begin{array}{ccc}\|e^2-\pi\|& \text{ if } & k \text{ is odd,}\\ \|\eps-\pi\| & \text{ if }& k\text{ is even.}
\end{array}\right.$$
Consider first $k$ odd:
\begin{align*}
\|e^2-\pi\|&=\frac{1}{2}\left\|\F^{-1}(e^2-\pi)\right\|_{1}^{F(\mathbb{KP})}
\\&=\frac12 \|-e_1+7e_2-e_3-e_4-I_2\|_1^{F(\mathbb{KP})}
\\&=\frac12 \int_{\KP}\left(\left(-e_1+7e_2-e_3-e_4-I_2)^*(-e_1+7e_2-e_3-e_4-I_2\right)\right)^{1/2}
\\&=\frac12 \int_{\KP}(e_1+49e_2+e_3+e_4+I_2)^{1/2}
\\&=\frac12 \int_{\KP}(e_1+7e_2+e_3+e_4+I_2)
\\&=\frac12 \left[\frac{1}{8}(1+7+1+1+2+2)\right]=\frac{1}{16}(14)=\frac{7}{8}.
\end{align*}
A similar calculation shows that $\|\eps-\pi\|=7/8$. In fact this is, as a random walk, pretty much the same as the random walk on $\Z_8$ driven by $\nu=\delta^4$ which just alternates between $0$ and $4$.

\subsection*{The States of $F(\mathbb{KP})$}
A parameterisation of the states of $F(\mathbb{KP})$ --- with respect to the dual basis to the natural basis on $F(\mathbb{KP})$ --- is described in Franz and Gohm \citep{franzgohm}. However what is more interesting  are the entries of $\mathbb{CKP}$ that are dual to the matrix elements of the irreducible representations (see the next section) --- and Franz and Gohm write the states with respect to this basis also. Where $\rho^{\tau}$ is dual to the trivial representation matrix element $\rho_{\tau}=\mathds{1}_{\KP}$, the $\rho^i$ are dual to the $\rho_i$ and the $\rho^{ij}$ are dual to the $\rho_{ij}$, all states are of the form
\begin{align}
\nu&=(\mu_1+\mu_2+\mu_3+\mu_4+\mu_5)\rho^{\tau}+(\mu_1-\mu_2-\mu_3+\mu_4-z\mu_5)\rho^a\nonumber
\\&+(\mu_1-\mu_2-\mu_3+\mu_4+z\mu_5)\rho^b+(\mu_1+\mu_2+\mu_3+\mu_4-\mu_5)\rho^c\nonumber
\\&+(\mu_1+\mu_2-\mu_3-\mu_4)\rho^{11}+(\mu_1-\mu_2+\mu_3-\mu_4)\rho^{22}\label{bas}
\\&+\frac{x+y}{\sqrt{2}}\mu_5 \rho^{12}+\frac{x-y}{\sqrt{2}}\mu_5\rho^{22}. \nonumber
\end{align}
The $\mu_i$ and $x,y,z$ are parameters. The $x,y,z\in\R$ are parameters such that the state on the $M_2(\C)$-factor of $F(\mathbb{KP})$ is a state; i.e. $x^2+y^2+z^2\leq 1$. The $\mu_i$ are convex coefficients so that $\mu_i\in\R^+$ such that $\sum_i\mu_i=1$:
\begin{align*}
\nu&=\mu_1e^1+\mu_2e^2+\mu_3e^3+\mu_4e^4
\\&+\frac{\mu_5}{2}\left((1+z)E^{11}+(x-iy)E^{12}+(x+iy)E^{21}+(1-z)E^{22}\right),
\end{align*}
where the $e^i$ are dual to the $e_i$ and the $E^{ij}$ are dual to the $E_{ij}$.

\subsection*{Representation Theory of $\KP$}
On Page 147 of Izumi and Kosaki \citep{mxe} the matrix elements of the non-trivial irreducible unitary representations of $\mathbb{KP}$ are listed. There are three non-trivial one dimensional representations $\{\rho_a,\rho_b,\rho_c\}$:
\begin{align*}
\rho_a&=e_1+e_2+e_3+e_4-I_2,
\\ \rho_b&=e_1+e_2-e_3-e_4\oplus\left(\begin{array}{cc}-1 & 0 \\ 0 & 1\end{array}\right)
\\ \rho_c&=e_1+e_2-e_3-e_4\oplus\left(\begin{array}{cc}1 & 0 \\ 0 & -1\end{array}\right)
\end{align*}
and one two dimensional representation $\rho$ with elements:
\begin{align*}
\rho_{11}&=e_1-e_2-e_3+e_4
\\ \rho_{12}&=0\oplus 0\oplus 0\oplus 0\oplus\left(\begin{array}{cc}0 & (1+i)/\sqrt{2} \\ (1-i)/\sqrt{2} & 0\end{array}\right)
\\ \rho_{21}&=0\oplus 0\oplus 0\oplus 0\oplus\left(\begin{array}{cc}0 & (1-i)/\sqrt{2} \\ (1+i)/\sqrt{2} & 0\end{array}\right)
\\ \rho_{22}&=e_1-e_2+e_3-e_4.
\end{align*}

\subsubsection*{Symmetric Random Walks on $\KP$}
Consider a random walk on $\KP$ driven by $\nu\in M_p(\KP)$: invoking the Upper Bound Lemma yields:
$$\|\nu^{\star k}-\pi\|^2\leq\frac12 \operatorname{Tr}\left[\left(\widehat{\nu}(\rho)^*\right)^k\widehat{\nu}(\rho)^k\right]+\frac14 \sum_{i\in\{a,b,c\}}\operatorname{Tr}\left[\left(\widehat{\nu}(\rho_i)^*\right)^k\widehat{\nu}(\rho_i)^k\right]$$
Note that the trace of a linear map $T:\C\rightarrow \C$, $\lambda\mapsto a\lambda$ (as the $\widehat{\nu}(\rho_i))^*\widehat{\nu}(\rho_i)$ are) is just given by $T(1_{\C})=a$. Note also that for such maps $T^*(\lambda)=\bar{a}\lambda$ so that $\text{Tr}(T^*T)=|a|^2$. Finally such maps commute so that $(T^*)^kT^k=(T^*T)^{2k}$ so
$$\text{Tr}((T^*)^kT^k)=|a|^{2k}.$$
Note further that where $\rho_i$ is a one dimensional representation;
\begin{align*}
\widehat{\nu}(\rho_i)(1)&=(I_{\overline{C}}\otimes\nu)\overline{\rho_i}(1)=(I_{\overline{C}}\otimes\nu)(1\otimes\rho_i^*)
\\&=1\otimes \nu(\rho_i^*)\cong \overline{\nu(\rho_i)},
\end{align*}
as $\nu$ is a state and so:

$$\|\nu^{\star k}-\pi\|^2\leq\frac12 \operatorname{Tr}\left[\left(\widehat{\nu}(\rho)^*\right)^k\widehat{\nu}(\rho)^k\right]+\frac14 \sum_{i\in\{a,b,c\}}|\nu(\rho_i)|^{2k}$$
 As a result of (\ref{bas}), these one-dimensional terms are particularly easy to calculate, for example:
 \begin{align*}
 \overline{\nu(\rho_b)}& =\overline{(\mu_1-\mu_2-\mu_3+\mu_4+z\mu_5)}=\mu_1-\mu_2-\mu_3+\mu_4+z\mu_5,
 \end{align*}
 with similar results for $\rho_a$ and $\rho_c$.

 \bigskip

 The term for the two dimensional representation $\rho$ is potentially more troublesome. Elementary calculations show that:
 $$\widehat{\nu}(\rho)=\left(\begin{array}{cc}\overline{\nu(\rho_{11})} & \overline{\nu(\rho_{12})} \\ \overline{\nu(\rho_{21})} & \overline{\nu(\rho_{22})}\end{array}\right),$$

 Note that all of the entries are real. In the classical case, $A=F(G)$, with $G$ a finite group, the assumption of symmetry of the driving measure allows linear algebraic facts to be exploited.   `Up' in $M_p(G)$ this is equivalent to $\nu=\nu\circ S$ where $S$ is the antipode on $F(G)$.

 \bigskip

  In the case of $A=F(\mathbb{KP})$, symmetric states, $\nu=\nu\circ S$, have the property that $\nu(\rho_{12})=\nu(\rho_{21})$ and so $\widehat{\nu}(\rho)^*=\widehat{\nu}(\rho)$. In order to guarantee symmetry of $\nu$, it is necessary that $y=0$ or (the stronger condition) $\mu_5=0$. With this assumption,
  $$\left(\widehat{\nu}(\rho)^*\right)^k\widehat{\nu}(\rho)^k=\widehat{\nu}(\rho)^k\widehat{\nu}(\rho)^k=\widehat{\nu}(\rho)^{2k}.$$
  Now as $\widehat{\nu}(\rho)$ is symmetric it is diagonalisable with eigenvalues $\lambda_1$ and $\lambda_2$. Furthermore in this basis of eigenvectors (of $\overline{\C^2}$), $\widehat{\nu}(\rho)^{2k}$ is given by
  $$\left(\begin{array}{cc}\lambda_1^{2k} & 0 \\ 0 & \lambda_2^{2k}\end{array}\right),$$
  so that $\text{Tr}[\widehat{\nu}(\rho)^{2k}]=\lambda_1^{2k}+\lambda_2^{2k}$.
  The eigenvalues of $\widehat{\nu}(\rho)$ are given by
  $$\lambda_{\pm}=\mu_1-\mu_4\pm\sqrt{(\mu_2-\mu_3)^2+\frac{\mu_5^2x^2}{2}}.$$
  This gives us the upper bound for symmetric $\nu$:
  \begin{align}
  \|\nu^{\star k}-\pi\|^2&\leq\frac14(\mu_1-\mu_2-\mu_3+\mu_4-z\mu_5)^{2k}\nonumber
  \\&+\frac14 (\mu_1-\mu_2-\mu_3+\mu_4+z\mu_5)^{2k}+\frac14 (\mu_1+\mu_2+\mu_3+\mu_4-\mu_5)^{2k}\nonumber
  \\ & +\frac12\left(\mu_1-\mu_4+\sqrt{(\mu_2-\mu_3)^2+\frac{\mu_5^2x^2}{2}}\right)^{2k}\nonumber
  \\ & +\frac12\left(\mu_1-\mu_4-\sqrt{(\mu_2-\mu_3)^2+\frac{\mu_5^2 x^2}{2}}\right)^{2k} \label{kpup}
  \end{align}

  Note that $e^2$ is given by $\mu_2=1$ and all other parameters zero and $\eps$ is given by $\mu_1=1$ and all other parameters zero.... and so applying the formula to the random walk driven by $\nu=e^2$ yields,
  $$\frac78 =\|(e^2)^{\star k}-\pi\|\leq\frac{\sqrt{7}}{2}.$$
A good question at this point is to find conditions on the parameters that guarantees convergence to zero. Using a CAS it is not hard to come up with examples of symmetric random walks on $\mathbb{KP}$ that converge. Note that if $\nu\in M_p(\mathbb{KP})$ is such that $\mu_5=0$, there cannot be convergence to the Haar measure because $\widehat{\nu}(\rho_c)=1$ in that case.
  \subsubsection*{Examples}
  \begin{enumerate}
  \item   Consider the state:
  $$\nu=\frac{1}{4}(e^2+e^3+e^4)+\frac{1}{8}(E^{11}+E^{22})$$
  The convolution powers converge to the Haar measure. Using (\ref{kpup})
  $$\|\nu^{\star k}-\pi\|\leq \sqrt{\frac32 \left(\frac14 \right)^{2k}+\frac14\left(\frac{1}{2}\right)^{2k}}\leq \frac{\sqrt{7}}{2}\left(\frac12 \right)^{k}$$
Using the Lower Bound Lemma --- and (\ref{bas}) to calculate $\nu(\rho_a)=1/2$ --- yields

$$\frac{1}{2}\cdot \left(\frac12 \right)^{k}\leq \|\nu^{\star k}-\pi\|\leq \frac{\sqrt{7}}{2}\cdot\left(\frac12 \right)^{k}.$$

This is not particularly interesting... the walk is supported on a commutative subalgebra!

\item Consider the state:
  $$\nu=\frac{1}{4}(e^3+e^4)+\frac{1}{4}(E^{11}+E^{12}+E^{21}+E^{22}),.$$
  The convolution powers converge to the Haar measure. Using (\ref{kpup})
  $$\|\nu^{\star k}-\pi\|\leq \sqrt{\frac12\left(\frac{\sqrt{2}-1}{4}\right)^{2k}+\frac12\left(\frac{\sqrt{2}+1}{4}\right)^{2k}}\leq \left(\frac{\sqrt{2}+1}{4} \right)^{k}.$$
This time the Lower Bound Lemma is of no use as each of one dimensional matrix elements have expectation zero under $\nu$. However (\ref{bas}) means that with a bit of combinatorics, $\nu^{\star k}(\rho_{ij})$ can be calculated for $\rho_{ij}$ a matrix element of the two dimensional representation $\rho$. In particular, $\rho_{12}$ is  unitary and has zero expectation under the random distribution and so is a suitable `test' function. The following result is used:
\begin{lemma}
For $N\in\N$, $\alpha=2+\sqrt{3}$ and $\beta=2-\sqrt{3}$:\label{lemmasum}
\begin{align*}
\sum_{w=0}^N{\binom{N+w}{2w+1}}2^w&=\frac{\alpha^N-\beta^N}{2\sqrt{3}}.
\\ \sum_{w=0}^N{\binom{N+w}{2w}}2^w&=\frac{(5\alpha-1)\alpha^N+(\alpha+1)\beta^N}{6\alpha}.
\end{align*}
\begin{proof}
Let $P_1(N)$ be the first claim and $P_2(N)$ be the second. A quick calculation shows that $P_1(1)$ and $P_2(1)$ are true. Assume $P_1(k)$ and $P_2(k)$ and consider $P_1(k+1)$:
\begin{align*}
\sum_{w=0}^{k+1}\binom{(k+1)+w}{2w+1}2^w&=\sum_{w=0}^{k+1}\left[\binom{k+w}{2w}+\binom{k+w}{2w+1}\right]2^w
\\&\underset{P_2(k)\text{ and }P_1(k)}{=}\frac{(5\alpha-1)\alpha^k+(\alpha+1)\beta^k}{6\alpha}+\frac{\alpha^k-\beta^k}{2\sqrt{3}}
\\&=\frac{\sqrt{3}[(5\alpha-1)\alpha^k+(\alpha+1)\beta^k]}{6\sqrt{3}\alpha}+\frac{3\alpha(\alpha^k-\beta^k)}{6\sqrt{3}\alpha}
\\&=\frac{1}{6\sqrt{3}\alpha}[\alpha^k(5\sqrt{3}-\sqrt{3}+3\alpha)-\beta^k(3\alpha-\sqrt{3}\alpha-\sqrt{3})]
\\&=\frac{1}{6\sqrt{3}\alpha}[\alpha^k(21+12\sqrt{3})-3\beta^k]
\\&\underset{\alpha\beta=1}{=}\frac{1}{2\sqrt{3}}\left[\alpha^k\beta\underbrace{(7+4\sqrt{3})}_{\alpha^2}-\beta^{k+1}\right]
\\&=\frac{\alpha^{k+1}-\beta^{k+1}}{2\sqrt{3}},
\end{align*}
and so $P_1(k+1)$ is true.

\bigskip

Now consider $P_2(k+1)$:
\begin{align*}
\sum_{w=0}^{k+1} \binom{(k+1)+w}{2w}2^w&=\binom{k+1}{0}2^0+\sum_{w=1}^{k+1}\binom{(k+1)+w}{2w}2^w
\\&=\binom{k+0}{0}2^0+\sum_{w=1}^{k+1}\left[\binom{k+w}{2w-1}+\binom{k+w}{2w}\right]2^w
\\&\underset{u=w-1}{=}\sum_{u=0}^k\binom{(k+1)+u}{2u+1}2^{u+1}+\sum_{w=0}^{k+1}\binom{k+w}{2w}2^w
\end{align*}
Note that
$$\binom{(k+1)+(k+1)}{2(k+1)+1}=\binom{k+(k+1)}{2(k+1)}=0$$
so
\begin{align*}
\sum_{w=0}^{k+1}\binom{(k+1)+w}{2w}2^w&=2\underbrace{\cdot \sum_{u=1}^{k+1}\binom{(k+1)+u}{2u+1}2^u}_{P_1(k+1)}+\underbrace{\sum_{w=0}^k\binom{k+w}{2w}2^w}_{P_2(k)}
\\&=2\cdot \frac{\alpha^{k+1}-\beta^{k+1}}{2\sqrt{3}}+\frac{(5\alpha-1)\alpha^k+(\alpha+1)\beta^k}{6\alpha}
\\&=\frac{1}{6\sqrt{3}\alpha}[\alpha^k(6\alpha^2+5\sqrt{3}\alpha-\sqrt{3})+\beta^k(-6+\sqrt{3}\alpha+\sqrt{3})]
\\&=\frac{1}{6\sqrt{3}\alpha}[\alpha^{k}(57+33\sqrt{3})+\beta^k(3\sqrt{3}-3)]
\\&\frac{1}{6\alpha}[\alpha^k(33+19\sqrt{3})+\beta^k(3-\sqrt{3})]
\\&=\frac{(5\alpha-1)\alpha^{k+1}+(\alpha+1)\beta^{k+1}}{6\alpha}
\end{align*}
and so $P_2(k+1)$ is also true. By induction the result holds $\bullet$
\end{proof}
\end{lemma}
A simple inductive argument shows that any $\nu\in \C \G$
$$\nu^{\star k}(a)=\nu^{\otimes k}\left(\Delta^{(k-1)}(a)\right).$$
Therefore
$$\nu^{\star k}(\rho_{12})=\sum_{m_1,m_2,\dots,m_{k-1}}\nu(\rho_{1m_1})\nu(\rho_{m_1 m_2})\cdots\nu(\rho_{m_{k-1}2}).$$
Note the indices
$$1\raw m_1\raw m_2\raw m_3\raw\cdots \raw m_{k-1}\raw 2.$$
In the particular case of $\rho_{12}\in F(KP)$, the $m_i\in\{1,2\}$ and for the specific $\nu\in M_p(\KP)$ given above, $\nu(\rho_{22})=0$ means that $m_s=2\Raw m_{s+1}\neq 2$, $m_{k-1}\overset{!}{=}1$ and the $k+1$ indices can be considered as a path $1\raw 2$ of length $k$ in the graph shown below.

 \begin{figure}[ht]\cone\epsfig{figure=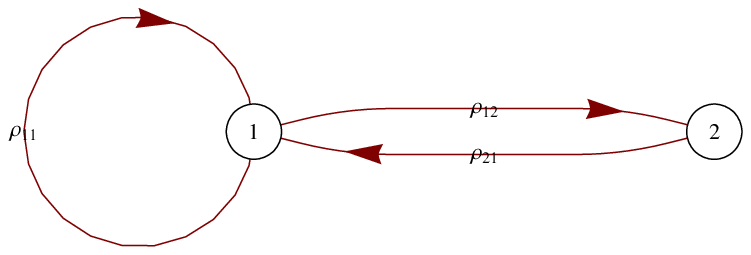}\ctwo\end{figure}

 Call $1\raw 2\raw 1$ a \emph{return}, $R$, $1\raw 1$ a \emph{loop}, $L$ and $1\raw 2$ a \emph{go}, $G$. Clearly
 $$1\raw m_1\raw \cdots\raw m_{k-2}\raw 1\raw 2$$
 consists of returns and loops followed by a go.  Note that
$$\nu(\rho_{11})=-\frac12,\qquad\nu(\rho_{12})=\nu(\rho_{21})=\frac{1}{2\sqrt{2}}\qquad\text{ and }\qquad\nu(\rho_{22})=0.$$
 Furthermore define
 $$\nu(R)=\nu(\rho_{12})\nu(\rho_{21}),\qquad\,\,\nu(L)=\nu(\rho_{11}),\qquad\,\,\,\text{ and }\nu(G)=\nu(\rho_{12})$$
 so that
 $$\nu^{\star k}(\rho_{ij})=\sum_{\underset{X_i\in\{L,R\}}{\text{paths }1\raw 1}}\nu(X_1)\nu(X_2)\cdots \nu(X_{\#L+\#R})\nu(G).$$
 Note that $|R|=2$ and so
 $$\#L+2\#R+\#G=k\Raw \#R=\frac{k-\#L-1}{2}.$$
 Suppose there are $\ell$ loops and so $\dsp \frac{k-\ell-1}{2}$ returns and the length $k$ path with $\ell$ loops looks like:
 $$\underbrace{X_1X_2\cdots X_M}_{\text{$\ell$ $L$s and $(k-\ell-1)/2$ $R$s}}G$$
 and so there are $\dsp \binom{\frac{k+\ell-1}{2}}{\ell}$ paths from $1\raw 1$ with $\ell$ loops.  For each of these paths with $\ell$ loops
 \begin{align*}
 \prod_{i=1}^{\frac{k+\ell-1}{2}}\nu(X_i)\nu(G)&=\nu(L)^\ell \nu(R)^{\frac{k-\ell-1}{2}}\nu(G)
 \\&=\left(-\frac{1}{2}\right)^{\ell}\left(\frac{1}{8}\right)^{\frac{k-\ell-1}{2}}\left(\frac{1}{2\sqrt{2}}\right)
 \\&=\frac{1}{2^{\frac32 k}}(-\sqrt{2})^\ell.
 \end{align*}
.

 Let $k$ be even so that $\ell$ is odd. Therefore, summing over the paths with $\ell$ loops from 1 to $k-1$:
 $$\|\nu^{\star k}-\pi\|\geq \frac{1}{2^{\frac32 k+1}}\left|\sum_{\underset{\text{odd}}{\ell=1}}^{k-1}\binom{\frac{k+\ell-1}{2}}{\ell}(-\sqrt{2})^{\ell}\right|.$$
 Reindexing using $\ell=2w+1$, and using $\binom{k/2+k/2}{2(k/2)+1}=0$ gives
 \begin{align*}
 \|\nu^{\star k}-\pi\|&\geq \frac{1}{2^{\frac{3}{2} k+1}}\left|\sum_{w=0}^{\frac{k}{2}-1}\binom{\frac{k}{2}+w }{2w+1}(-\sqrt{2})^{2w+1}\right|
 \\ &=\frac{1}{2^{\frac32 k+1}}\sqrt{2}\cdot\sum_{w=0}^{k/2}\binom{\frac{k}{2}+w}{2w+1}2^w
 \\ &\underset{\text{Lemma \ref{lemmasum}}}{=}\frac{1}{2^{\frac32 k+1}}\sqrt{2}\cdot\frac{\alpha^{k/2}-\beta^{k/2}}{2\sqrt{3}}
 \\&=\frac{1}{2^{\frac{3}{2}k+\frac{3}{2}}\sqrt{3}}(\alpha^{k/2}-\beta^{k/2})
 \\&=\frac{1}{2\sqrt{6}}\left((8\beta)^{-k/2}-(8\alpha)^{-k/2}\right)
\\&\approx \frac{1}{2\sqrt{6}}\left(8\beta\right)^{-k/2}
 \end{align*}
 for $k$ large.

\newpage

In the case of $k$ odd, the number of loops is even and the lower bound is given by
 $$\|\nu^{\star k}-\pi\|\geq \frac{1}{2^{\frac32 k+1}}\left|\sum_{\underset{\text{even}}{\ell=0}}^{k-1}\binom{\frac{k+\ell-1}{2}}{ \ell}(-\sqrt{2})^{\ell}\right|.$$
 Reindexing using $\ell=2u$ gives
 \begin{align*}
 \|\nu^{\star k}-\pi\|&\geq \frac{1}{2^{\frac{3}{2} k+1}}\left|\sum_{u=0}^{\frac{k-1}{2}}\binom{\frac{k-1}{2}+u}{2u}(-\sqrt{2})^{2u}\right|
 \\ &=\frac{1}{2^{\frac32 k+1}}\sum_{u=0}^{(k-1)/2}\binom{\frac{k-1}{2}+u}{2u}2^u
 \\ &\underset{\text{Lemma \ref{lemmasum}}}{=}\frac{1}{2^{\frac32 k+1}}\frac{(5\alpha-1)\alpha^{(k-1)/2}+(\alpha+1)\beta^{(k-1)/2}}{6\alpha}
\\&=\frac{1}{2\sqrt{6}\,8^{k/2}}\left[\left(\frac{5\alpha-1}{\alpha^{3/2}\,\sqrt{6}}\right)\alpha^{k/2}+\left(\frac{\alpha+1}{\alpha\,\sqrt{6}\,\sqrt{\beta}}\right)\beta^{k/2}\right] \end{align*}
Using the fact that $\sqrt{\alpha}=(\sqrt{2}+\sqrt{6})/2$ it can be shown that this is the same as the $k$ even case except for a sign change:
\begin{align*}
  \|\nu^{\star k}-\pi\|&\geq \frac{1}{2\sqrt{6}}\left((8\beta)^{-k/2}+(8\alpha)^{-k/2}\right)
 \\&\geq \frac{1}{2\sqrt{6}}\left((8\beta)^{-k/2}-(8\alpha)^{-k/2}\right)
\\&\approx \frac{1}{2\sqrt{6}}\left(8\beta\right)^{-k/2}
 \end{align*}
 for $k$ large.

 \bigskip

Therefore, for any $k$:

$$\frac{1}{2\sqrt{6}}\left((8\beta)^{-k/2}-(8\alpha)^{-k/2}\right)\leq \|\nu^{\star k}-\pi\|\leq \left(\frac{\sqrt{2}+1}{4}\right)^{k}.$$
\item Non-symmetric walks on $\KP$ can still be analysed but things are slightly messier as $\widehat{\nu}(\rho)$ is no longer equal to $\widehat{\nu}(\rho)^*$ necessarily. To see what needs to be done see the analysis for the representations $\kappa^{1,v}$ and $\kappa^{k-1,v}$ of $\KP_n$ below.

  \end{enumerate}

\section{Families of Walks on the Sekine Quantum Groups}
To use the quantum Diaconis--Shahshahani Upper Bound Lemma, the representation theory of the quantum group must be well understood.
The representation theory of the Sekine quantum groups changes according to the parity of the parameter $n$ and the below restricts to $n$ odd.
 \subsection*{Representation Theory for $n$ Odd}
 For $n$ odd there are $2n$ one dimensional representations and $\binom{n}{2}$ two dimensional representations. Consider the convolution algebra $(F(\mathbb{KP}_n),\star_A)$. Sekine gives $2n$ minimal one-dimensional central projections, $\binom{n}{2}$ minimal two-dimensional central projections  and matrix units in the two-dimensional subspaces. S\'{e}bastian Palcoux (private communication, March 2016) suggests a connection between projections and matrix units in the convolution algebra and the comultiplication in the algebra of functions. Palcoux's approach uses slightly different Fourier transforms and convolutions --- and the language of planar algebras (see \citep{Palc}) --- therefore the result could not be used directly. However there was enough to find the correct matrix elements (of the irreducible representations). See the Appendix to see the proof that these are indeed the matrix elements. As far as the author knows this is not in the existing literature.

 \bigskip

  Let $\ell\in\{0,1,\dots,n-1\}$. Then
\beq
\rho_\ell^{\pm}=\sum_{i,j\in\Z_n}\zeta_n^{i\ell}e_{(i,j)}\pm\sum_{m=1}^n E_{m,m+\ell},
\enq
are the $2n$ matrix elements of the one dimensional representations so that
$$\kappa_\ell^{\pm}(\lambda)=\lambda\otimes\rho_\ell^{\pm}\text{ and }\Delta(\rho_\ell^{\pm})=\rho_\ell^{\pm}\otimes \rho_\ell^{\pm}.$$
 Note that $\rho_0^+=\ind_{\KP_n}$ is the matrix element of the trivial representation.

Now let $u\in\{0,1,\dots,n-1\}$ and $v\in\{1,2,\dots,(n-1)/2\}$. Each pair gives a two dimensional representation $\kappa^{u,v}:\C^2\raw \C^2\otimes F(\KP_n)$ with matrix elements:
\begin{align*}
\rho_{11}^{u,v}&=\sum_{i,j\in\Z_n} \zeta_n^{iu+jv}e_{(i,j)}
\\ \rho_{12}^{u,v}&=\sum_{m=1}^n\zeta_n^{-mv}E_{m,m+u}
\\ \rho_{21}^{u,v}&=\sum_{m=1}^n\zeta_n^{mv}E_{m,m+u}
\\ \rho_{22}^{u,v}&=\sum_{i,j\in \Z_n}\zeta_n^{iu-jv}e_{(i,j)}.
\end{align*}
Consider the basis of  $\C\KP_n$ dual to $\{e_{(i,j)}\,:\,i,j\in\Z_n\}\cup \{E_{ij}\,:\,i,j=1,2,\dots,n\}$ given by
\begin{align*}
e^{(i,j)}e_{(r,s)}=\delta_{i,r}\delta_{j,s}&\text{\qquad and }&e^{(i,j)}E_{rs}=0,
\\ E^{ij}e_{(r,s)}=0&\text{\qquad and }&E^{ij}E_{rs}=\delta_{i,r}\delta_{j,s}.
\end{align*}
Let $\mu\in \C\KP_n$:
$$\mu=\sum_{i,j\in\Z_n}x_{(i,j)}e^{(i,j)}+\sum_{p,q=1}^na_{pq}E^{pq}.$$
Franz and Skalski \citep{idempotent} show that $\mu\in M_p(\KP_n)$ if and only if
\begin{itemize}
\item $x_{(i,j)}\geq0$ for all $i,j\in\Z_n$,
\item the matrix $A=(a_{pq})$ is positive,
\item $\operatorname{Tr}(\mu)=\sum_{i,j\in\Z_n}x_{(i,j)}+\sum_{p=1}^na_{pp}=1$.
\end{itemize}
\subsection*{A Random Walk on $\mathbb{KP}_n$ for $n$ Odd}
Consider the state $\dsp \nu=\frac{1}{4}(e^{(0,1)}+e^{(1,0)}+E^{11}+E^{12}+E^{21}+E^{22})\in M_p(\KP_n)$. The Quantum Diaconis--Shahshahani Upper Bound Lemma gives:
$$\|\nu^{\star k}-\pi\|^2\leq\frac14 \sum_{\alpha\in\operatorname{Irr}(\KP_n)\backslash\{\tau\}}d_\alpha\operatorname{Tr}\left[\left(\hat{\nu}(\alpha)^*\right)^k\hat{\nu}(\alpha)^k\right]$$
Unlike the commutative examples above, and like the example of random walks on $\KP$, the calculation must be split up as there is very different behaviour over different representations.

\bigskip

\subsubsection*{Upper Bounds}
\emph{For $\dsp k=\frac{n^2}{80}+\alpha n^2$ with $\alpha\geq 1$ and  $n\geq 7$}
$$\|\nu^{\star k}-\pi\|\leq 1.11 e^{-\alpha\pi^2}.$$
\begin{proof}

Define
\beq
f(k,n):=e^{-\pi^2(2k-1)/n^2}.
\enq

\subsubsection*{$\kappa_{0}^-$, $\kappa_{1}^-$ and $\kappa_{n-1}^-$}
As has been seen in the example of random walks on $\KP$, for a one-dimensional representation with matrix element $\rho_\alpha$, $d_\alpha\operatorname{Tr}\left[\left(\hat{\nu}(\alpha)^*\right)^k\hat{\nu}(\alpha)^k\right]=|\nu(\rho_\alpha)|^{2k}$. \new Therefore consider
$$|\nu(\rho_{1}^-)|^{2k}=\left|\frac{1+\zeta_n-1}{4}\right|^{2k}=\left|\frac{\zeta_n}{4}\right|^{2k}=\frac{1}{4^{2k}}.$$
Similarly, the contribution from $\kappa_{n-1}^-$ is the same while $\nu(\rho_{0}^-)=0$ so the contribution to the sum from these three representations is
\begin{align*}\frac{2}{4^{2k}}&=2\left(\frac{1}{4}\right)^{2k}\overbrace{\left(e^{\pi^2/n^2}\right)^{2k}e^{-\pi^2/n^2}f(k,n)}^{=1}
\\&\leq 2\left(\frac{e^{\pi^2/n^2}}{4}\right)^{2k}f(k,n),\end{align*}
where the fact that $e^{-\pi^2/n^2}\leq 1$ was used.
\subsubsection*{$\kappa_{1}^+$ and $\kappa_{n-1}^+$}
In both cases
$$|\nu(\rho_{i}^+)|^{2k}=\frac{|2+\zeta_n|^{2k}}{4^{2k}}\leq \left(\frac{3}{4}\right)^{2k}.$$
Therefore the contribution to the sum is given by:
\begin{align*}
2\left(\frac{3}{4}\right)^{2k}&=2\left(\frac{3}{4}\right)^{2k}\left(e^{\pi^2/n^2}\right)^{2k}e^{-\pi^2/n^2}f(k,N)
\\&\leq 2\left(\frac34 e^{\pi^2/n^2}\right)^{2k}f(k,n)
\end{align*}
\subsubsection*{$\kappa_{\ell}^{\pm}$ for $\ell=2,\dots,n-2$}
For each $\ell$,
$$\nu(\rho_{\ell}^\pm)=\frac{1+\zeta_n^\ell}{4},$$
and because there is a term from the $\rho_{\ell}^+$ as well as the $\rho_{\ell}^-$, the relevant sum is
$$\sum_{\underset{\ell=2,\dots,n-2}{\alpha=\kappa_{\ell}^{\pm}}}d_\alpha\operatorname{Tr}\left[\left(\hat{\nu}(\alpha)^*\right)^k\hat{\nu}(\alpha)^k\right]=\frac{2}{4^{2k}}\sum_{\ell=2}^{n-2}|1+\zeta_n^{\ell}|^{2k}$$
Note that
\begin{align*}
|1+\zeta_n^{\ell}|^2&=\left|1+\cos\left(\frac{2\pi \ell}{n}\right)+i\sin\left(\frac{2\pi \ell}{n}\right)\right|^2
\\&=1^2+\cos^2\left(\frac{2\pi \ell}{n}\right)+\sin^2\left(\frac{2\pi \ell}{n}\right)+2\cos\left(\frac{2\pi \ell}{n}\right)
\\&=2+2\cos\left(\frac{2\pi \ell}{n}\right)
\\&=4\cos^2\left(\frac{\pi \ell}{n}\right)
\end{align*}
and so the sum is
\begin{align*}
\frac{2}{4^{2k}}\sum_{\ell=2}^{n-2}4^k\cos^{2k}\left(\frac{\pi \ell}{n}\right)&\leq \frac{2}{4^k}\sum_{\ell=1}^{n-1}\cos^{2k}\left(\frac{\pi \ell}{n}\right)
\end{align*}
Sums such as these have been tackled in the authors MSc Thesis. The following appear in Lemma 3.4.1 of that work:

\begin{enumerate}
\item \emph{For $x\in[0,\pi/2]$},
\beq
\cos x\leq e^{-x^2/2}\label{lem2}
\enq
\item \emph{For any $x>0$}
\beq
\sum_{j=1}^\infty e^{-(j^2-1)x}\leq \sum_{j=0}^\infty e^{-3jx}\label{lem3}
\enq
\end{enumerate}
Now using the fact that (proved in the appendix of \citep{MSc})
\beq
\left|\cos\left(\ell\frac{\pi}{n}\right)\right|=\left|\cos\left(s\frac{\pi}{n}\right)\right|\,\,\text{ for any }s\in [\ell]_n,\label{lem321}
\enq
note that for $\dsp \ell=1,\dots,\frac{n-1}{2}$ that
$$\left|\cos\left(\frac{\pi}{n}\ell\right)\right|=\left|\cos\left(\frac{\pi}{n}(-\ell)\right)\right|\underset{(\ref{lem321})}{=}\left|\cos\left(\frac{\pi}{n}(n-\ell)\right)\right|.$$
Therefore the sum is
$$\frac{2}{4^k}\sum_{\ell=1}^{n-1}\cos^{2k}\left(\frac{\pi \ell}{n}\right)=\frac{1}{4^{k-1}}\sum_{\ell=1}^{\frac{n-1}{2}}\cos^{2k}\left(\frac{\pi \ell}{n}\right).$$
Applying (\ref{lem2}) yields
\begin{align*}
\frac{1}{4^{k-1}}\sum_{\ell=1}^{\frac{n-1}{2}}\cos^{2k}\left(\frac{\pi \ell}{n}\right)&\leq\sum_{\ell=1}^{(n-1)/2}e^{-\pi^2\ell^2k/n^2}
\\& \leq \frac{1}{4^{k-1}}\,e^{-\pi^2k/n^2}\sum_{\ell=1}^\infty e^{-\pi^2 (\ell^2-1)k/n^2},
\end{align*}
and so with (\ref{lem3})
\begin{align*}
\frac{1}{4^{k-1}}\,e^{-\pi^2k/n^2}\sum_{\ell=1}^\infty e^{-\pi^2 (\ell^2-1)k/n^2}&\leq\frac{1}{4^{k-1}}e^{-\pi^2k/n^2}\sum_{\ell=0}^\infty e^{-3\pi^2 \ell k/n^2}
\\& = \frac{1}{4^{k-1}}\,\frac{e^{-\pi^2 k/n^2}}{1-e^{-3\pi^2 k/n^2}}.
\end{align*}
Now if $k\geq n^2/40$, $\dsp \left(1-e^{-3\pi^2k/n^2}\right)>\frac{1}{2}$, and
it follows that
\begin{align*}
\sum_{\underset{\ell=2,\dots,n-2}{\alpha=\kappa_{\ell}^{\pm}}}d_\alpha\operatorname{Tr}\left[\left(\hat{\nu}(\alpha)^*\right)^k\hat{\nu}(\alpha)^k\right]&=\frac{2}{4^{k-1}}e^{-\pi^2 k/n^2}
\\&=\frac{8}{4^k}e^{\pi^2k/n^2}e^{-\pi^2/n^2}f(k,n)
\\&\leq 8\left(\frac{e^{\pi^2/n^2}}{4}\right)^kf(k,n)
\end{align*}

\subsubsection*{$\kappa^{0,v}$ for $v=1,\dots,\frac{n-1}{2}$}
Using the fact that $\hat{\nu}(\kappa)=(I_{\overline{V}}\otimes \nu)\overline{\kappa}$ and the definition of $\overline{\kappa}$ shows that
$$\hat{\nu}(\kappa^{0,v})=\left(\begin{array}{cc}\overline{\nu(\rho_{11}^{0,v})} & \overline{\nu(\rho_{12}^{0,v})}
\\ \overline{\nu(\rho_{21}^{0,v})} & \overline{\nu(\rho_{22}^{0,v})}\end{array}\right)=\frac{1}{4}\left(\begin{array}{cc}1+\zeta_n^{-v} & \zeta_n^v+\zeta_n^{2v} \\ \zeta_n^{-v}+\zeta_{n}^{-2v} & 1+\zeta_n^v\end{array}\right).$$
In a private communication, the following approach was suggested.
Write
$$A_v:=4\hat{\nu}(\kappa^{0,v})=(1+\zeta_n^v)\left(\begin{array}{cc}\zeta_n^{-v} & \zeta_n^v \\ \zeta_n^{-2v} & 1\end{array}\right)=(1+\zeta_n^v)\alpha^T\otimes\beta,$$
where $\alpha=\left(1\;\;\zeta_n^{-v}\right)$ and $\beta=\left(\zeta_n^{-v}\;\; \zeta_n^v\right)$. This implies that
$$A_v^*=\left(1+\zeta_n^{-v}\right)\beta^*\otimes\bar{\alpha}.$$
That both matrices have rank one reduces the computation of traces to scalar products of $\alpha,\bar{\alpha},\beta,\bar{\beta}$. In particular
\begin{align*}
\operatorname{Tr}\left(\left(A_v^*\right)^kA_v^k\right)&=
\left(1+\zeta_n^v\right)^k\left(1+\zeta_n^{-v}\right)^k \left(\bar\alpha \cdot \bar{\beta}^T\right)^{k-1}\left(\bar{\alpha}\cdot \alpha^T\right)\left(\beta\cdot\alpha^T\right)^{k-1}\left(\beta\cdot\bar{\beta}^T\right)
\\ &=4\left(1+\zeta_n^v\right)^{2k-1}\left(1+\zeta_n^{-v}\right)^{2k-1}.
\end{align*}
Note this includes a product of a complex number and its conjugate and so is
$$
4\left(|1+\zeta_n^v|^{2}\right)^{2k-1}.
$$
Note that
\begin{align*}
|1+\zeta_n^v|^2&=4\cos^2\left(\frac{\pi v}{n}\right)
\\\Raw \operatorname{Tr}\left(\left(A_v^*\right)^kA_v^k\right)&=4\left(|1+\zeta_n^v|^{2}\right)^{2k-1}
\\ &=4(4^{2k-1})\cos^{4k-2}\left(\frac{\pi v}{n}\right)=4^{2k}\cos^{4k-2}\left(\frac{\pi v}{n}\right)
\end{align*}
Noting that
$$\operatorname{Tr}\left(\left(\hat{\nu}(\kappa^{0,v})^*\right)^k\hat{\nu}(\kappa^{0,v})^k\right)=\frac{1}{4^{2k}}\operatorname{Tr}\left(\left(A_v^*\right)^kA_v^k\right),$$
the contribution to the Upper Bound Lemma sum is
$$\sum_{\alpha=\kappa^{0,v}}d_\alpha\operatorname{Tr}\left[\left(\hat{\nu}(\alpha)^*\right)^k\hat{\nu}(\alpha)^k\right]=2\sum_{v=1}^{\frac{n-1}{2}}\cos^{4k-2}\left(\frac{\pi v}{n}\right).$$
Applying (\ref{lem2}) yields
\begin{align*}
\sum_{\alpha=\kappa^{0,v}}d_\alpha\operatorname{Tr}\left[\left(\hat{\nu}(\alpha)^*\right)^k\hat{\nu}(\alpha)^k\right]&\leq2 \sum_{v=1}^{\frac{n-1}{2}}e^{-\pi^2v^2(2k-1)/n^2}
\\&\leq 2\,e^{-\pi^2(2k-1)/n^2}\sum_{v=1}^\infty e^{-\pi^2 (v^2-1)(2k-1)/n^2},
\end{align*}
and so with (\ref{lem3})
\begin{align*}
\sum_{\alpha=\kappa^{0,v}}d_\alpha\operatorname{Tr}\left[\left(\hat{\nu}(\alpha)^*\right)^k\hat{\nu}(\alpha)^k\right]&\leq 2\,e^{-\pi^2(2k-1)/n^2}\sum_{v=0}^\infty e^{-3\pi^2 v (2k-1)/n^2}
\\&= 2\,\frac{e^{-\pi^2(2k-1)/n^2}}{1-e^{-3\pi^2 (2k-1)/n^2}}.
\end{align*}
If $\dsp k\geq\frac{n^2}{80}+\frac{1}{2}$, then $\dsp \left(1-e^{-3\pi^2(2k-1)/n^2}\right)>\frac{1}{2}$, and
it follows that
\beqa
\sum_{\alpha=\kappa^{0,v}}d_\alpha\operatorname{Tr}\left[\left(\hat{\nu}(\alpha)^*\right)^k\hat{\nu}(\alpha)^k\right]\leq 4 e^{-\pi^2(2k-1)/n^2}=4f(k,n).
\enqa
\subsubsection*{$\kappa^{u,v}$ for $u=2...n-2$ and $v=1...\frac{n-1}{2}$}
In this case the Fourier transform at the representation is diagonal and so calculating the relevant trace is straightforward:
$$\sum_{\underset{u\neq0,1,n-1}{\alpha=\kappa^{u,v}}}d_\alpha\operatorname{Tr}\left[\left(\hat{\nu}(\alpha)^*\right)^k\hat{\nu}(\alpha)^k\right]=\frac{2}{4^{2k}}\sum_{\substack{u=2,\dots, n-2\\ v=1,\dots,\frac{n-1}{2}}}(|\zeta_n^u+\zeta_n^v|^{2k}+|\zeta_n^u+\zeta_n^{-v}|^{2k}).$$
Note
\begin{align*}
|\zeta_n^u+\zeta_n^v|=|(\zeta_n^u+\zeta_n^v)\zeta_n^{-v}|&=|\zeta_n^{u-v}+1|
\\&=4\cos^2\left(\frac{\pi(u-v)}{n}\right)\text{ and similarly}
\\ \left|\zeta_n^u+\zeta_n^{-v}\right|^2&=4\cos^2\left(\frac{\pi(u+v)}{n}\right).
\end{align*}

Therefore the relevant sum is
\begin{align*}
&\frac{2}{4^{2k}}\sum_{\substack{u=2,\dots,n-2\\ v=1,\dots,\frac{n-1}{2}}}\left(4^k\cos^{2k}\left(\frac{\pi(u-v)}{n}\right)+4^k\cos^{2k}\left(\frac{\pi(u+v)}{n}\right)\right)
\\&=\frac{2}{4^{k}}\sum_{\substack{u=2,\dots, n-2\\ v=1,\dots,\frac{n-1}{2}}}\left(\cos^{2k}\left(\frac{\pi(u-v)}{n}\right)+\cos^{2k}\left(\frac{\pi(u+v)}{n}\right)\right)
\\&=\frac{2}{4^{k}}\sum_{\underset{v=-\frac{n-1}{2},\dots,-1,1,\dots,\frac{n-1}{2}}{u=2,\dots, n-2}}\cos^{2k}\left(\frac{\pi(u+v)}{n}\right)
\\&=\frac{2}{4^{k}}\sum_{u=2}^{n-2}\left(\sum_{t\in\{u-\frac{n-1}{2},\dots,u+\frac{n-1}{2}\}\backslash\{u\}}\cos^{2k}\left(\frac{\pi t}{n}\right)\right)
\end{align*}
Note that for each $u$, because $\dsp u-\frac{n-1}{2}+n=u+\frac{n-1}{2}+1$, the following, $\mod n$, is a length $n-1$ sequence of consecutive terms:
$$u+1,u+2,\dots,u+\frac{n-1}{2},u-\frac{n-1}{2},u-\left(\frac{n-1}{2}-1\right),\dots,u-1.$$
Using (\ref{lem321}) the sum is therefore given by
$$\frac{2}{4^{k}}\sum_{u=2}^{n-2}\sum_{s=1}^{n-1}\cos^{2k}\left(\frac{\pi s}{n}\right)=\frac{2(n-3)}{4^k}\sum_{s=1}^{n-1}\cos^{2k}\left(\frac{\pi s}{n}\right).$$

Therefore, using similar techniques to those employed handling $\kappa_{\ell}^{\pm}$ ($\ell\neq0,1,n-1$), shows that if $\dsp k\geq n^2/40$:
\begin{align*}
\sum_{\underset{u\neq0,1,n-1}{\alpha=\kappa^{u,v}}}d_\alpha\operatorname{Tr}\left[\left(\hat{\nu}(\alpha)^*\right)^k\hat{\nu}(\alpha)^k\right]&
\leq \frac{2(n-3)}{4^{k-1}}e^{-\pi^2k/n^2}
\\&=\frac{2(n-3)}{4^{k-1}}e^{\pi^2k/n^2}e^{-\pi^2/n^2}f(k,n)
\\ &\leq 8(n-3)\left(\frac{e^{\pi^2/n^2}}{4}\right)^kf(k,n).
\end{align*}
\subsubsection*{$\kappa^{1,v}$ and $\kappa^{n-1,v}$ for $v=1,\dots,\frac{n-1}{2}$}
In this case
$$\hat{\nu}(\kappa^{1,v})=\frac{1}{4}\underbrace{\left(\begin{array}{cc}\zeta_n^{-1}+\zeta_n^{-v} & \zeta_n^v
\\ \zeta_n^{-v} & \zeta_n^{-1}+\zeta_n^v\end{array}\right)}_{=:B_v}.$$
The eigenvalues of $B_v$ are $\alpha_v:=\zeta_n^v+\zeta_n^{-v}+\zeta_n^{-1}$ and $\zeta_n^{-1}$ with eigenvectors $(1\,\,\,1)^T$ and $(\zeta_n^{-2v}\,\,\,-1)^T$. Therefore writing $B_v=PDP^{-1}$
\begin{align*}
B_v&=\frac{1}{1+\zeta_n^{2v}}\left(\begin{array}{cc}1 & \zeta_n^{2v}\\ 1 & -1\end{array}\right)\left(\begin{array}{cc}\alpha_v & 0 \\ 0 & \zeta_n^{-1}\end{array}\right)\left(\begin{array}{cc}1 & \zeta_n^{2v}\\ 1 & -1\end{array}\right)
\\ \Raw B_v^k&=\frac{1}{1+\zeta_n^{2v}}\left(\begin{array}{cc}1 & \zeta_n^{2v}\\ 1 & -1\end{array}\right)\left(\begin{array}{cc}\alpha_v^k & 0 \\ 0 & \zeta_n^{-k}\end{array}\right)\left(\begin{array}{cc}1 & \zeta_n^{2v}\\ 1 & -1\end{array}\right)
\\&=PD^kP^{-1}.
\end{align*}
From this it is a tedious but straightforward calculation to calculate
$$B_v^k(B_v^*)^k=PD^kP^{-1}(P^*)^{-1}\overline{D^k}P^*,$$
and find that its trace is given by
\begin{align*}
&\frac{1}{|1+\zeta_n^{2v}|^2}\left(4-2\zeta_n^k\alpha_v^k+\zeta_n^{k-2v}\alpha_v^k+\zeta_n^{k+2v}\alpha_v^k\right.
\\&\left.-2\zeta_n^{-k}\overline{\alpha_v}^k+\zeta_n^{-k-2v}\overline{\alpha_v}^k+\zeta_n^{-k+2v}\overline{\alpha_v}^k+4|\alpha_v|^{2k}\right).
\end{align*}
Note that
\begin{align*}&-2\zeta_n^k\alpha_v^k+\zeta_n^{k-2v}\alpha_v^k+\zeta_n^{k+2v}\alpha_v^k-2\zeta_n^{-k}\overline{\alpha_v}^k+\zeta_n^{-k-2v}\overline{\alpha_v}^k+\zeta_n^{-k+2v}\overline{\alpha_v}^k\\&=\zeta_n^{-k-2v}(\overline{\alpha_v}^k+\alpha_v^k\zeta_n^{2k})(\zeta_n^{2v}-1)^2,
\end{align*}
which can be seen by multiplying out. Secondly, similarly to above work,
\begin{align*}
|1+\zeta_n^{2v}|^2&=4\cos^{2}\left(\frac{2\pi v}{n}\right)
\\ \Raw\frac{1}{|1+\zeta_n^{2v}|^2}&=\frac14\sec^2\left(\frac{2\pi v}{n}\right).
\end{align*}
Recalling that $B_v=4\hat{\nu}(\kappa^{1,v})$ and that $\operatorname{Tr}(AB)=\operatorname{Tr}(BA)$
\begin{align*}
&\sum_{\underset{v=1,\dots,\frac{n-1}{2}}{\alpha=\kappa^{1,v}}}d_\alpha\operatorname{Tr}\left[\left(\hat{\nu}(\alpha)^*\right)^k\hat{\nu}(\alpha)^k\right]
\\&=\frac{2}{4^{2k+1}}\sum_{v=1}^{\frac{n-1}{2}}\sec^2\left(\frac{2\pi v}{n}\right)\left[4+4|\alpha_v|^{2k}+\zeta_n^{-k-2v}(\overline{\alpha_v}^k+\alpha_v^k\zeta_n^{2k})(\zeta_n^{2v}-1)^2\right].
\end{align*}
Note
\begin{align*}
\zeta_n^{-k-2v}(\overline{\alpha_v}^k+\alpha_v^k\zeta_n^{2k})(\zeta_n^{2v}-1)^2&=\frac{\overline{\alpha_v}^k+\alpha_v^k\zeta_n^{2k}}{\zeta_n^k}\cdot\left(\frac{\zeta_n^{2v}-1}{\zeta_n^v}\right)^2
\\&=\left(\overline{\alpha_v\zeta_n}^k+\left(\alpha_v\zeta_n\right)^k\right)\cdot\left(\zeta_n^v-\overline{\zeta_n}^v\right)^2
\\&=-8\Re\left((\alpha_v\zeta_n)^k\right)\sin^2\left(\frac{2\pi v}{n}\right)
\end{align*}
and so
\begin{align*}
&\sum_{\underset{v=1,\dots,\frac{n-1}{2}}{\alpha=\kappa^{1,v}}}d_\alpha\operatorname{Tr}\left[\left(\hat{\nu}(\alpha)^*\right)^k\hat{\nu}(\alpha)^k\right]
\\&=\frac{2}{4^{2k+1}}\sum_{v=1}^{\frac{n-1}{2}}\sec^2\left(\frac{2\pi v}{n}\right)\left[4+4|\alpha_v|^{2k}-8\sin^2\left(\frac{2\pi v}{n}\right)\Re\left((\alpha_v\zeta_n)^k\right)\right].
\end{align*}
A similar analysis shows that
$$\sum_{\underset{v=1,\dots,\frac{n-1}{2}}{\alpha=\kappa^{n-1,v}}}d_\alpha\operatorname{Tr}\left[\left(\hat{\nu}(\alpha)^*\right)^k\hat{\nu}(\alpha)^k\right],$$
gives the same trace and so the contribution from these two representations is
\beq\frac{1}{4^{2k-1}}\sum_{v=1}^{\frac{n-1}{2}}\sec^2\left(\frac{2\pi v}{n}\right)\left[1-2\sin^2\left(\frac{2\pi v}{n}\right)\Re\left((\alpha_v\zeta_n)^k\right)+|\alpha_v|^{2k}\right]\enq
Consider first
\begin{align*}
1-2\sin^2\left(\frac{2\pi
 v}{n}\right)\Re\left((\alpha_v\zeta)^k\right)+|\alpha_v|^{2k}&\leq 1+2\sin^2\left(\frac{2\pi
 v}{n}\right)\left|\Re\left((\alpha_v\zeta)^k\right)\right|+|\alpha_v|^{2k}
\\&\leq 1+2\sin^2\left(\frac{2\pi
 v}{n}\right)\left|(\alpha_v\zeta)^k\right|+|\alpha_v|^{2k}
\\&\leq 1+2\sin^2\left(\frac{2\pi
 v}{n}\right)\left|\alpha_v\right|^k+|\alpha_v|^{2k}
\\&\leq 1+2\sin^2\left(\frac{2\pi
 v}{n}\right)3^k+3^{2k}
\end{align*}
In terms of efficiency, while $n$ can be considered large, $k=\mathcal{O}(n^2)$ and for $k\approx \frac{n}{2}\mod n$,
$$-\Re\left((\alpha_v\zeta)^k\right)\approx +\Re\left((\alpha_v)^k\right).$$
The largest problem is that
$$\alpha_v=2\cos\left(\frac{2\pi v}{n}\right)+\zeta^{-1}$$
has a large real part for $n$ large and $v$ small but as $v\rightarrow \frac{n-1}{2}$
$$\alpha_v\approx -1,$$
rather than $\alpha_v\approx 3$ as is the case for $v$ small.

\bigskip

Therefore
\begin{align*} &\frac{1}{4^{2k-1}}\sum_{v=1}^{\frac{n-1}{2}}\sec^2\left(\frac{2\pi v}{n}\right)\left(1-2\sin^2\left(\frac{2\pi
 v}{n}\right)\Re\left((\alpha_v\zeta)^k\right)+|\alpha_v|^{2k}\right)\\ &\leq \frac{1}{4^{2k-1}}\sum_{v=1}^{\frac{n-1}{2}}\sec^2\left(\frac{2\pi v}{n}\right)\left(1+2\sin^2\left(\frac{2\pi
 v}{n}\right)3^k+3^{2k}\right)
\\&=\frac{1}{4^{2k-1}}\left((1+3^{2k})\sum_{v=1}^{\frac{n-1}{2}}\sec^2\left(\frac{2\pi v}{n}\right)+2\cdot 3^{k}\sum_{v=1}^{\frac{n-1}{2}}\tan^2\left(\frac{2\pi v}{n}\right)\right).
\end{align*}
Starting with $\exp(iw\pi/N)^{2N}=(-1)^w$ and then using Euler's Formula, the Binomial Theorem and taking imaginary parts, the following may be derived:
\beq
\sum_{v=1}^{\frac{n-1}{2}}\tan^2\left(\frac{2 \pi v}{n}\right)=\binom{n}{2}.
\enq
Using this, and $\sec^2A=1+\tan^2A$:
\begin{align*}
&\frac{1}{4^{2k-1}}\left((1+3^{2k})\left(\frac{n-1}{2}+\frac{n(n-1)}{2}\right)+2\cdot 3^{k}\frac{n(n-1)}{2}\right)
\\&=2\frac{n-1}{4^{2k}}\left[(n+1)3^{2k}+2n\cdot 3^k+n+1\right]
\\&\leq 2\frac{n-1}{4^{2k}}\left[(n+1)3^{2k}+2(n+1)3^k+n+1\right]
\\&\leq 2\frac{n^2-1}{4^{2k}}\left[3^{2k}+2\cdot 3^k+1\right]
\\&=2(n^2-1)\left(\frac{3}{4}\right)^{2k}\left(1+\frac{2}{3^k}+\frac{1}{3^{2k}}\right)
\\&\underset{k\geq 49}{\leq} 4(n^2-1)\left(\frac{3}{4}\right)^{2k}
\\&= 4(n^2-1)\left(\frac{3}{4}\right)^{2k}e^{2k\pi^2/n^2}e^{-\pi^2/n^2}f(k,n)
\\&\leq 4(n^2-1)\left(\frac{3}{4}e^{\pi^2/n^2}\right)^{2k}f(k,n).
\end{align*}
Putting all the bounds together, with $\gamma_n=e^{\pi^2/n^2}$:
\begin{align*}
\|\nu^{\star k}-\pi\|^2&\leq f(k,n)\left[\frac12 \cdot \left(\frac{\gamma_n}{4}\right)^{2k}+\frac12 \cdot \left(\frac34 \gamma_n\right)^{2k}+2\cdot\left(\frac{\sqrt{\gamma_n}}{2}\right)^{2k}\right.
\\&\left.+1+2(n-3)\cdot\left(\frac{\sqrt{\gamma_n}}{2}\right)^{2k}+ (n^2-1)\left(\frac34 \gamma_n\right)^{2k}\right]
\\&=f(k,n)\left[1+\frac12 \cdot \left(\frac{\gamma_n}{4}\right)^{2k}+\left( n^2-\frac12\right)\left(\frac{3}{4}\gamma_n\right)^{2k}\right.
\\&\qquad \left.+(2n-4)\left(\frac{\sqrt{\gamma_n}}{2}\right)^{2k}\right]
\end{align*}
\end{proof}
Take $n\geq 7$ and $k=\frac{n^2}{80}+\alpha n^2$ with $\alpha\geq 1$.

\bigskip

Using the fact that $n\geq 7$:
\begin{align*}
\frac12 \cdot \left(\frac{\gamma_n}{4}\right)^{2k}&=\frac12 \left(\frac{e^{\pi^2/40}}{4^{n^2/40}}\right)\left(\frac{e^{\pi^2}}{4^{n^2}}\right)^{2\alpha}
\\&\leq \frac{1}{8}(10^{-25})^{2\alpha}=\frac18 (10^{-50})^{\alpha}.
\end{align*}
Using the fact that $(x-1/2)(3/4)^x$ is decreasing for $x>4$ (at least),
\begin{align*}
\left(n^2-\frac{1}{2}\right)\left(\frac{3}{4}\gamma_n\right)^{2k}&=\left(n^2-\frac12 \right)\underbrace{\left(\frac34 \right)^{n^2/40}e^{\pi^2/40}}_{< 0.8997}\left(\frac34\right)^{2\alpha n^2}(e^{\pi^2})^{2\alpha}
\\&\leq \frac{9}{10}\underbrace{\left(n^2-\frac{1}{2}\right)\left(\frac{3}{4}\right)^{n^2}}_{<0.00004}\underbrace{e^{\pi^2}}_{< 19334}\left[\left(\frac{3}{4}\right)^{n^2}e^{\pi^2}\right]^{2\alpha-1}
\\&\leq \frac{9}{10}\cdot 10^{-4}\cdot 19334\left[\underbrace{\left(\frac{3}{4}\right)^{n^2}e^{\pi^2}}_{< 0.0146}\right]^{2\alpha-1}
\\&\leq \frac{7}{4}\left(\frac{3}{200}\right)^{2\alpha-1}
\\&=\frac{21}{800}\left(\frac{9}{40000}\right)^{\alpha}\leq \frac{7}{250}\left(\frac{1}{40000}\right)
\end{align*}
Using the fact that $(2x-4)(1/2)^{x^2}$ is decreasing for $x>3$ (at least),
\begin{align*}
(2n-4)\left(\frac{\sqrt{\gamma_n}}{2}\right)^{2k}&=(2n-4)\left(\frac{e^{\pi^2/80}}{2^{n^2/40}}\right)\left(\frac{e^{\pi^2}}{2^{2n^2}}\right)^{\alpha}
\\&=\underbrace{\frac{e^{\pi^2/80}}{2^{n^2/40}}}_{< 8273}\cdot \underbrace{(2n-4)\cdot \left(\frac{1}{2^{\alpha}}\right)^{n^2}}_{< 10^{-13}}\left(\underbrace{\frac{e^{\pi^2}}{2^{n^2}}}_{<10^{-10}}\right)^{\alpha}
\\&\leq 8300\cdot 10^{-13}\cdot 10^{-10\alpha}\leq 10^{-9-10\alpha}.
\end{align*}
Putting these all together
\begin{align*}
\|\nu^{\star k}-\pi\|&\leq \sqrt{f(k,n)\left[1+\frac18 10^{-50\alpha}+\frac{7}{250}\left(\frac{1}{40000}\right)^{\alpha}+10^{-9-10\alpha}\right]}
\\&=\sqrt{e^{-2k\pi^2/n^2}e^{\pi^2/n^2}\left[1+\frac18 10^{-50\alpha}+\frac{7}{250}\left(\frac{1}{4000}\right)^{\alpha}+10^{-9-10\alpha}\right]}
\\&\leq 1.11 e^{-k\pi^2/n^2}\qquad\bullet
\end{align*}
This  gets closer to $1e^{-k\pi^2/n^2}$ for $k$ a larger multiple of $n^2$ --- and $n$ itself large also.

\subsubsection*{Lower Bounds}
Using the Lower Bound Lemma, note that
$$\|\nu^{\star k}-\pi\|\geq \frac12 |\nu(\rho)|^k,$$
for any one dimensional representation matrix element $\rho$. In particular, the largest in magnitude occurs for $\rho_{1}^+$ and this yields:
$$\|\nu^{\star k}-\pi\|\geq \frac{1}{2}\left|\frac{2+\zeta_n}{4}\right|^k=\frac{1}{2^{2k+1}}\left(\sqrt{5+4\cos\left(\frac{2\pi }{n}\right)}\right)^k\approx \frac12\left(\frac{3}{4}\right)^k,$$
for $n$ large. Unfortunately this bound is wildly ineffective for $n$ large.

\chapter{Further Problems}
\section{Conditions for Convergence to Random}
In the classical case, a random walk on a finite group, $G$, starting at the identity and driven by a probability $\nu\in M_p(G)$ converges to the Haar measure on $G$ if and only if the probability is not concentrated on a subgroup (irreducibility) or on the coset of a normal subgroup (aperiodicity). Subgroups can be quantised using the quantisation functor. A subgroup $(H,m_H,e_H,{}^{-1,H})$ of a group $(G,m,e,{}^{-1})$ is a group together with a monomorphism/injection $\iota:H\raw G$ that satisfies:
$$\iota\circ m_H=m\circ(\iota\times \iota).$$
Applying the $\mathcal{Q}$ functor --- and noting that the dual of an injection is a surjection --- leads to the following definition (consistent with the standard definition (Definition 1.17, \citep{idempotent}). In particular, $\pi=\mathcal{Q}(\iota)$, $\Delta_{F(\mathbb{H})}=\mathcal{Q}(m_H)$ and $\Delta=\mathcal{Q}(m)$:

\bigskip

\begin{definition}
If $\G$ and $\mathbb{H}$ are  quantum groups and $\pi:F(\G) \raw F(\mathbb{H})$ is a surjective
unital $^*$-homomorphism such that
$$\Delta_{F(\mathbb{H})}\circ \pi=(\pi\otimes\pi)\circ \Delta,$$
then $\mathbb{H}$ is called a \emph{quantum subgroup} of $\G$.
\end{definition}

\bigskip

What are necessary and sufficient conditions on a probability on a quantum group that ensure its convolution powers converge to the random distribution, $\pi$? There are some results that explore the question at hand --- such as by Franz and Skalski (Proposition 2.4, \citep{Erg}) that shows if $\nu\in M_p(\G)$ is faithful then $\dsp \nu^{\star k}\raw \int_{\G}$. Of course, this is a very strong requirement in the classical case: equivalent to $\nu(\delta_s)>0$ for all $s\in G$.

\bigskip

In the classical case, if the convolution powers converge to an idempotent probability ($\phi\star \phi=\phi$),  then $\phi$ must be the Haar measure on a subgroup \citep{Hey77}. However, Pal \citep{Pal} shows that the idempotent $(e^1+e^4)/4+E^{11}/2\in M_p(\mathbb{KP})$ is not the Haar measure on any subgroup of $\mathbb{KP}$. Franz and Skalski suggest that this shows that the conditions for convergence of random walks on not-necessarily commutative quantum groups are ``clearly more complicated'' than the classical case. Franz and Skalski show, however, that idempotent probabilities on finite quantum groups \emph{are} Haar measures on sub-\emph{hyper}-subgroups. See \citep{idempotent}  for details. A 2013 paper by Wang \citep{Wang1} explores the concept of a quantum \emph{normal} subgroup and perhaps adapting these ideas to the realm of hypergroups might lead towards a satisfactory result.

\bigskip

On the one hand, not having this result is a deficiency of this work: on the other hand the quantum Diaconis--Shahshahani Upper Bound Lemma holds when the random walk does not converge in distribution to the random distribution but \emph{can} be used to show that specific $\nu\in M_p(\G)$ do converge to random.

\section{Spectral Analysis}
In the classical case, if the driving probability is symmetric, $\nu=\nu\circ S$, then the stochastic operator is a self-adjoint operator and therefore the stochastic operator can be diagonalised  and the powers easily calculated. This should also be possible for random walks on quantum groups.

\begin{theorem}
If a probability $\nu$ on a finite quantum group $\G$ is symmetric in the sense that $\nu=\nu\circ S$, then the stochastic operator $P_\nu\in L(F(\G))$ is self-adjoint.
\end{theorem}
\newpage
\begin{proof}
By the quantum finite Peter--Weyl Theorem, the set\newline $\{\rho_{ij}^{\alpha},\,\alpha\in\operatorname{Irr}(\G),\, i,j=1\dots d_\alpha \}$ is a basis for $F(\G)$. Calculate
\begin{align*}
P_\nu(\rho_{ij}^{\alpha})&=(\nu\otimes I_{F(\G)})\circ \Delta(\rho_{ij}^{\alpha})
\\&=\sum_{k}\nu(\rho_{ik}^{\alpha})\rho_{kj}^{\alpha}
\\&=\sum_{k}\nu(S(\rho_{ik}^{\alpha}))\rho_{kj}^{\alpha}
\\&=\sum_{k}\nu(\left(\rho_{ki}^{\alpha}\right)^*)\rho_{kj}^{\alpha}
\\&=\sum_{k}\overline{\nu(\rho_{ki}^{\alpha})}\rho_{kj}^{\alpha}=P_{\nu}^*(\rho_{ij}^{\alpha}),
\end{align*}
where Proposition \ref{matrixel}, the fact that $\nu$ is a state ($\nu(a^*)=\overline{\nu(a)}$) and $S(\rho_{ij}^{\alpha})=\left(\rho_{ji}^{\alpha}\right)^*$ were used $\bullet$
\end{proof}

\bigskip

Preceding the development of the Diaconis approach to random walks on finite quantum groups was this spectral analytic approach which culminates in the result that the convergence to random is largely controlled by the second largest (in magnitude) eigenvalue of the stochastic operator, $\lambda_{\star}$. Diaconis' approach is superior as the calculation of the second highest eigenvalue is far from straightforward for larger groups and furthermore the bound is not particularly sharp due to the information
loss in disregarding the rest of the spectrum of the stochastic operator. Regardless, it might be fruitful to try and prove a result similar to the classical result
$$\|\nu^{\star k}-\pi\|^2\leq \frac{|G|-1}{4}\lambda_{\star}^{2k},$$
to see are the upper bounds derived in this paper much of an improvement on the rough second-largest-eigenvalue-in-magnitude analysis.

\section{Lower Bounds and Cut-Off\label{LBCOP}}
As detailed in the introduction, a failure to generate effective lower bounds means that the sharpness of the upper bounds has not been tested. There are two ways in which lower bounds can show the efficiency of upper bounds.

\bigskip

For the simple walk on the circle, for $n\geq 7$ and $k\geq n^2/40$, the bounds are given by:
$$\frac12 e^{-\pi^4k/2n^4}\cdot e^{-\pi^2k/2n^2}\leq \|\nu_n^{\star k}-\pi\|\leq e^{-\pi^2k/2n^2},$$
and so for $k=\mathcal{O}(n^2)$, the total variation distance lies in an envelope between two relatively close bounds. With this lower bound being of the order of the upper bound, it is clear the upper bound is relatively sharp.

\bigskip

For the nearest neighbour walk on the $n$-cube, there are instead a pair of bounds --- one for prior and one for after $t_n=\frac{1}{4}(n+1)\ln n$. More specifically there are a pair of bounds for (separately $c<0$, $c>0$):
$$k_c=t_n+c\cdot \frac{n+1}{4}$$
For $n$ large, roughly, these bounds are of the form:
$$ 1\underset{c<0}{\approx}\|\nu_n^{k_c}-\pi\|\underset{c>0}{\approx} 0.$$
This shows something far more qualitatively interesting. This shows that for $k<t_n$, the random walk is far from random but then, quite abruptly, for $k>t_n$ the random walk converges to random. Apart from proving this \emph{cut-off phenomenon} holds, such a pair of bounds show the correct number of transitions required to force convergence. There are various formalisations of this cut-off phenomenon, but all, for a given random walk, require a family of groups and driving probabilities $(G_n,\nu_n)$ and a \emph{mixing time} $t_n$ such that for $k\ll t_n$ the walk is far from random and for $k\gg t_n$ the walk is close to random. In practise, one must find such a mixing time and also another function $g(n)$, $g(n)\ll t_n$, such that for $k=t_n-c\cdot g(n)$
$$\|\nu^k-\pi\|\approx 1$$
and for $k=t_n+c\cdot g(n)$
$$\|\nu^k-\pi\|\approx 0.$$
Intuition, based on the fact that for random walks of \emph{moderate} growth, such as the simple walk on the circle and a particular walk on the Heisenberg Group \citep{chenseventeen}, suggests that the family of random walks on the Sekine Quantum Groups does not experience the cut-off phenomenon. If this is true, future work on this family of random walks must concentrate on finding lower bounds of the order of the upper bound. The bounds in this work read
$$\frac12 \left(\frac34\right)^k\leq \|\nu^{\star k}-\pi\|\leq 1.11\left(e^{-\pi^2/n^2}\right)^k,$$
and for $k>32$ the lower bound is less than $1/20000$ and ineffective as the upper bound needs far more transitions (for larger $n$) --- $n^2$ --- to get below $1/200000$. If the random walk truly does not have the cut-off phenomenon, then possibly using a subspace $F(\mathbb{S}_n)\subset F(\mathbb{KP}_n)$, a more effective lower bound might be found and this will be an object of future study.

\bigskip

If the random walk does exhibit the cut-off phenomenon, then it might not actually be the case that the mixing time is $\mathcal{O}(n^2)$. Perhaps if the ideas of the papers of Diaconis and Saloff-Coste \citep{chensixteen,chenseventeen} --- which are useful for finding the correct order of the mixing time \emph{and} for finding lower bounds  --- can be adapted to the quantum group setting then this study of this random walk on $\mathbb{KP}_n$ can be brought to a satisfactory conclusion.

\bigskip

If it is the case that the random walk on $\mathbb{KP}_n$ does not experience cut-off and the barrier is a moderate growth condition, then perhaps by putting some $n$- dependence on the driving probability $\nu\in M_p(\mathbb{KP}_n)$ could give a walk that does indeed exhibit the cut-off phenomenon. Note that the driving probability on the walk on the simple group (and the referenced walk on the Heisenberg group) has no $n$-dependence. This leads to moderate growth and no cut-off phenomenon. In contrast, the driving probability for the nearest neighbour walk on the $n$-cube \emph{has} $n$-dependence, \emph{exponential} growth, and experiences the cut-off phenomenon.

\bigskip

A more thorough study of the walk on $\widehat{S_n}$ will also be an object of future study.

\section{Compact Matrix Quantum Groups\label{compact}}
Consider a compact group as defined in Section \ref{Haar}. Consider $\mathcal{L}^2(G)$, the space of all square-integrable (with respect to the Haar measure) functions on $G$. So much of Section \ref{Basics} carries through to the compact case and, in fact, the Peter--Weyl Theorem is true for compact groups.

\begin{theorem}(Peter--Weyl Theorem)\label{PWC}
  Let $\mathcal{I}=\operatorname{Irr}(G)$ be an index set for a family of pairwise-inequivalent \emph{finite dimensional} irreducible representations of a compact group $G$. Where $d_{\alpha}$ is the dimension of the vector space on which $\rho^\alpha$ acts ($\alpha\in\mathcal{I}$),
$$\left\{\rho_{ij}^\alpha\,|\,i,j=1\dots d_\alpha,\,\alpha\in \mathcal{I}\right\},$$
the set of matrix elements of $G$, is an orthogonal basis of $L^2(G)$.
\end{theorem}

 \begin{proof}
 The proof is similar to the finite case apart from the fact that the regular representation is no longer finite dimensional necessarily and a few other issues. See  \citep{Pont} for a full proof $\bullet$
 \end{proof}

This result means that  the (classical) Diaconis--Shahshahani Upper Bound \new Lemma still holds for compact groups --- although it seems to be all but intractable except for conjugate-invariant measures. Rosenthal \citep{RosDiac} was the first author to successfully use the Upper Bound Lemma in order to get rates of convergence for a random walk on a compact group ($SO(n)$). With great difficulty and a lot of non-trivial computation, Hough and Jiang \citep{hough} extended Rosenthal's work greatly. Varj\'{u} \citep{Varju} takes another approach to random walks on compact groups  which might also be quantisable.

\bigskip

In a seminal paper, Woronowicz introduced compact matrix quantum groups \citep{193}. Compact matrix quantum groups are well-behaved and not-necessarily finite dimensional quantum groups.
\begin{definition}
A \emph{compact matrix quantum group} is given by a $\mathrm{C}^*$-algebra $A=C(\mathbb{G})$ generated by the entries of a unitary matrix $u\in M_n(A)$ such that the following formulae define morphisms of $\mathrm{C}^*$-algebras:
$$\Delta(u_{ij})=\sum_k u_{ik}\otimes u_{kj},\qquad\eps(u_{ij})=\delta_{i,j},\qquad S(u_{ij})=u_{ji}^*.$$
\end{definition}
As the Peter--Weyl Theorem holds for compact matrix quantum groups (see  \citep{202} for a full proof) it should be possible --- barring technical problems such as a non-involutive antipode --- to prove the quantum Diaconis--Shahshahani Upper Bound Lemma for compact matrix quantum groups and analyse random walks on them.

\section{Convolution Factorisations of the Random Distribution\label{securb}}
In the classical case, Urban \citep{urban} studies the problem of factorising the random distribution as
\beq \int_G=\nu_m\star\nu_{m-1}\star\cdots \star \nu_1.\label{urban}\enq
As far as the author knows, this problem has not been studied in the quantum group setting. In the quantum setting, Urban's more precise question asks given a subspace $F(\mathbb{S})\subset F(\G)$ with the property that, where $S:F(\G)\raw F(\G)$ is the antipode, $S(F(\mathbb{S}))\subset F(\mathbb{S})$, does there exists a finite number of convolutions of symmetric probability measures $\nu_i\in M_p(\mathbb{G})$ supported on $F(\mathbb{S})$ such that (\ref{urban}) holds. As Urban uses Diaconis--Fourier theory to attack this problem, this is ripe for an attack in the quantum group case using the quantised machinery.

\section{Ces\`{a}ro Averages}
Another possible arena for future study, and perhaps a departure from the realm of random walks on quantum groups, would be a study of the Ces\`{a}ro means of a state $\nu\in M_p(\G)$
$$\nu_n=\frac{1}{n}\sum_{k=1}^n\nu^{\star k}.$$
It is not hard to see that $\nu_n\in M_p(\G)$ and it can be shown that $\nu_n$ always converges to an idempotent state $\nu_\infty$ (Theorem 7.1, \citep{franzgohm}). If $\nu$ is faithful, then the Ces\`{a}ro means converge to the random distribution. The reason that a study of these probabilities might be fruitful is that these Ces\`{a}ro means are well studied by the quantum group community. Indeed, in various contexts, the existence of the Haar measure is shown by taking a faithful state $\nu$ and showing that necessarily $\nu_n$ converges to an invariant state. It might also be interesting to see what Fourier Theory can say in light of calculations such as:
$$P_\nu (\nu_n)=\frac{n+1}{n}\cdot \nu_{n+1}-\frac{1}{n}\cdot \nu.$$
Of course, any problem in the theory of random walks on finite groups  --- if concerning global rather local behaviour --- is suitable for an attack in the quantum group setting. The classic work of Diaconis \citep{PD} contains a metaphorical ream of questions and problems that could be asked in the quantum group setting.

\begin{appendices}
\chapter{Matrix Elements of the Odd Sekine Quantum Groups}
Timmermann (Proposition 3.1.7 ii., iii., \citep{Timm}) shows that if a (finite) family $\{\rho_{ij}\}$ of elements of $F(\G)$ satisfy
$$\Delta(\rho_{ij})=\sum_k\rho_{ik}\otimes \rho_{kj}\quad\text{ and }\quad\eps(\rho_{i,j})=\delta_{i,j},$$
then the $\{\rho_{ij}\}$ are the matrix elements of a representation of $\G$.

\bigskip

Consider for $\ell\in\{0,\dots,n-1\}$
$$\rho_\ell^{\pm}=\sum_{i,j\in\Z_n}\zeta_n^{i\ell}e_{(i,j)}\pm\sum_{m=1}^n E_{m,m+\ell}.$$
Note that the counit of $F(\mathbb{KP}_n)$ is projection onto the $e_{(0,0)}$ factor and so $\eps(\rho_{\ell}^{\pm})=\zeta_n^{0\cdot \ell}=1$. Using the comultiplication given by (\ref{oneD}) and (\ref{Mfact}) consider
\begin{align*}
  \Delta(\rho_{\ell}^{\pm}) &= \sum_{i,j\in \Z_n}\zeta_n^{i\ell}\Delta(e_{(i,j)})\pm \sum_{m=1}^n\Delta(E_{m,m+\ell}) \\
   &= \sum_{i,j\in \Z_n}\zeta_{n}^{i\ell}\left(\sum_{u,v\in \Z_n}e_{(u,v)}\otimes e_{(i-u,j-v)}+\frac{1}{n}\sum_{u,v=1}^{n}\zeta_n^{i(u-v)}E_{u,v}\otimes E_{u+j,v+j}\right) \\
   & \pm \sum_{m=1}^n\left(\sum_{u,v\in\Z_n}\zeta_n^{-uL}e_{(-u,-v)}\otimes E_{m-v,m+\ell-v}+\sum_{u,v\in \Z_n}\zeta_n^{uL}E_{m-v,m+\ell-v}\otimes e_{(u,v)}\right) \\
\end{align*}
Write this as a sum of four terms:
\begin{align*}
  \Delta(\rho_{\ell}^{\pm}) &= \sum_{i,j,u,v\in \Z_n}\zeta_n^{i\ell}e_{(u,v)}\otimes e_{(i-u,j-v)}+\frac{1}{n}\sum_{\underset{i,j\in\Z_n}{u,v=1}}^n \zeta_n^{i(\ell+u-v)}E_{u,v}\otimes E_{u+j,v+j} \\
  &\pm \sum_{\underset{u,v\in\Z_n}{m=1}}^n\zeta_n^{-u\ell}e_{(-u,-v)}\otimes E_{m-v,m+\ell-v}\pm\sum_{\underset{u,v\in\Z_n}{m=1}}^{n}\zeta_n^{u\ell}E_{m-v,m+\ell-v}\otimes e_{(u,v)}
 \end{align*}

From the second term extract
\begin{align*}
\frac{1}{n}\sum_{i\in \Z_n}\zeta_n^{i(\ell+u-v)}E_{u,v}\otimes E_{u+j,v+j}&=\frac{1}{n}E_{u,v}\otimes E_{u+j,v+j}\sum_{i=0}^{n-1}\zeta_n^{i(\ell+u-v)}.
\end{align*}
If $\ell+u-v=0\Raw v=u+\ell$ then the summands are one and thus the second term equals $E_{u,v}\otimes E_{u+j,v+j}$. Otherwise
$$\sum_{i=0}^{n-1}\left(\zeta_{n}^{\ell+u-v}\right)^i=\frac{\zeta_n^{(\ell+u-v)n}-1}{\zeta_n^{\ell+u-v}-1}=0.$$
Therefore $v\overset{!}{=}u+\ell$ and so the second term is
$$\sum_{\underset{j\in\Z_n}{u=1}}^nE_{u,u+\ell}\otimes E_{u+j,u+j+\ell}.$$
Now consider:
\begin{align*}
\rho_{\ell}^{\pm}\otimes \rho_{\ell}^{\pm}&=\left(\sum_{s,t\in\Z_n}\zeta_n^{s\ell}e_{(s,t)}\pm\sum_{w=1}^n E_{w,w+\ell}\right)\otimes \left(\sum_{p,q\in\Z_n}\zeta_n^{p\ell}e_{(p,q)}\pm\sum_{r=1}^n E_{r,r+\ell}\right)
\\ &=\sum_{s,t,p,q\in \Z_n}\zeta_n^{(s+p)\ell}e_{(s,t)}\otimes e_{(p,q)}+\sum_{w,r=1}^nE_{w,w+\ell}\otimes E_{r,r+\ell}
\\&\pm \sum_{\underset{s,t\in\Z_n}{r=1}}^n\zeta_n^{s\ell}e_{(s,t)}\otimes E_{r,r+\ell}\pm\sum_{\underset{p,q\in\Z_n}{w=1}}^n\zeta_n^{p\ell}E_{w,w+\ell}\otimes e_{(p,q)}
\end{align*}
In the first term, reindex $s\raw u$, $t\raw v$, $s+p\raw i$, $t+q\raw j$. In the second reindex $w\raw u$ and $r\raw u+j$. In the third term reindex $s\raw -u$, $t\raw -v$ and $r\raw m-v$. In the fourth term reindex $p\raw u$, $q\raw v$ and $w\raw m-v$. Applying these shows that the $\rho_{\ell}^{\pm}$ are matrix elements of a one dimensional representation (and thus irreducible).

\bigskip

Now let $u\in\{0,1,\dots,n-1\}$ and $v\in\{1,2,\dots,(n-1)/2\}$ and consider elements:
\begin{align*}
\rho_{11}^{u,v}&=\sum_{i,j\in\Z_n} \zeta_n^{iu+jv}e_{(i,j)}
\\ \rho_{12}^{u,v}&=\sum_{m=1}^n\zeta_n^{-mv}E_{m,m+u}
\\ \rho_{21}^{u,v}&=\sum_{m=1}^n\zeta_n^{mv}E_{m,m+u}
\\ \rho_{22}^{u,v}&=\sum_{i,j\in \Z_n}\zeta_n^{iu-jv}e_{(i,j)}.
\end{align*}
Note that $\eps(\rho_{ij}^{u,v})=\delta_{i,j}$. If the $\rho_{ij}^{u,v}$ are to be matrix elements it must hold that
$$\Delta(\rho_{ij}^{u,v})=\rho_{i1}^{u,v}\otimes \rho_{1j}^{u,v}+\rho_{i2}^{u,v}\otimes \rho_{2j}^{u,v}.$$
Consider
\begin{align*}
\Delta(\rho_{11}^{u,v})&=\sum_{i,j\in\Z_n}\zeta_n^{iu+jv}\Delta(e_{(i,j)})
\\&=\sum_{i,j\in\Z_n}\zeta_n^{iu+jv}\left(\sum_{s,t\in\Z_n}e_{(s,t)}\otimes e_{(i-s,j-t)}+\frac{1}{n}\sum_{p,q=1}^n\zeta_n^{i(p-q)}E_{pq}\otimes E_{p+j,q+j}\right)
\\&=\sum_{i,j,s,t\in\Z_n}\zeta_n^{iu+jv}e_{(s,t)}\otimes e_{(i-s,j-t)}
\\&+\frac{1}{n}\sum_{\underset{i,j\in\Z_n}{p,q=1}}^n\zeta_n^{i(u+p-q)+jv}E_{p,q}\otimes E_{p+j,q+j}.
\end{align*}
For a fixed $p,q$, the second term is given by
$$\frac{1}{n}\sum_{i,j\in\Z_n}\left(\zeta_n^{u+p-q}\right)^i\zeta_n^{jv}E_{p,q}\otimes E_{p+j,q+j}=\frac{1}{n}\sum_{j=0}^{n-1}\zeta_n^{jv}E_{p,q}\otimes E_{p+j,q+j}\sum_{i=0}^{n-1}\left(\zeta_n^{u+p-q}\right)^{i}.$$
Similarly to the geometric series calculation above, the $\sum_i$ term is zero unless $q=p+u$ in which case it is equal to $n$. Therefore
$$\Delta(\rho_{11}^{u,v})=\sum_{\underset{j\in\Z_n}{p=1}}^{n}\zeta_n^{jv}E_{p,p+u}\otimes E_{p+j,p+j+u}.$$

Now consider
\begin{align*}
  \rho_{11}^{u,v}\otimes \rho_{11}^{u,v}+\rho_{12}^{u,v}\otimes \rho_{21}^{u,v} & =\left(\sum_{a,b\in \Z_n}\zeta_n^{au+bv}e_{(a,b)}\right)\otimes \left(\sum_{c,d\in\Z_n}\zeta_n^{cu+dv}e_{(c,d)}\right) \\
   & +\left(\sum_{r=1}^n\zeta_n^{-rv}E_{r,r+u}\right)\otimes\left(\sum_{w=1}^n\zeta_n^{wv}E_{w,w+u}\right) \\
   & =\sum_{a,b,c,d\in \Z_n}\zeta_n^{(a+c)u+(b+d)v}e_{(a,b)}\otimes e_{(c,d)} \\
   & +\sum_{r,w=1}^n\zeta_n^{(w-r)v}E_{r,r+u}\otimes E_{w,w+u}.
\end{align*}
Apply the reindexing $a\raw s$, $b\raw t$, $c\raw i-s$, $d\raw j-t$, $r\raw p$ and $w\raw p+j$ to see that this equals $\Delta(\rho_{11}^{u,v})$.

\bigskip

Consider now
\begin{align*}
\Delta(\rho_{12}^{u,v})&=\sum_{m=1}^n\zeta_n^{-mv}\Delta(E_{m,m+u})
\\&=\sum_{\underset{i,j\in\Z_n}{m=1}}^n\zeta_n^{-mv-iu}e_{(-i,-j)}\otimes E_{m-j,m+u-j}+\sum_{\underset{i,j\in\Z_n}{m=1}}^n\zeta_n^{iu-mv}E_{m-j,m+u-j}\otimes e_{(i,j)}
\\&\underset{\text{change sign of first term indices}}{=}\sum_{\underset{i,j\in\Z_n}{m=1}}^n\zeta_n^{iu-mv}e_{(i,j)}\otimes E_{m+j,m+u+j}
\\&+\sum_{\underset{i,j\in\Z_n}{m=1}}^n\zeta_n^{iu-mv}E_{m-j,m-j+u}\otimes e_{(i,j)}.
\end{align*}
Now consider
\begin{align*}
\rho_{11}^{u,v}\otimes \rho_{12}^{u,v}+\rho_{12}^{u,v}\otimes \rho_{22}^{u,v}&=\left(\sum_{i,j\in\Z_n}\zeta_n^{iu+jv}e_{(i,j)}\right)\otimes\left(\sum_{s=1}^n\zeta_n^{-sv}E_{s,s+u}\right)
\\&+\left(\sum_{t=1}^n\zeta_n^{-tv}E_{t,t+u}\right)\otimes \left(\sum_{i,j\in \Z_n}\zeta_n^{iu-jv}e_{(i,j)}\right)
\\&=\sum_{\underset{i,j\in\Z_n}{s=1}}^n\zeta_n^{iu+v(j-s)}e_{(i,j)}\otimes E_{s,s+u}
\\&+\sum_{\underset{i,j\in\Z_n}{t=1}}^n\zeta_n^{iu+v(-j-t)}E_{t,t+u}\otimes e_{(i,j)}\underset{s\raw m+j\text{ and }t\raw m-j}{=}\Delta(\rho_{12}^{u,v}).
\end{align*}

Similar calculations for $\rho_{21}^{u,v}$ and $\rho_{22}^{u,v}$ show that the $\rho_{ij}^{u,v}$ are the matrix elements of a two dimensional representation denoted by $\kappa^{u,v}$. It remains to show that the representations are irreducible.

\bigskip

\begin{definition}
The \emph{character} of a representation $\kappa$, with matrix elements $\{\rho_{ij}:1\leq i,j\leq d_\kappa\}$ is the element
$$\chi(\kappa)=\sum_{i=1}^{d_{\kappa}}\rho_{ii}.$$
\end{definition}
The irreducibility or otherwise of a representation can be tested using characters.

\bigskip

\begin{theorem}
A representation of a finite quantum group $\G$ is irreducible if and only if $\dsp\int_{\G} \chi(\kappa)^*\chi(\kappa)=1$.
\end{theorem}
\begin{proof}
Suppose that $\kappa$ is irreducible with matrix elements $\{\rho_{ij}\}$:
\begin{align*}
  \int_{\G}\chi(\kappa)^*\chi(\kappa) & =\int_{\G}\left(\sum_{i=1}^{d_\kappa}\rho_{ii}\right)^*\left(\sum_{j=1}^{d_\kappa}\rho_{jj}\right) \\
   & =\sum_{i,j=1}^{d_\kappa}\int_{\G}\rho_{ii}^*\rho_{jj}.
\end{align*}
Using Proposition \ref{ortho} this is easily seen to be one.

\bigskip

On the other hand if $\kappa$ is not irreducible then by Theorem \ref{directsum} it is the direct sum of $r>1$ irreducible representations $\kappa_i$ and Timmermann (Proposition 3.2.14, \citep{Timm}) shows that in that case
$$\chi(\kappa)=\sum_{i=1}^r\chi(\kappa_i),$$

and so
\begin{align*}
\int_{\G}\chi(\kappa)^*\chi(\kappa)&=\int_{\G}\left(\sum_{i=1}^r\chi(\kappa_i)\right)^*\left(\sum_{j=1}^r\chi(\kappa_j)\right)
\\&\underset{\text{Prop. \ref{ortho1}}}{=}\sum_{i=1}^r\int_{\G}\chi(\kappa_i)^*\chi(\kappa_i)=\sum_{i=1}^r1=r\,\,\,\bullet
\end{align*}
\end{proof}

\bigskip

Note that the character of $\kappa^{u,v}$ is given by
\begin{align*}
\chi(\kappa^{u,v})&=\sum_{i,j\in\Z_n}\zeta_n^{iu+jv}e_{(i,j)}+\sum_{i,j\in\Z_n}\zeta_n^{iu-jv}e_{(i,j)}
\\&=\sum_{i,j\in\Z_n}(\zeta_n^{iu+jv}+\zeta_n^{iu-jv})e_{(i,j)}
\\ \Raw \chi(\kappa^{u,v})^*&=\sum_{i,j\in\Z_n}(\zeta_n^{-iu-jv}+\zeta_n^{-iu+jv})e_{(i,j)}
\\ \Raw \chi(\kappa^{u,v})^*\chi(\kappa^{u,v})&=\sum_{i,j\in \Z_n}|\zeta_n^{iu+jv}+\zeta_n^{iu-jv}|^2e_{(i,j)}
\\ \Raw \int_{\mathbb{KP}_n}\chi(\kappa^{u,v})^*\chi(\kappa^{u,v}) &=\frac{1}{2n^2}\sum_{i,j\in \Z_n}|\zeta_n^{iu}(\zeta_n^{jv}+\zeta_n^{jv})|^2
\\&=\frac{1}{2n^2}\sum_{i,j\in\Z_n}|\zeta_n^{jv}+\zeta_n^{-jv}|^2
\\&=\frac{1}{2n^2}\cdot n\sum_{j=0}^{n-1}\left|2\cos\left(\frac{2\pi j v}{n}\right)\right|^2
\\&=\frac{2}{n}\sum_{j=0}^{n-1}\cos^2\left(\frac{2\pi j v}{n}\right)
\\&=\frac{2}{n}\sum_{j=0}^{n-1}\left(\frac12 +\frac12 \cos\left(\frac{4\pi j v}{n}\right)\right)
\\&=1+\frac{1}{n}\sum_{j=0}^{n-1}\cos\left(\frac{4\pi j v}{n}\right).
\end{align*}
Note that
\begin{align*}
\sum_{j=0}^{n-1}\cos\left(\frac{4\pi j v}{n}\right)&=\Re\left(\sum_{j=0}^{n-1}\left(e^{4\pi vi/n}\right)^j\right)=\Re\left(\frac{e^{4\pi vin/n}-1}{e^{4\pi v i/n}-1}\right).
\end{align*}
Note that $e^{4\pi vin/n}=e^{2\pi i(2v)}=1$. Also $4\pi v/n$ cannot be a multiple of $2\pi$ as $v\in\{1,\dots,(n-1)/2\}$. Therefore $e^{4\pi v i/n}-1\neq 0$ and the sum is zero. This yields $\dsp \int_{\mathbb{KP}_n}\chi(\kappa^{u,v})^*\chi(\kappa^{u,v})=1$ and therefore the $\kappa^{u,v}$ are irreducible representations.

\end{appendices}

\setcitestyle{numbers}
\bibliographystyle{plainnat}
\bibliography{phdbib}

\end{document}